\documentclass[11pt,a4paper,fleqn]{article}

\usepackage{amsthm}
\usepackage{amssymb}
\usepackage{amstext}
\usepackage{amsmath}
\usepackage{url}
\usepackage[colorlinks]{hyperref}
\hypersetup{
linkcolor=blue,
citecolor=blue,
}
\usepackage[margin = 40pt,font=small,labelfont=bf]{caption}
\usepackage{arydshln}
\usepackage{pdfsync}
\usepackage[latin1]{inputenc}
\usepackage[T1]{fontenc}
\usepackage{color}
\usepackage{algorithm}
\usepackage{algorithmic}
\usepackage{enumerate}
\usepackage{bbm}
\usepackage{ae,aecompl}
\usepackage{pdfpages}
\usepackage{epstopdf}
\usepackage{graphicx}
\usepackage{epsfig}
\usepackage{subfigure}
\usepackage{cite}
\usepackage[textwidth=4cm, textsize=footnotesize]{todonotes}
\usepackage{xcolor}
\usepackage{ulem}

\numberwithin{equation}{section}
\theoremstyle{plain}
\newtheorem{thm}{Theorem}[section]

\newtheorem{prop}{Proposition}[section]
\newtheorem{cor}{Corollary}[section]
\newtheorem{lem}{Lemma}[section]
\newtheorem{deff}{Definition}[section]

\font\calcal=cmsy10 scaled\magstep1
\def\build#1_#2^#3{\mathrel{\mathop{\kern 0pt#1}\limits_{#2}^{#3}}}
\def\liml{\build{\longrightarrow}_{}^{{\mbox{\calcal L}}}}

\renewcommand{\leq}{\leqslant}
\renewcommand{\geq}{\geqslant}

\def\videbox{\mathbin{\vbox{\hrule\hbox{\vrule height1.4ex \kern.6em\vrule height1.4ex}\hrule}}}

\newcommand{\bv}{\boldsymbol{v}}

\newcommand{\He}{H_{\varepsilon}}
\newcommand{\he}{h_{\varepsilon}}

\newcommand{\Id}{\mbox{$I_{J}$}}
\newcommand{\proj}{\mbox{$P_{J}$}}

\newcommand{\E}{{\mathbb E}}
\newcommand{\R}{{\mathbb R}}

\newcommand{\cF}{{\mathcal F}}

\newcommand{\wh}{\widehat}
\newcommand{\wt}{\widetilde}

\newcommand{\XX}{\ensuremath{\mathcal X}}
\newcommand{\YY}{\ensuremath{\mathcal Y}}

\newcommand{\MM}{\ensuremath{\mathcal M}}

\newcommand{\hVn}{\wh{V}_{n}}

\newcommand{\hVkm}{\wh{V}_{k-1}}

\newcommand{\hUn}{\widetilde{U}_{n}}
\newcommand{\hUnp}{\widetilde{U}_{n+1}}

\newcommand{\hWn}{\wh{W}_n}

\newcommand{\tcr}[1]{\textcolor{black}{#1}}
\newcommand{\tcb}[1]{\textcolor{black}{#1}}
\newcommand\rst{\bgroup\markoverwith{\textcolor{black}{\rule[0.5ex]{2pt}{0.4pt}}}\ULon}

\newcommand{\CB}[1]{{\color{black} #1}}

\makeatletter
\newcommand{\tpmod}[1]{{\@displayfalse\pmod{#1}}}
\makeatother

\def\diag{\mathop{\rm diag}\nolimits}%
\def\Sp{\mathop{\rm Sp}\nolimits}%

\newcommand{\thefont}[2]{\fontsize{#1}{#2}\fontshape{n}\selectfont}
\newcommand{\1}{\rlap{\thefont{10pt}{12pt}1}\kern.16em\rlap{\thefont{11pt}{13.2pt}1}\kern.4em}

\def\Tr{\mathrm{Tr}}

\title{A stochastic Gauss-Newton algorithm for regularized semi-discrete optimal transport}

\author{Bernard Bercu$^{1}$, J\'{e}r\'{e}mie Bigot$^{1}$, S\'{e}bastien Gadat$^{2}$, Emilia Siviero$^{3}$   \\
\\  $^{1}$Institut de Math\'ematiques de Bordeaux et CNRS  (UMR 5251), Universit\'e de Bordeaux \vspace{0.1cm}  \\ $^{2}$Toulouse School of Economics, Universit\'{e} Toulouse 1 Capitole \\ $^{3}$ LTCI, T\'{e}l\'{e}com Paris, Institut Polytechnique de Paris }

\date{\today}

\begin{document}
\sloppy

\maketitle

\thispagestyle{empty}

\vspace{-5ex}

\begin{abstract}
We introduce a new second order stochastic algorithm to estimate the entropically regularized optimal transport cost  between two probability measures. The source measure can be arbitrary chosen, either absolutely continuous or discrete, while the target measure is assumed to be discrete. To solve  the semi-dual formulation of such a regularized and semi-discrete optimal transportation problem, we propose to consider a stochastic Gauss-Newton algorithm that uses a sequence of data sampled from the source measure.   This algorithm is shown to be adaptive to the geometry of the underlying convex optimization problem with no important hyperparameter to be accurately tuned. We establish the almost sure convergence and the asymptotic normality  of  various estimators of interest that are constructed from this  stochastic Gauss-Newton algorithm. We also analyze their non-asymptotic rates of convergence for the expected quadratic risk in the absence of strong convexity  of the underlying objective function. The results of numerical experiments from simulated  data are also reported   to illustrate the finite sample properties of this  Gauss-Newton algorithm for stochastic regularized optimal transport, and to show its advantages over the use of the  stochastic gradient descent, stochastic Newton and ADAM algorithms.
\end{abstract}




\section{Introduction}


\subsection{Computational optimal transport for data science}

The use of optimal transport (OT) and Wasserstein distances for data science has recently gained an increasing interest in various research fields such as machine learning \cite{rolet2016fast,NIPS2016_6566,Frogner:2015,pmlr-v97-gordaliza19a,Flamary2018,NIPS2017_6792,NIPS2015_5680}, statistics \cite{Stochastic_Bigot_Bercu,bigot:hal-01790015,bigot2017geodesic,cazelles:hal-01581699,Pana15,Pana17,sommerfeld2016inference,KlattTM20,RIGOLLET20181228} and image processing or computer vision \cite{gramfort2015fast,rabin2015convex,benamou2015iterative,ferradans2014regularized,2015-bonneel-siims,Solomon:2015}. Solving a problem of OT between two probability measures $\mu$ and $\nu$ is known to be computationally challenging, and entropic regularization \cite{cuturi,cutpey2018} has emerged as an efficient tool to approximate and smooth the variational Wasserstein problems arising in computational optimal transport for data science. A detailed presentation of the recent research field of computational optimal transport is given in \cite{bookOT}, while recent reviews on the application of optimal transport to statistics can be found in \cite{bigotReview,Pana18}. 

Recently, approaches \cite{Stochastic_Bigot_Bercu,NIPS2016_6566}  based on first order stochastic algorithms have gained  popularity to solve (possibly regularized) OT problems using data sampled from $\mu$. These approaches are based on the semi-dual formulation  \cite{cutpey2018} of  regularized OT problems that can be rewritten as a {\it non-strongly convex} stochastic optimization problem. In this paper, for the purpose of obtaining stochastic algorithms for OT with faster convergence in practice, we introduce a second order stochastic algorithm to solve regularized semi-discrete OT between an arbitrary probability measure $\mu$, typically absolutely continuous,  and a {\it known} discrete measure $\nu$ with finite support of size $J$. More precisely, we focus on the estimation of an entropically regularized optimal transport cost $W_{\varepsilon}(\mu,\nu)$ between such measures (where $\varepsilon > 0$ is an entropic regularization parameter) using a class of stochastic quasi-Newton algorithms that we refer to as Gauss-Newton algorithms and which use the knowledge of a sequence $(X_n)$ of independent random vectors sampled from $\mu$.

Applications of semi-discrete optimal transport can be found in computational geometry and computer graphics \cite{Merigot11,merigot18b}, as well as in the  problem of optimal allocation of resources from online observations \cite{Stochastic_Bigot_Bercu}. For an introduction to semi-discrete optimal transport problems and related references, we also refer to  \cite[Chapter 5]{bookOT}.  In a deterministic setting where the full knowledge of $\mu$ is used and in the unregularized case, the convergence of a Newton algorithm for semi-discrete optimal transport has been studied in \tcb{depth} in \cite{merigot18}.  An extension of the formulation of semi-discrete OT to include an entropic regularization  is  proposed in \cite{cutpey2018}. The main advantage of incorporating such a regularization term in classical OT is to obtain a dual formulation leading to a smooth convex minimization problem allowing the implementation of simple and numerically more stable algorithms as shown in \cite{cutpey2018}.  \tcb{The use of regularized semi-discrete OT has then found applications in image processing using generative adversarial  networks  \cite{pmlr-v84-genevay18a,NEURIPS2018_5a9d8bf5}. In these works,  samples from the generative model are typically drawn  from an  absolutely continuous source measure in order to fit a discrete target distribution. } 

\subsection{Main contributions and structure of the paper}

As discussed above, we introduce a  stochastic Gauss-Newton (SGN)  algorithm  for regularized semi-discrete OT for the purpose of estimating $W_{\varepsilon}(\mu,\nu)$, and the main goal of this paper is to study the statistical properties of such an approach. This algorithm is shown to be adaptive to the geometry of the underlying convex optimization problem with no important hyperparameter to be accurately tuned. Then, the main contributions of our work are to derive the almost sure rates of convergence, the asymptotic normality and the non-asymptotic rates of convergence (in expectation) of various estimators of interest that are constructed using the SGN algorithm to be described below. Although the underlying stochastic optimization problem is not strongly convex, fast rates of convergence can be obtained by \tcb{combining}
the so-called notion  of {\it generalized self-concordance} introduced in   \cite{Bach14} that has been shown to hold for regularized OT in \cite{Stochastic_Bigot_Bercu}, 
and the Kurdyka-\L ojasiewicz inequality as studied in \cite{gadat:hal-01623986}. We also report the results from various numerical experiments on simulated data to illustrate the finite sample properties of this algorithm, and to compare its performances with those of the stochastic gradient descent (SGD), stochastic Newton (SN) and ADAM \cite{KingmaB14} algorithms for stochastic regularized OT.

The paper is then organized as follows.  The definitions of regularized semi-discrete OT and the stochastic algorithms used for solving this problem are given in Section \ref{sec:def}. The main results of the paper are stated in Section \ref{sec:main}, while the important properties related to the regularized OT are given in Section \ref{sec:useful}.  In Section \ref{sec:computation}, we describe  a fast implementation of the  SGN algorithm for regularized semi-discrete OT, and we assess the numerical ability of SGN to solve OT problems. In particular, we report numerical experiments on simulated data to compare the performances of various stochastic algorithms for regularized semi-discrete OT. The statistical properties of the SGN  algorithm are established in an extended  Section \ref{sec:prop} that gathers the proof of our main results.   Finally,  two technical appendices \ref{Appendix} and \ref{Appendix-KL}  contain the proofs of auxiliary results.


\section{A stochastic Gauss-Newton algorithm for regularized semi-discrete OT}
\label{sec:def}
In this section, we introduce the notion of regularized semi-discrete OT and the stochastic algorithm that we propose to solve this problem.

\subsection{Notation, definitions and main assumptions on the OT problem}\label{subsec:defOT}

Let  $\mathcal{X}$ and $\mathcal{Y}$ be two metric spaces. Denote by $\mathcal{M}_+^1(\mathcal{X})$ and $\mathcal{M}_+^1(\mathcal{Y})$ the sets of probability measures on $\mathcal{X}$ and $\mathcal{Y}$, respectively. Let ${\mathbf 1}_J$ be the column vector of $\R^J$ with all coordinates equal 
to one, and denote by $\langle \ ,\ \rangle$ and $\| \ \|$ the standard inner product and norm in $\R^{J}$.  We also use $\lambda_{\max}(A)$ and $\lambda_{\min}(A)$ to denote the largest  and  smallest eigenvalues of a symmetric matrix $A$, whose spectrum will be denoted by $\Sp(A)$ and Moore-Penrose inverse by $A^{-}$. By a slight abuse of notation, we sometimes denote by  $\lambda_{\min}(A)$ the smallest non-zero eigenvalue of  a positive semi-definite matrix $A$. Finally, $\| A \|_{2}$ and  $\| A \|_F$ 
stand for the operator and Frobenius norms of $A$, respectively.
For $\mu \in \MM_{+}^1 (\XX)$ and $\nu \in \MM_{+}^{1}(\YY)$, let $\Pi (\mu,\nu)$ be the set of probability measures on  $\XX \times \YY$ with marginals $\mu$ and $\nu$. As formulated in \cite{NIPS2016_6566}, the problem of entropically regularized optimal transport between
$\mu \in \mathcal{M}_{+}^{1}(\XX)$ and $\nu \in \mathcal{M}_{+}^{1}(\YY)$ is defined as follows.

\begin{deff}   \label{def:primal}
For any  $(\mu,\nu) \in \mathcal{M}_{+}^{1}(\XX) \times \mathcal{M}_{+}^{1}(\YY)$, the Kantorovich formulation of the  
regularized optimal transport between $\mu$ and $\nu$ is the following convex minimization problem 
\begin{equation}
 \quad W_{\varepsilon}(\mu,\nu) = \min_{ \substack{\pi \in \Pi(\mu,\nu)} } \int_{\XX \times \YY} c(x,y) d\pi(x,y) + \varepsilon \textrm{KL} (\pi | \mu \otimes \nu), 
 \label{Primal}
\end{equation}
where $c : \XX\times\YY \to \R$ is a lower semi-continuous function referred to as 
the cost function of moving mass from location $x$ to $y$,  $\varepsilon \geq 0$ is 
a  regularization parameter, and $ \textrm{KL}$
stands for the \textcolor{black}{Kullback-Leibler divergence between $\pi$ and a positive measure $\xi$ on $\XX \times \YY$, up to the  additive term $\int_{\XX \times \YY}  d\xi(x,y)$, namely}
$$
 \textrm{KL}(\pi|\xi) = \int_{\XX \times \YY} \Bigl( \log \Bigl( \dfrac{d\pi}{d\xi}(x,y)\Bigr) -1  \Bigr) d\pi (x,y).
 $$
\end{deff}

\noindent
For $\varepsilon = 0$, the quantity $W_{0}(\mu,\nu)$ is the {\it standard OT cost},  while for $\varepsilon > 0$,
we  refer to $W_{\varepsilon}(\mu,\nu)$ as the {\it regularized OT cost} between 
the two probability measures $\mu$ and $\nu$. In this framework, we shall consider cost functions that are lower semi-continuous and that satisfy the following standard assumption (see e.g.\ \cite[Part I-4]{villani2008optimal}), for all $(x,y) \in \XX\times\YY$,
\begin{equation}
\label{Condcost}
0 \leq c(x,y) \leq c_{\XX}(x) + c_{\YY}(y),
\end{equation}
where $c_{\XX}$ and $c_{\YY}$ are real-valued functions such that $\int_{\XX} c_{\XX}(x) d\mu(x)  < + \infty$ and  $\int_{\YY} c_{\YY}(y) d\nu(y)  < + \infty$. 
Under condition \eqref{Condcost}, $W_{\varepsilon}(\mu,\nu)$ is finite regardless any value of the regularization parameter $\varepsilon \geq 0$. Moreover, note that $W_{\varepsilon}(\mu,\nu)$ can be negative for $\varepsilon > 0$, and that we always have the lower bound
$W_{\varepsilon}(\mu,\nu) \geq - \varepsilon$.
In this paper, we concentrate on {\it the regularized case where $\varepsilon > 0$}, and on the {\it semi-discrete setting} where $\mu \in \mathcal{M}_{+}^{1}(\XX)$ is an arbitrary probability measure (e.g.\ either discrete  or absolutely continuous with respect to the Lesbesgue measure), and $\nu$ is a discrete measure with finite support $\YY = \{y_1,\ldots,y_J\}$ that can be written as
\begin{equation*}
\label{measure-nu}
\nu = \sum_{j=1}^{J} \nu_{j} \delta_{y_j}.
\end{equation*}
Here,  $\delta$ stands for the standard Dirac measure, the locations $\{y_1,\ldots,y_J\}$  as well as the positive weights  $\{\nu_1, \ldots, \nu_J\}$ 
are assumed to be known and summing up to one.  We shall also use the notation
$$
\min(\nu) = \min_{1 \leq j \leq J}\{\nu_j\} \hspace{1.5cm} \mbox{and} \hspace{1.5cm} \max(\nu) = \max_{1 \leq j \leq J}\{\nu_j\}.
$$
 We shall also sometimes refer to the {\it discrete setting} when $\mu$ is also a discrete measure.
Let us now define the semi-dual formulation of the minimization problem \eqref{Primal} as introduced in \cite{NIPS2016_6566}. In the semi-discrete setting and for $\varepsilon > 0$, using the semi-dual formulation of the  minimization problem \eqref{Primal}, it follows that $W_{\varepsilon}(\mu,\nu)$ can be expressed as the following convex optimization problem
\begin{equation} 
\label{Semi-dualdisc}
W_{\varepsilon}(\mu,\nu) = - \inf_{\substack{v \in \R^J}} H_{\varepsilon}(v)
\end{equation}
with
\begin{equation} 
\label{DefH}
H_{\varepsilon} (v) = \E[ h_{\varepsilon}(X,v) ] = \int_{\XX} h_\varepsilon (x,v) d\mu(x), 
\end{equation}
where  $X $ stands for a random variable drawn from the unknown distribution $\mu$, and for any $(x,v) \in \XX \times  \R^J$,
\begin{equation} \label{Defh}
h_\varepsilon (x,v) =  
\varepsilon + \varepsilon \log 
\Bigl( \sum_{j=1}^{J} \exp \Bigl( \dfrac{ v_j - c(x,y_j) }{\varepsilon} \Bigr) \nu_j  \Bigr)
- \sum_{j=1}^{J} v_j \nu_j .
\end{equation}

\noindent
Throughout the paper, we shall assume that,  for any $\varepsilon > 0$, there exists $v^\ast \in \R^J$ that minimizes the function $H_{\varepsilon}$, leading to
$$
W_{\varepsilon}(\mu,\nu) = - H_{\varepsilon}(v^{\ast}).
$$
The above equality is the  key  result allowing to formulate \eqref{Semi-dualdisc} as a convex stochastic minimization problem, and to consider the issue of estimating $W_{\varepsilon}(\mu,\nu)$  in the setting of stochastic optimization. For a discussion on sufficient conditions implying the existence of such a minimizer $v^{\ast}$, we refer to \cite[Section 2]{Stochastic_Bigot_Bercu}. As discussed in Section \ref{sec:useful},  the function $H_{\varepsilon}$ possesses a one-dimensional subspace of global minimizers, defined by $\{v^\ast+ t \bv_J, \ t \in \R\}$ where 
$$
\bv_J=\frac{1}{\sqrt{J}}\mathbf{1}_J.
$$
Therefore, we will constrain our algorithm to live in $\langle \bv_J \rangle^\perp$, which denotes the orthogonal complement of the one-dimensional subspace  
$\langle \bv_J \rangle$  of $\R^J$  spanned by
$\bv_J$. In that setting, the  OT problem \eqref{Semi-dualdisc} becomes identifiable.

\subsection{\tcb{Pre-conditionned} stochastic algorithms}\label{sec:NewtonStoAlgo}

In the context of regularized OT, we first introduce a general class of stochastic \tcb{pre-conditionned} algorithms that are also referred to as 
quasi-Newton algorithms in the literature. Starting from Section \ref{subsec:defOT}, our approach is inspired by the recent works \cite{Stochastic_Bigot_Bercu,NIPS2016_6566}, which use the property that
$$
W_{\varepsilon}(\mu,\nu) = -H_{\varepsilon} (v^\ast) = - \min_{\substack{v \in \R^J}} \E[ h_{\varepsilon}(X,v) ]
$$
where $h_{\varepsilon}(x,v)$ is the smooth function defined by \eqref{Defh} that is simple to compute.
For a sequence $(X_n)$ of independent and identically distributed  random variables sampled from the
distribution $\mu$, the class of pre-conditionned stochastic
algorithms is defined as the following family of recursive stochastic algorithms to estimate the minimizer $v^\ast \in  \langle \bv_J \rangle^\perp$ of $H_{\varepsilon}$. 
These algorithms can be written as
\begin{equation}
\wh{V}_{n+1} = \proj \left( \wh{V}_n - n^{\alpha} S_n^{-1} \nabla_v h_{\varepsilon}(X_{n+1}, \wh{V}_n) \right)   \label{eq:SNgen}
\end{equation}
for some constant $0 \leq \alpha < 1/2$, where $\nabla_v h_{\varepsilon}$ stands for the gradient of $h_{\varepsilon}$ with respect to $v$, 
$\wh{V}_{0}$  is a random vector belonging to $\langle \bv_J \rangle^\perp$, and $S_n$ is a symmetric and {\it  positive definite} $J \times J$ random matrix 
which is measurable with respect to the $\sigma$-algebra $\cF_{n}=\sigma(X_1,\ldots,X_{n})$. In addition, $P_J$ is the orthogonal projection matrix 
onto $\langle \bv_J \rangle^\perp$,
$$
\proj =  \Id - \bv_J \bv_J^T.
$$ 
This stochastic algorithm allows us to estimate $W_{\varepsilon}(\mu,\nu)$ by the recursive estimator
\begin{equation}
\label{DefWn}
\wh{W}_n = -\frac{1}{n}\sum_{k=1}^n  h_{\varepsilon}(X_k, \wh{V}_{k-1}). 
\end{equation}
The special case where $S_n = s^{-1} n \Id$ with some constant $s > 0$, corresponds to the well-known stochastic gradient descent (SGD) 
algorithm that has been introduced in the context of stochastic OT in \cite{NIPS2016_6566}, and recently investigated in \cite{Stochastic_Bigot_Bercu}. Following  some recent contributions \cite{Godichon2020,Bercu_Godichon_Portier2020} in stochastic optimization using Newton-type stochastic algorithms, another potential choice 
is $S_n = \mathbb{S}_n$ where  $\mathbb{S}_n$ is the natural Newton recursion defined as
\begin{align}
\mathbb{S}_n = \Id + \sum\limits_{k=1}^{n} \nabla_v^2 \he(X_{k}, \wh{V}_{k-1})   = \mathbb{S}_{n-1}  + \nabla_v^2 \he(X_{n}, \wh{V}_{n-1}) \label{eq:SnSN}
 \end{align}
where $\nabla_v^2 h_{\varepsilon}$ stands for the Hessian matrix of $h_{\varepsilon}$ with respect to $v$, 
We refer to the choice \eqref{eq:SnSN} for $\mathbb{S}_n$ as the  stochastic Newton (SN) algorithm. Unfortunately, from a computational point of view, a major limitation of this SN algorithm is the need to compute the inverse of $\mathbb{S}_n$ at each iteration $n$ in equation \eqref{eq:SNgen}. As $\mathbb{S}_n$ is given by the recursive equation \eqref{eq:SnSN}, it is tempting to use the  Sherman-Morrison-Woodbury  (SMW) formula \cite{Hager1989}, that is recalled in Lemma \ref{lem:SMW}, in order to compute $\mathbb{S}_n^{-1}$ from the knowledge of $\mathbb{S}_{n-1}^{-1}$ in a recursive manner. However, as detailed in Section \ref{sec:recSN} of Appendix A, the Hessian matrix $\nabla_v^2 \he(X_{n}, \wh{V}_{n-1})$ does not have a sufficiently low-rank structure that would lead to a fast recursive approach to compute 
$\mathbb{S}_n^{-1}$. Therefore, for the SN algorithm, the computational cost to evaluate $\mathbb{S}_n^{-1}$  appears to be of order $\mathcal{O}(J^3)$ which only leads to a feasible algorithm  for  very small values of $J$. This important computational limitation then drew our investigation towards the SGN algorithm instead of the SN approach.

\subsection{The  stochastic Gauss-Newton algorithm \label{sec:GN}} 
Historically, the  Gauss-Newton adaptation of the Newton algorithm consists in replacing the  Hessian matrix $\nabla_v^2 \he(X_{n}, \wh{V}_{n-1})$ by a tensor product of the gradient  $\nabla_v \he(X_{n}, \wh{V}_{n-1})$. In our framework, it leads to another pre-conditionned stochastic
algorithm. We introduce $S_n$ recursively as  
\begin{eqnarray}
S_n & = & \Id +\sum\limits_{k=1}^{n}  \nabla_v \he(X_{k}, \wh{V}_{k-1}) \nabla_v \he(X_{k}, \wh{V}_{k-1})^T + \gamma   \left(1+  \left\lfloor \frac{k}{J} \right\rfloor \right)^{-\beta}  Z_k Z_k^T   \label{eq:SnSGN} \\
& = & S_{n-1} +   \nabla_v \he(X_{n}, \wh{V}_{n-1}) \nabla_v \he(X_{n}, \wh{V}_{n-1})^T+ \gamma  \left(1+  \left\lfloor \frac{n}{J} \right\rfloor \right)^{-\beta}  Z_{n} Z_{n}^T,  \nonumber
 \end{eqnarray}
for some constants $\gamma > 0$ and  $0 < \beta < 1/2$,  and where $(Z_1,\ldots,Z_{n})$ is a {\it deterministic} sequence of vectors defined, for all $1 \leq k \leq n$, by
$$
Z_k = \sqrt{\nu_{\ell_k}} e_{\ell_k} 
$$
with $\ell_k = 1+ (k-1) \tpmod{J}$, where $(e_1,\ldots,e_J)$ stands for the canonical basis of $\R^J$. We shall refer to the choice \eqref{eq:SnSGN} for $S_n$ as the regularized stochastic Gauss-Newton (SGN) algorithm and from now on, the notation $S_n$ refers to this definition. We also use the convention that $S_0=I_J$.


\section{Main results on the SGN algorithm}\label{sec:main}

Throughout this section, we investigate the statistical properties of the recursive sequence $(\wh{V}_n)$ defined by \eqref{eq:SNgen} with $0 \leq \alpha < 1/2$, where 
$(S_n)$ is the sequence of random matrices defined by  \eqref{eq:SnSGN} with $0 < \beta < 1/2$ that yields the SGN algorithm. The initial value  $\wh{V}_0$ is assumed to be a square integrable random vector that belongs to $\langle \bv_J \rangle^\perp$. Then, thanks to the projection step in equation \eqref{eq:SNgen},  it follows that 
for all $n \geq 1$, $\wh{V}_n$ also belongs to  $\langle \bv_J \rangle^\perp$. To derive the convergence properties of the SGN algorithm, we first need to introduce the matrix-valued function $G_{\varepsilon}(v)$ defined as 
\begin{equation}
G_{\varepsilon}(v) = \mathbb{E}\left[\nabla_{v} \he(X, v) \nabla_{v} \he(X, v)^T\right] \label{eq:defG}
\end{equation}
that will be shown to be a key quantity to analyze the SGN algorithm. In particular, we shall derive our results under the following  assumption on the smallest eigenvalue 
of $G_{\varepsilon}(v^\ast)$ associated to eigenvectors belonging to $\langle \bv_J \rangle^\perp$.
\ \vspace{2ex}\\
{\bf Invertibility assumption.}
The matrix $G_{\varepsilon}(v^\ast)$ satisfies
$$
\min_{v \in \langle \bv_J \rangle^\perp} \Bigl \{\frac{v^T G_{\varepsilon}(v^\ast) v }{\|v\|^2} \Bigr \}> 0.
$$

\noindent
In all the sequel, we suppose that this invertibility assumption is satisfied. We denote by $G_{\varepsilon}^{-}(v^\ast)$ the Moore-Penrose inverse of 
$G_{\varepsilon}(v^\ast)$ and by $G_{\varepsilon}^{-1/2}(v^\ast)$ its square-root.
 We now discuss, in what follows, the next keystone inequality.
\begin{prop}\label{prop:keystone}
Assume that the regularization parameter $\varepsilon>0$ satisfies 
\begin{equation}
\varepsilon \leq    \frac{\min(\nu)}{\max(\nu) - \min(\nu)}. \label{eq:condeps}
\end{equation}
Then, in the sense of partial ordering between positive semi-definite matrices, we have
\begin{equation}
G_{\varepsilon}(v^\ast) \leq  \nabla^2 \He(v^\ast). \label{ineqGH}
\end{equation}
\end{prop}

\noindent
Inequality \eqref{ineqGH} is an important property of the SGN algorithm to prove its adaptivity to the geometry of the stochastic optimization problem \eqref{Semi-dualdisc}. 
\tcb{Of course, one can observe that no hyperparameter  depending on the Hessian of $H_{\varepsilon}$ needs to be tuned to run this algorithm provided that condition \eqref{eq:condeps} holds.}
One can also remark that there is no restriction on the regularization parameter $\varepsilon$ when $\nu$ is the uniform distribution, that is when $\nu_j = 1/J$, for all $1 \leq j \leq J$, implying that $\max(\nu) = \min(\nu)$. Throughout the paper, we suppose that condition \eqref{eq:condeps} holds true.
\tcb{Below, we denote by $\lambda^{\langle \bv_J \rangle^\perp}_{\min}(A)$ the smallest non-zero eigenvalue of a positive semi-definite matrix $A$, when the associated eigenvectors  belong to $\langle \bv_J \rangle^\perp$, the orthogonal complement of $\bv_J$ .}

It immediately follows from inequality \eqref{ineqGH} that
\begin{equation} \label{ineq:eiglower}
1 \leq \tcb{\lambda_{\min}^{\langle \bv_J \rangle^\perp}(G_{\varepsilon}^{-1/2}(v^\ast) \nabla^2 H_{\varepsilon}(v^\ast)G_{\varepsilon}^{-1/2}(v^\ast))}.
\end{equation}
\tcb{Inequality \eqref{ineq:eiglower} will be a key property in this paper to derive the rates of convergence of the estimators obtained from the SGN algorithm. Note that the (pseudo) inverse of the Hessian matrix $\nabla^2 \He(v^\ast)$ somehow represents an ideal deterministic pre-conditioning  matrix, whose use would lead to the second order Newton algorithm: this ideal pre-conditioned algorithm is \textit{non-adaptive} since it requires the use of $\nabla^2 \He(v^\ast)$, which is unknown in practice.}

\tcb{Indeed, in our SGN algorithm, adaptivity is tightly related to the limiting recursion induced by Equation \eqref{eq:SNgen}. If we admit (temporarily) the almost sure convergence of the SGN algorithm towards $v^*$ and of $n^{-1} S_n$ towards $ G_{\epsilon}(v^*)$
, the recursion induced by \eqref{eq:SNgen} looks very similar to a discretization of a dynamical system with a step size $n^{-(1-\alpha)}$ and with a limiting linearized drift of the form $-G_{\epsilon}(v^*)^{-1} \nabla^2 H_{\epsilon}(v^*)(v-v^*)$. For further details on this point, we refer to the so-called ODE method (see \textit{e.g.} \cite{benaim1996asymptotic} ).
The keystone property induced by Proposition \ref{prop:keystone} is that thanks to Equation \eqref{ineqGH} and Equation \eqref{ineq:eiglower}, the linearized drift of the limiting deterministic dynamical system has its eigenvalues that are lower bounded by $1$, regardless of the value of the Hessian matrix $\nabla^2 H_{\epsilon}(v^*)$. This  translates  an adaptation of the algorithm to the curvature of $ H_{\epsilon}$ near the target point $v^*$. Therefore, the matrix $G_\epsilon(v^*)$ that is learnt on-line, and automatically adapts to the eigenspaces associated to the smallest eigenvalues of $\nabla^2 H_{\epsilon}(v^*)$.} Therefore, one may  interpret  \eqref{ineq:eiglower} as the adaptivity of the SGN algorithm to the geometry of the semi-dual formulation of regularized OT.

\subsection{Almost sure convergence}
 
The almost sure convergence of the sequences $(\wh{V}_n)$, $(\wh{W}_n)$ and $(\overline{S}_n)$ are as follows where
$
\overline{S}_n = \frac{1}{n} S_n.
$

\begin{thm} \label{theo:asconvVn} Assume that $\alpha \in  [0,1/2[$ and  $\alpha +\beta < 1/2$. Then, we have 
\begin{equation}
\lim_{n \to + \infty} \wh{V}_n = v^{\ast}\hspace{1cm} \text{a.s.} \label{eq:convasVn}
\end{equation}
and
\begin{equation}  \label{eq:convasSnSGN}
\lim_{n \to + \infty} \overline{S}_n  = G_{\varepsilon}(v^\ast) \hspace{1cm} \text{a.s.}
\end{equation}
\end{thm}

\noindent
The following result   is an immediate corollary of Theorem \ref{theo:asconvVn}, thanks to the continuity of the function $h_\varepsilon$.
\begin{cor}\label{cor:convWnas}
Assume that $\alpha \in  [0,1/2[$ and   $\alpha +\beta < 1/2$. Suppose that the cost function $c$ satisfies, for any $1 \leq j \leq J$,
\begin{equation}
\label{Integrabilitycost2}
\int_{\XX} c^2(x,y_j) d\mu(x) < + \infty.
\end{equation}
Then, we have
$$
\lim_{n \to + \infty} \wh{W}_n = W_{\varepsilon}(\mu,\nu) \hspace{1cm} \text{a.s.}
$$
\end{cor}

\noindent
We now derive results on the almost sure rates of convergence of the sequences $(\wh{V}_n)$ and $(\overline{S}_n)$ that are the keystone in the proof of the asymptotic normality of the estimator $\wh{V}_n$  studied in Section \ref{sec:TCL}. We emphasize that we restrict our study to the case $\alpha = 0$, which yields the fastest rates of convergence and that corresponds to the meaningful situation from the numerical point of view.

\begin{thm} \label{theo:asrates}
Assume that  $\alpha = 0$. Then, we have the almost sure rate of convergence
\begin{equation}
\bigl \|\wh{V}_n - v^{\ast} \bigr \|^2 =\mathcal{O}\Bigl(\frac{\log n}{n} \Bigr) \hspace{1cm} \text{a.s.} \label{eq:asrateVn}
\end{equation}
In addition, we also have
\begin{equation}  \label{eq:asrateSnSGN}
\bigl \|  \overline{S}_n- G_{\varepsilon}(v^\ast) \bigr \|_F= \mathcal{O}\Bigl(\frac{1}{n^\beta} \Bigr)  \hspace{1cm} \text{a.s.}
\end{equation}
and
\begin{equation}  \label{eq:asrateinvSnSGN}
\bigl\|  \overline{S}_n^{-1}- G_{\varepsilon}^{-} (v^\ast) \bigr \|_F= \mathcal{O}\Bigl(\frac{1}{n^\beta} \Bigr) \hspace{1cm} \text{a.s.}
\end{equation}
\end{thm}

\subsection{Asymptotic normality} \label{sec:TCL}

The asymptotic normality of our estimates depends on the magnitude of the smallest eigenvalue (associated to eigenvectors belonging to 
$\langle \bv_J \rangle^\perp$)  of the matrix
\begin{equation}  
\label{eq:defGammaast}
\Gamma_{\varepsilon}(v^\ast) = G_{\varepsilon}^{-1/2}(v^\ast) \nabla^2 H_{\varepsilon}(v^\ast)G_{\varepsilon}^{-1/2}(v^\ast).
\end{equation}
Thanks to the key inequality \eqref{ineq:eiglower}, we have that the smallest  eigenvalue of $\Gamma_{\varepsilon}(v^\ast)$  is always greater than $1$, in the sense that
$$
\min_{v \in \langle \bv_J \rangle^\perp} \frac{v^T \Gamma_{\varepsilon}(v^\ast) v }{\|v\|^2} \geq 1.
$$
One can observe that we also restrict our study to the case $\alpha = 0$ which yields the usual $\sqrt{n}$ rate of convergence for the central limit theorem  that is stated below.
 
\begin{thm} \label{theo:anVn}
Assume that   $\alpha = 0$. Then, we have the asymptotic normality
\begin{equation}  \label{eq:anVn}
\sqrt{n} \bigl( \wh{V}_n -v^\ast \bigr) \liml \mathcal{N}\Bigl(0, G_{\varepsilon}^{-1/2}(v^\ast)  \bigl( 2 \Gamma_{\varepsilon}(v^\ast) -\proj \bigr)^{-} G_{\varepsilon}^{-1/2}(v^\ast)\Bigr).
\end{equation}
In addition,  suppose that the cost function $c$ satisfies, for any $1 \leq j \leq J$,
\begin{equation}
\label{Integrabilitycost4}
\int_{\XX} c^4(x,y_j) d\mu(x) < + \infty.
\end{equation}
Then, we also have
\begin{equation}  \label{eq:anWn}
\sqrt{n} \bigl( \wh{W}_n -W_{\varepsilon}(\mu,\nu) \bigr) \liml \mathcal{N}\bigl(0,\sigma_\varepsilon^2\bigr)
\end{equation}
where the asymptotic variance $\sigma_{\varepsilon}^2 = \E[h^2_\varepsilon(X,v^\ast)] -  W^2_{\varepsilon}(\mu,\nu)$.
 \end{thm}
 
 \noindent
 In order to discuss the above result on the asymptotic normality of $ \wh{V}_n$,  we denote by
 $$
\Sigma_{\varepsilon}(v^\ast)= G_{\varepsilon}^{-1/2}(v^\ast)  \bigl( 2 \Gamma_{\varepsilon}(v^\ast) -\proj \bigr)^{-} G_{\varepsilon}^{-1/2}(v^\ast)
 $$
 the asymptotic covariance matrix in \eqref{eq:anVn}. 
 One can check  that $ \Sigma_{\varepsilon}(v^\ast)$ satisfies the Lyapunov equation
 \begin{equation}
 \Bigl(\frac{1}{2}\proj - A   \nabla^2 H_{\varepsilon}(v^\ast) \Bigr) \Sigma_{\varepsilon}(v^\ast)  + \Sigma_{\varepsilon}(v^\ast) \Bigl(\frac{1}{2}\proj - A   \nabla^2 H_{\varepsilon}(v^\ast) \Bigr) ^T 
 = - A G_{\varepsilon}(v^\ast) A  \label{eq:Lyap}
\end{equation}
with $A =  G_{\varepsilon}^{-}(v^\ast)$. Moreover, one can observe that
 $$
 G_{\varepsilon}(v^\ast) =  \lim_{n \to + \infty} \E[\varepsilon_{n+1}\varepsilon_{n+1}^T| \cF_n] \hspace{1cm} \text{a.s.}
 $$
 is the asymptotic covariance matrix of the martingale increment $\varepsilon_{n+1} = \nabla_v \he(X_{n+1}, \wh{V}_n) - \nabla \He(\wh{V}_n)$. Hence, to better interpret the asymptotic covariance matrix $ \Sigma_{\varepsilon}(v^\ast)$, let us consider the following sub-class  of pre-conditionned stochastic algorithms 
\begin{equation}
\wt{V}_{n+1} = \proj \Bigl( \wt{V}_n - \frac{1}{n} A \nabla_v h_{\varepsilon}(X_{n+1}, \wt{V}_n) \Bigr), \label{eq:deterpre-cond}
\end{equation}
 where $A$ is a deterministic positive semi-definite matrix satisfying the stability condition
 \begin{equation}
A  \nabla^2 H_{\varepsilon}(v^\ast) \geq \frac{1}{2}\proj. \label{eq:stabcond}
\end{equation}
Then, adapting well-known results on stochastic optimisation (see e.g.\ \cite{Duflo1997,Pelletier}), one may prove that
$$
\sqrt{n} \bigl( \wt{V}_n -v^\ast \bigr) \liml \mathcal{N}\bigl(0, \Sigma(A)\bigr)
$$
where $ \Sigma(A)$ is the solution of  the Lyapunov equation \eqref{eq:Lyap} with $ \Sigma(A)$  instead of  $\Sigma_{\varepsilon}(v^\ast)$. Hence, the asymptotic 
normality of the SGN algorithm coincides with the one of the  \tcb{pre-conditionned} stochastic algorithm \eqref{eq:deterpre-cond} for the  choice $A = G_{\varepsilon}^{-}(v^\ast)$. Hence, the main advantage of the SGN algorithm is to be fully data-driven as $G_{\varepsilon}^{-}(v^\ast)$ is obviously unknown. Among the deterministic pre-conditionning matrices satisfying condition \eqref{eq:stabcond},  it is also known (see e.g.\ \cite{Duflo1997,Pelletier}) that the best choice is to take 
$A =  \nabla^2 H_{\varepsilon}^{-}(v^\ast)$ that corresponds to an ideal Newton algorithm and which yields the optimal asymptotic covariance matrix
$$
\Sigma^\ast =  \Sigma\left( \nabla^2 H_{\varepsilon}(v^\ast) \right) = \nabla^2 H_{\varepsilon}(v^\ast)^{-}  G_{\varepsilon}(v^\ast)  \nabla^2 H_{\varepsilon}(v^\ast)^{-} \leq \Sigma_{\varepsilon}(v^\ast).
$$
Therefore, the SGN algorithm does not yield an estimator $\wh{V}_n$ having an asymptotically optimal covariance matrix. Note that, as shown in \cite[Theorem 3.4]{Stochastic_Bigot_Bercu}, using an average version of the standard SGD algorithm, that is with $S_n  = s^{-1} n \Id $ and $0 < \alpha < 1/2$, allows to obtain an estimator having an asymptotic distribution with optimal covariance matrix $\Sigma^\ast$. However, in numerical experiments, it appears that the choice of $s$ for the averaged SGD algorithm is crucial but difficult to tune. The results from \cite{Stochastic_Bigot_Bercu} suggests to take $s=\varepsilon/(2\min(\nu))$ which follows from the property that
$$
\lambda_{\min}\left( \nabla^{2} H_{\varepsilon}(v^\ast) \right)    \geq \frac{\min(\nu)}{\varepsilon}, 
$$
that is discussed in Section \ref{sec:useful}. Hence,  the choice  $s =\varepsilon/(2\min(\nu))$ ensures that the pre-conditioning matrix
$
A  = s^{-1} \Id 
$
satisfies the stability condition \eqref{eq:stabcond}. However,  as shown by the numerical experiments carried out in Section \ref{sec:computation}, 
it appears that the SGN algorithm automatically  adapts to the geometry of the optimisation problem with better results than the SGD algorithm. 
Finally, we remark from the asymptotic normality \eqref{eq:anWn} and \cite[Theorem 3.5]{Stochastic_Bigot_Bercu} that the asymptotic variance of 
the recursive estimator $\wh{W}_n$ is the same when $\wh{V}_n $ is either computed using the SGN or the SGD algorithm.

\subsection{Non-asymptotic rates of convergence} \label{sec:rate}

The last contribution of our paper is to derive non-asymptotic upper bounds on the expected risk of various estimators arising from the use of the SGN algorithm when $(S_n)$ is the sequence of positive definite  matrices defined by \eqref{eq:SnSGN}.
In particular, we derive the rate of convergence of the expected quadratic risks
$$
\E \bigl [ \bigl \|\wh{V}_n - v^\ast \bigr \|^{2} \bigr ] \hspace{1cm} \text{and} \hspace{1cm}  
\E \bigl [ \bigl \|\overline{S}_n -  G_{\varepsilon}(v^\ast)\bigr \|^{2}_{F}\bigr ].
$$
We also analyze the rate of convergence of the expected excess risk
$
W_{\varepsilon}(\mu,\nu) - \E [\wh{W}_n]
$
of the recursive estimator $\wh{W}_n$  defined by \eqref{DefWn}  used to approximate the regularized OT cost $W_{\varepsilon}(\mu,\nu)$.

\begin{thm} \label{theo:rates-nonasymp}
Assume that $\alpha \in ]0,1/2[$ and that $4\beta<1-2\alpha$. Then, there exists a positive constant $c_{\varepsilon}$ such that for any $n \geq 1$,
\begin{equation}  \label{eq:nonasymp1}
\E \bigl [ \bigl \|\wh{V}_n - v^\ast \bigr \|^{2} \bigr ] \leq \frac{c_{\varepsilon}}{n^{1-\alpha}} \hspace{1cm} \text{and} \hspace{1cm} 
\E \bigl [ \bigl \|\overline{S}_n -  G_{\varepsilon}(v^\ast)\bigr \|^{2}_{F}\bigr ] \leq \frac{c_{\varepsilon}}{n^{2 \beta}}.
\end{equation}
Moreover, we also have
\begin{equation}  \label{eq:nonasymp2}
\bigl|\E [ \wh{W}_n ] -W_{\varepsilon}(\mu,\nu)\bigr|\leq  \frac{c_{\varepsilon}}{n^{1-\alpha}},
\end{equation}
and if  the cost function $c$ satisfies
$
\int_{\XX} c^2(x,y_j) d\mu(x) < + \infty,
$
for any $1 \leq j \leq J$, then
\begin{equation}  \label{eq:nonasymp3}
\E \bigl[\bigl| \wh{W}_n-W_{\varepsilon}(\mu,\nu)\bigr| \bigr]\leq \frac{c_{\varepsilon} }{\sqrt{n}}.
\end{equation}
 \end{thm}

Note that the value of the constant $c_{\varepsilon}$ appearing in Theorem \ref{theo:rates-nonasymp}  may also depend on $\alpha$ and $\beta$, but we remove this dependency in the notation to simplify the presentation. One can observe that choosing $\alpha > 0$ allows the algorithm to be fully adaptative in the sense that no important hyperparameter needs to be tune to obtain non-asymptotic rates of convergence. The case $\alpha =0$ could also be considered but this will require to introduce a multiplicative positive constant $c$ in the definition of the SGN algorithm by replacing $n^{\alpha}S_n^{-1}$  in equation \eqref{eq:SNgen} by $c S_n^{-1}$. Then, provided that $c$ is sufficiently large, one may obtain faster rate of convergence for the expected quadratic risk of the order $\mathcal{O}(\log(n)/n)$.  However, in our numerical experiments, we have found that introducing such a large multiplicative constant $c$ makes the convergence of the SGN algorithm too slow. Therefore, results on non-asymptotic convergence rates in the case $\alpha = 0$ are not reported here.


\section{\tcr{Implementation of the SGN algorithm and numerical experiments}} \label{sec:computation}


\tcr{In this section, we first discuss computational considerations on the implementation of  the SGN algorithm, and we also make several remarks  to justify its use. Then, we report the results of numerical experiments.}

\subsection{A fast recursive approach to compute $S_n^{-1}$.}  In this paragraph, we discuss on the computational benefits of using the Gauss-Newton method as an alternative to the Newton algorithm.  A key point to define the SGN algorithm consists in replacing in equation \eqref{eq:SnSN} that defines the SN algorithm, the positive definite Hessian matrices $ \nabla_v^2 \he(X_{k}, \wh{V}_{k-1})$ by the tensor product $ \nabla_v \he(X_{k}, \wh{V}_{k-1}) \nabla_v \he(X_{k}, \wh{V}_{k-1})^T$  of the gradient of $\he$ at $(X_{k}, \wh{V}_{k-1})$. \tcr{A second important ingredient in the definition \eqref{eq:SnSGN} of the SGN algorithm} is  the additive regularization terms $\gamma   \left( 1+  \left\lfloor \frac{k}{J} \right\rfloor \right)^{-\beta}  Z_k Z_k^T$  whose role is discussed in the sub-section below.

Note that  $Z_k Z_k^T = \nu_{\ell_{k}} e_{\ell_{k}}e_{\ell_{k}}^T$ (with $\ell_{k} = (k-1) \tpmod{J} + 1$) is a diagonal matrix, such that all its diagonal elements are equal to zero, except the $\ell_{k}$-th one which is equal to  $\nu_{\ell_{k}}$. In this manner, the difference
$
S_n-S_{n-1} =  \phi_{n} \phi_{n}^T + \gamma \left( 1+  \left\lfloor \frac{n}{J} \right\rfloor \right)^{-\beta}  Z_{n} Z_{n}^T
$
is thus the sum of two rank one matrices, where $ \phi_{n} =  \nabla_v \he(X_{n}, \wh{V}_{n-1})$. Therefore, 
one may easily obtain $S_{n}^{-1}$ from the knowledge of $S_{n-1}^{-1}$ as follows.  Introducing the intermediate matrix
$S_{n-\frac{1}{2}} = S_{n-1} + \gamma \left( 1+  \left\lfloor \frac{n}{J} \right\rfloor \right)^{-\beta}  Z_{n} Z_{n}^T$, we observe that $S_{n}  = S_{n-\frac{1}{2}}   +\phi_{n} \phi_{n}^T$. Consequently,  by applying  the SMW formula \eqref{eq:SMW}, we first notice that
\begin{align*}
S_{n-\frac{1}{2}}^{-1} &= (S_{n-1}+\gamma \left( 1+  \left\lfloor \frac{n}{J} \right\rfloor \right)^{-\beta}  Z_{n} Z_{n}^T)^{-1} \nonumber \\
& =S_{n-1}^{-1} - (Z_{n}^T S_{n-1}^{-1} Z_{n} + \gamma^{-1} \left( 1+  \left\lfloor \frac{n}{J} \right\rfloor \right)^{\beta} )^{-1} S_{n-1}^{-1} Z_{n} Z_{n}^T S_{n-1}^{-1}
\end{align*}
Using that $Z_{n} = \sqrt{\nu_{\ell_{n}}} e_{\ell_{n}}$, we furthermore have that
\begin{equation}
 \label{eq:rec_inv_Rn1}
S_{n-\frac{1}{2}}^{-1} =S_{n-1}^{-1}-\nu_{\ell_{n}} \frac{(S_{n-1}^{-1})_{.,\ell_{n}} (S_{n-1}^{-1})_{\ell_{n},.}^T}{\nu_{\ell_{n}} (S_{n-1}^{-1})_{\ell_{n},\ell_{n}}+\gamma^{-1}  \left( 1+  \left\lfloor \frac{n}{J} \right\rfloor \right)^{\beta} }.
\end{equation}
Secondly,   applying again the SMW formula \eqref{eq:SMW}, we obtain that
\begin{eqnarray}
S_{n}^{-1} & =&  S_{n-\frac{1}{2}}^{-1} - \frac{S_{n-\frac{1}{2}}^{-1}\phi_{n}  \phi_{n}^T S_{n-\frac{1}{2}}^{-1}}{\phi_{n}^T S_{n-\frac{1}{2}}^{-1} \phi_{n} + 1} . \label{eq:rec_inv_Rn2}
\end{eqnarray}

Hence, the recursive formulas \eqref{eq:rec_inv_Rn1} and \eqref{eq:rec_inv_Rn2}   allow, at each iteration $n$,  a much more faster computation of $S_{n}^{-1}$ from the knowledge of  $S_{n-1}^{-1}$, which is a key advantage of the SGN algorithm  over the use of the SN algorithm. Indeed, the cost of computing $S_n^{-1}$ using the above recursive formulas is that of matrix vector multiplication which is of order $\mathcal{O}(J^2)$.

\subsection{\tcr{The role of regularization}.}  Let us denote by
$
R_{n} = \sum\limits_{k=1}^{n} \gamma   \left( 1+  \left\lfloor \frac{k}{J} \right\rfloor \right)^{-\beta}  Z_k Z_k^T
$
the sum of the deterministic regularization terms in \eqref{eq:SnSGN} implying that $S_n$ can be decomposed as
\begin{equation} \label{eq:SnRSGN}
S_n =  \Id +\sum\limits_{k=1}^{n}  \nabla_v \he(X_{k}, \wh{V}_{k-1}) \nabla_v \he(X_{k}, \wh{V}_{k-1})^T + R_{n}.
\end{equation}
If $n = p J$ for some integer $p \geq 1$, the regularization by the matrices $\gamma  \left( 1+  \left\lfloor \frac{k}{J} \right\rfloor \right)^{-\beta}  Z_k Z_k^T$ in \eqref{eq:SnSGN} sum up to a simple expression given by
$$
R_{n} = \left(\sum\limits_{m=1}^p  m^{-\beta} \right)  \gamma  \diag(\nu).
$$
\tcb{ The following  two important  comments can be made  to clarify the role of this regularization effect:}

\begin{description}
\item[-]
\tcb{adding the supplementary matrix $R_n$ in \eqref{eq:SnRSGN} implies that $S_n$ is invertible as soon as $n\ge J$ with a known lower bound on its smallest eigenvalue. Indeed, thanks to the condition $0 < \beta < 1/2$ and to the property that $$\left( \frac{1}{p} \sum\limits_{m=1}^p  m^{-\beta} \right) \sim  \frac{1}{1-\beta} p^{-\beta} ,$$ the additive  term $R_{n}$ allows to regularize the smallest eigenvalue of $S_n$. This is important for the evolution of the stochastic algorithm: this regularization allows to show that $\wh{V}_n$ converges almost surely to $v^{\ast}$. More precisely, while $\hat{V}_{n+1}-\hat{V}_n$ is essentially modified in the direction $-n^{\alpha} S_n^{-1} \nabla_{v} h_{\epsilon}(X_{n+1},\hat{V}_n)$, it is well known that too large step sizes are prohibited to obtain a good behavior of  stochastic algorithms. Therefore, taking a sufficiently small $\beta$  guarantees a suitable upper bound of the increments of the SGN, that in turn limits the effect of the noise  at each iteration of the algorithm.}

\item[-]
\tcb{the growth of  $R_n$ is sublinear for large values of $n$ whereas 
$$\sum\limits_{k=1}^{n}  \nabla_v \he(X_{k}, \wh{V}_{k-1}) \nabla_v \he(X_{k}, \wh{V}_{k-1})^T$$ grows linearly with $n$ so that this last term will become dominant in the decomposition of $S_n$, inducing a ``learning'' of the curvature of the landscape function $H_{\epsilon}$. Recalling that  $\overline{S}_n = \frac{1}{n} S_n$, it will be shown 
 in Section \ref{sec:prop} that }
$$
\tcb{\lim_{n \to + \infty} \overline{S}_n =  G_{\varepsilon}(v^\ast) \quad \mbox{a.s.}}
$$
\tcb{with}
$$
\tcb{G_{\varepsilon}(v^\ast) = \mathbb{E}\left[\nabla_{v} \he(X, v^\ast) \nabla_{v} \he(X, v^\ast)^T\right] =  \diag(\nu) - \nu\nu^T - \varepsilon \nabla^2 \He(v^\ast),}
$$
\tcb{where the last equality above follows   the proof of Proposition \ref{prop:keystone}. Hence, when $n \longrightarrow +\infty$, the regularization disappears as long as $\beta>0$. Note that this would not be the case if $\beta$ was chosen to be equal to $0$. }
\end{description}  
  
\noindent
To sum up, taking $\beta \in (0,1/2)$ will be a crucial  assumption to derive the almost sure convergence rates that are stated in Theorem \ref{theo:asrates}.


\subsection{Numerical experiments} \label{sec:num}
 
In this section, we report  numerical results on the performances of stochastic algorithms for regularized optimal transport when the source measure $\mu$ is either discrete or absolutely continuous. We  shall compare the SGD, ADAM, SGN and SN algorithms. For the SGD algorithm,  following the results in \cite{Stochastic_Bigot_Bercu},  we took $\alpha = 1/2$ and  $S_n = s^{-1} n \Id$ with $s = \varepsilon/(2\min(\nu))$. The ADAM algorithm has been implemented following the parametrization made in the seminal paper \cite{KingmaB14} except the value of the stepsize (as defined in   \cite[Algorithm 1]{KingmaB14}) that is set to $0.005$ instead of $0.001$, which improves the performances of ADAM in our numerical experiments.  For the SGN algorithm, we  set $\alpha = 0$ and we have taken  $\gamma = 10^{-3}$ (a small value) and $\beta = 0.49$.  \tcb{For the results reported in this paper, we have found that  the performances of the SGN algorithm are not very sensitive to the value of $\beta  \in (0,1/2)$}. Finally, for the SN algorithm, we chose $\alpha = 0$ and $S_n =  \mathbb{S}_n $ as defined by \eqref{eq:SnSN}. 

\tcb{In the discrete setting, we shall also compare the performances of these stochastic algorithms to those of the Sinkhorn algorithm \cite{cuturi}, which is a deterministic iterative procedure that uses the full knowledge of the measures $\mu$ and $\nu$ at each iteration. Let us recall that, for the SGD and ADAM algorithms,  the computational cost of one iteration from $n$ to $n+1$ is of order $\mathcal{O}(J)$, while it is of order $\mathcal{O}(J^2)$ for the SGN algorithm and $\mathcal{O}(J^3)$ for the SN algorithm. Each iteration of the Sinkhorn algorithm is of order   $\mathcal{O}(IJ)$, where $I$ denotes the size of the support of $\mu$ in the discrete setting.}

 \tcb{ In these numerical experiments, we investigate the numerical behavior of the recursive estimators $\wh{W}_n$ and  $\wh{V}_n$. The performances of the various stochastic algorithms used to compute these estimators are compared in terms of the expected excess risks $ \E  \bigl[ \bigl|  \wh{W}_n  -W_{\varepsilon}(\mu,\nu)\bigr| \bigr]$ and $ \E   \bigl[  \bigl\| \wh{V}_n -v^\ast \bigr\|^2  \bigr]$. For the SGN algorithm, we also analyze the convergence of the estimator $\overline{S}_n$ to the matrix $G_{\varepsilon}(v^\ast)$. The expected value involved in these expected risks is approximated using $100$ Monte-Carlo replications. When the measure $\mu$ is discrete, we use the Sinkhorn algorithm \cite{cuturi} to preliminary compute $W_{\varepsilon}(\mu,\nu)$ and $v^\ast$. When $\mu$ is absolutely continuous,  the regularized OT cost  $W_{\varepsilon}(\mu,\nu)$ is preliminary approximated by running the SN algorithm with a very large value of  iterations (e.g.\ $n = 10^6$). To the best of our knowledge, apart from stochastic approaches as in \cite{NIPS2016_6566}, 
there is no other method to evaluate $W_{\varepsilon}(\mu,\nu)$ in the semi-discrete setting.  Note that  we shall compare the evolution of these  excess risks  as a function of the computational time (observed on the computer) of each algorithm. Moreover, the estimators $\wh{W}_n$ and $\wh{V}_n$ obviously depends on   the regularization parameter $\varepsilon$. However, for the sake of simplicity, we have chosen  to denote them as $\wh{W}_n$ and $\wh{V}_n$, although we carry out  numerical experiments for different values of $\varepsilon$. Finally, we also  analyze  the asymptotic distributions of $\wh{W}_n$ and  $\wh{V}_n$ to illustrate the results on asymptotic normality given in Section \ref{sec:TCL}.
}

\subsubsection{Discrete setting in dimension $d=2$}

\CB{In this section, the cost function is chosen as the squared Euclidean distance that is $c(x,y) = \|x-y\|^2$.
We  focus our attention when both $\mu = \sum_{i=1}^{I} \mu_i \delta_{x_i}$ and $\nu  = \sum_{j=1}^{J} \nu_j \delta_{y_j}$ are uniform discrete measures supported on $\R^2$, that is $\mu_i = \frac{1}{I}$ and $ \nu_j= \frac{1}{J}$. The points $(x_i)_{1 \leq i \leq I}$   (resp.\ $(y_j)_{1 \leq j \leq J}$) are drawn randomly (once for all) from a Gaussian mixture with two components (resp.\ from the uniform distribution on $[0,1]^2$).    An example of two such measures is displayed in Figure \ref{fig:2D_measures} for $I=10^4$ and $J=100$. The number of iterations of the four stochastic algorithms is fixed to $n= 10^5$ except for the experiments on the asymptotic distribution of $\wh{W}_n$ and  $\wh{V}_n$,  where $n$ is let being larger. Finally, the Sinkhorn algorithm  is let running until convergence is reached to provide a reference value for  $W_{\varepsilon}(\mu,\nu)$ and $v^\ast$ considered as the ground truth.

\begin{figure}[htbp]
\centering
{\subfigure[]{\includegraphics[width=0.32 \textwidth,height=0.35\textwidth]{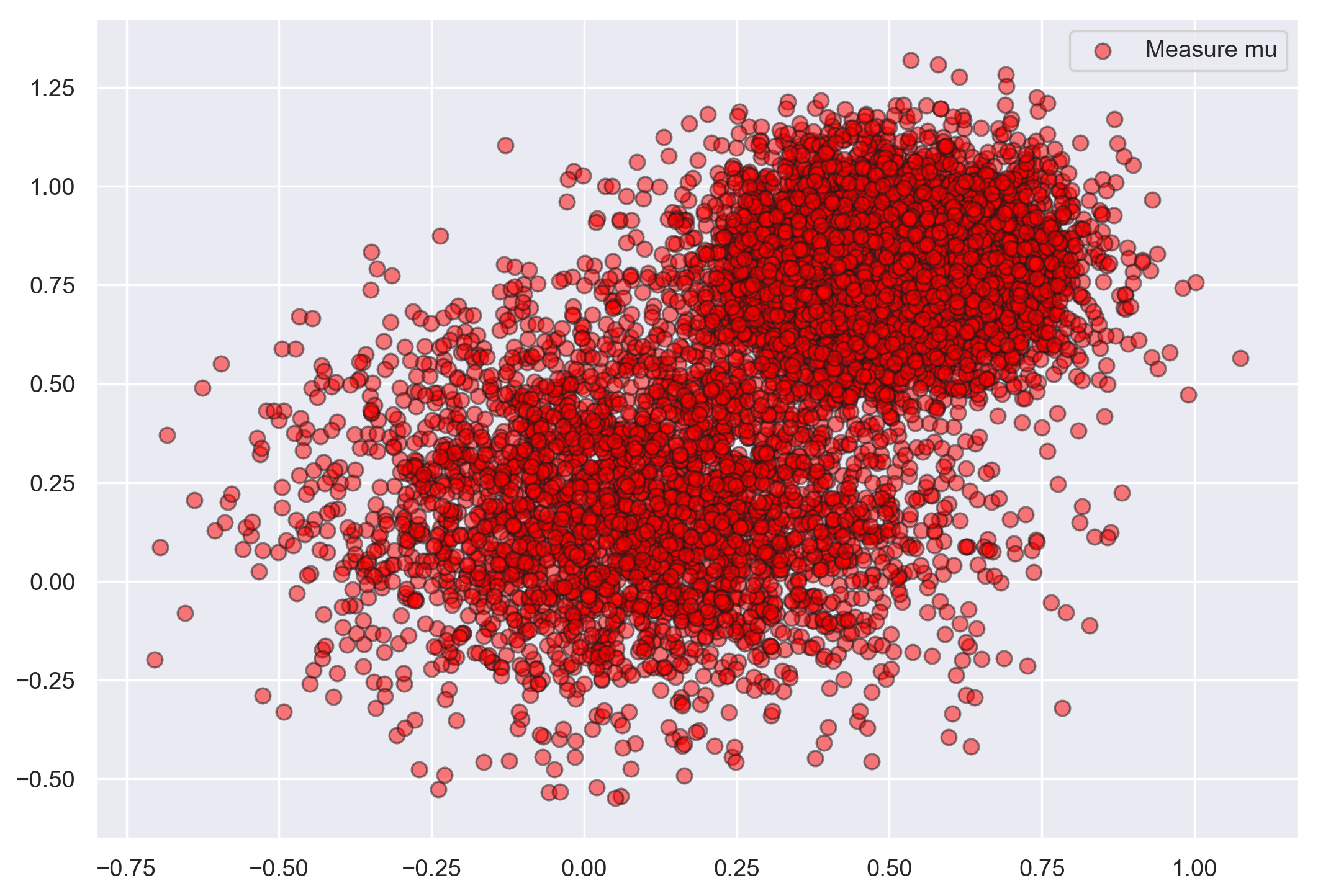}}}
{\subfigure[]{\includegraphics[width=0.32 \textwidth,height=0.35\textwidth]{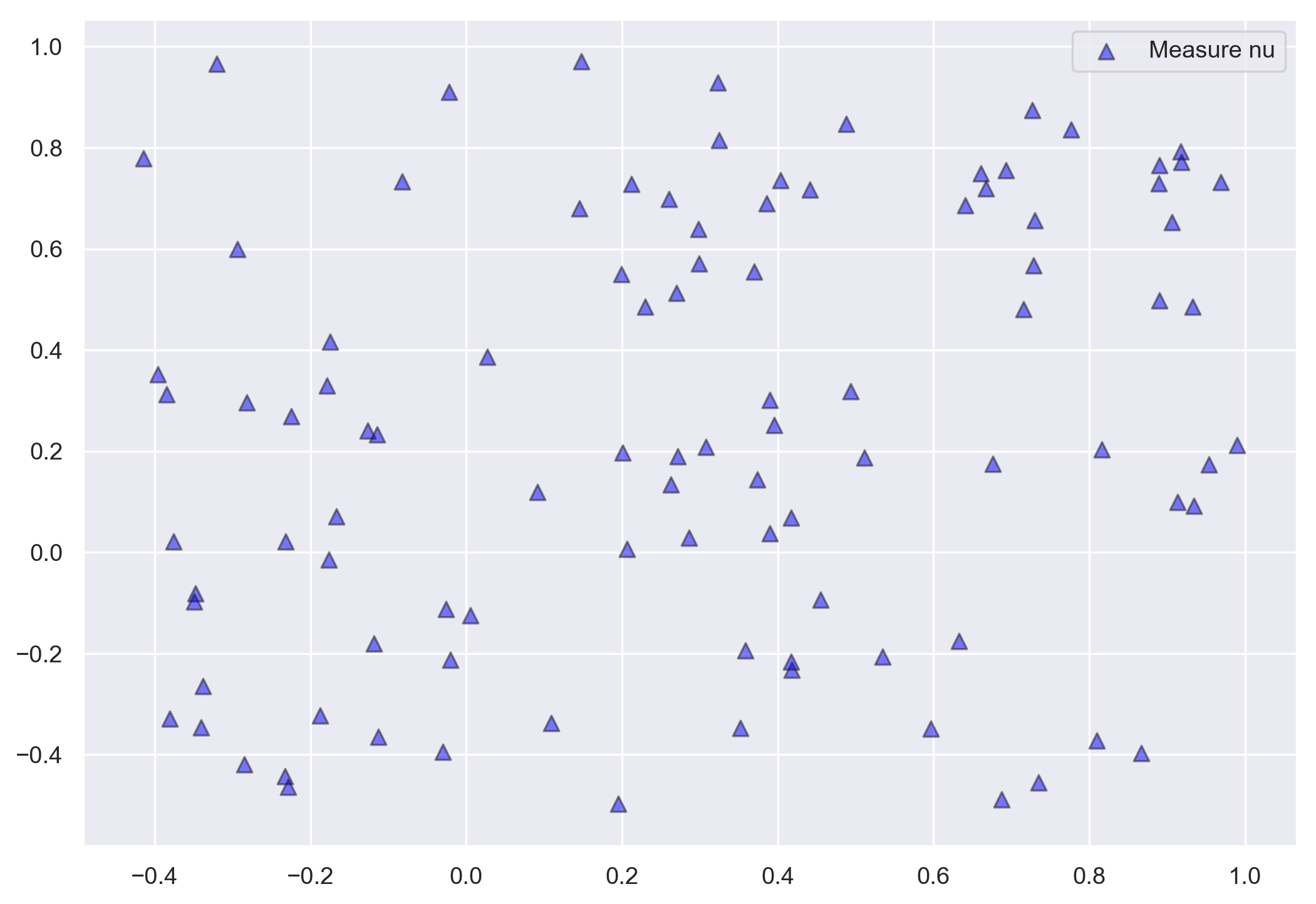}}}
{\subfigure[]{\includegraphics[width=0.32 \textwidth,height=0.35\textwidth]{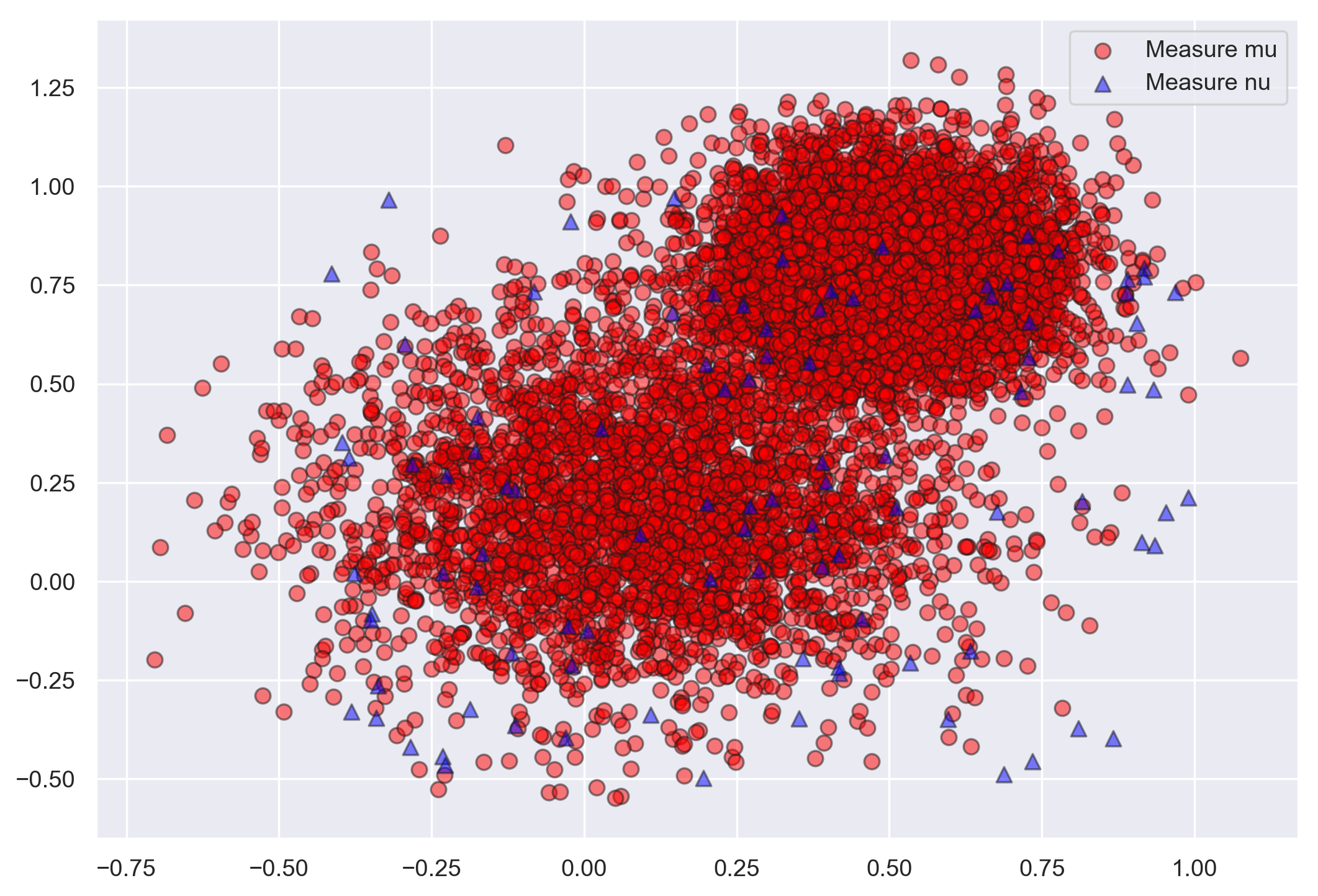}}}

\caption{(a) Discrete measure $\mu$ supported on $I = 10^4$ points drawn from a mixture of two Gaussian distributions, and (b) discrete measure $\nu$ supported on $J=100$ points randomly drawn from the uniform distribution on $[0,1]^2$. (c) Superposition of $\mu$ and $\nu$.  \label{fig:2D_measures}}
\end{figure}

\paragraph{Convergence of the excess risks.}

We first report results for $I=10^4$  (size of the support of $\mu$) and $J \in \{100,400\}$ (size of the support of $\nu$), and two small values of the regularization parameter $\varepsilon \in \{0.01, 0.005\}$. For different combinations of these hyperparameters, we display from Figure \ref{fig:excess_risk_d2_Wn_I_10p4_J100} to Figure \ref{fig:excess_risk_d2_Vn_I_10p4_J400} the value of the expected excess risks (in logarithmic scale) $  \E  \bigl[ \bigl|  \wh{W}_n  -W_{\varepsilon}(\mu,\nu)\bigr| \bigr]$ and $ \E   \bigl[  \bigl\| \wh{V}_n -v^\ast \bigr\|^2  \bigr]$ as functions of the averaged (along the 100  Monte Carlo replications) computational time of each iteration of the stochastic algorithms. We also draw the evolution of the metrics  (in logarithmic scale)  $\bigl|  W_k  -W_{\varepsilon}(\mu,\nu)\bigr| $ and $ \bigl\| V_k -v^\ast \bigr\|^2 $  as functions of the computational time of the iterations of the Sinkhorn algorithm, where $ W_k  \in \mathbb{R} $ and $V_k \in \mathbb{R}^J$ are the output of  the Sinkhorn algorithm at its $k$-th iteration. 

In Figure \ref{fig:excess_risk_d2_Wn_I_10p4_J100} to Figure \ref{fig:excess_risk_d2_Vn_I_10p4_J400}, the various curves are displayed as  functions of the computational time until the convergence  of the Sinkhorn algorithm is reached, that is until $k= k_{\max}$ (the maximum number of Sinkhorn iterations). Note that for $k \approx k_{\max}$  then $\bigl|  W_k  -W_{\varepsilon}(\mu,\nu)\bigr| \approx 0$ and $ \bigl\| V_k -v^\ast \bigr\|^2 \approx 0$. Hence, for such large values of $k$, these metrics have necessarily smaller values than those that are used to evaluate the stochastic algorithms. In the discussion that follows, we thus consider that the stochastic algorithms have reached convergence when the values of  either $  \E  \bigl[ \bigl|  \wh{W}_n  -W_{\varepsilon}(\mu,\nu)\bigr| \bigr]$ or $ \E   \bigl[  \bigl\| \wh{V}_n -v^\ast \bigr\|^2  \bigr]$ stabilize, although these metrics may be larger than the metrics used to evaluate the Sinkhorn algorithm for large values of the computational time. This is due to the randomness of the stochastic algorithms and their resulting positive variance (even for large values of $n$). Then, the following comments can be made from the output of these numerical experiments.

\begin{description}
\item[-] For $I = 10^4$, $J  = 100$ and $\varepsilon = 0.01$, the four stochastic algorithms reach convergence faster than the Sinkhorn algorithm. The convergence is much faster for  the metric  $ \E   \bigl[  \bigl\| \wh{V}_n -v^\ast \bigr\|^2  \bigr]$ than for the metric  $  \E  \bigl[ \bigl|  \wh{W}_n  -W_{\varepsilon}(\mu,\nu)\bigr| \bigr]$.

\item[-] For $I = 10^4$, $J  = 100$ and $\varepsilon = 0.005$,  SGD  fails to converge either for  the estimator  $ \wh{W}_n$ or    the estimator $ \wh{V}_n$. For this smallest value of $\epsilon$,  SGN  and SN have similar performances for the metric   $ \E   \bigl[  \bigl\| \wh{V}_n -v^\ast \bigr\|^2  \bigr]$. The SN algorithm is slightly better than SGN for the metric  $  \E  \bigl[ \bigl|  \wh{W}_n  -W_{\varepsilon}(\mu,\nu)\bigr| \bigr]$. We also observe that SN and SGN converge much faster than Sinkhorn, and that they have better performances than ADAM for the two metrics. 

\item[-]  For $I=10^4$, $J=400$ and  $\varepsilon \in \{0.01, 0.005\}$, it can be seen that the SGD algorithm does not converge. For the metric  $ \E   \bigl[  \bigl\| \wh{V}_n -v^\ast \bigr\|^2  \bigr]$, the convergence of the SN and SGN algorithms is much faster than Sinkhorn, and these two algorithms outperform ADAM.

\end{description}
 
Therefore, these numerical experiments suggest that the SGN algorithm has  interesting benefits over the SGD, ADAM and Sinkhorn algorithms  for moderate values of $J$ and for small values of the regularization parameter $\varepsilon$. In these settings, SGN seems to be particularly relevant  for the estimation of $v^\ast$, and it reaches performances similar to those of SN for the metric  $ \E   \bigl[  \bigl\| \wh{V}_n -v^\ast \bigr\|^2  \bigr]$. The SGN algorithm may also converge much faster than the Sinkhorn algorithm for either the estimation of $W_{\varepsilon}(\mu,\nu)$ or  $v^\ast$  as the size $I$ of the support of $\mu$ is large.

\paragraph{Asymptotic distribution of the stochastic algorithms.} Now, we illustrate the results from Section \ref{sec:TCL} on the asymptotic distributions of $\wh{W}_n$ and  $\wh{V}_n$. To this end, we consider the setting $I=10^3$ and $J=50$. Then, we display in Figure \ref{fig:TCL_Wn_eps_0_1}  for $\varepsilon = 0.1$ and $n= 2 \times 10^5$ iterations (resp.\  Figure \ref{fig:TCL_Wn_eps_0_01}  for $\varepsilon = 0.01$ and $n= 4 \times 10^5$ iterations) the histograms of 200 independent realizations of $$\widetilde{W}_n = \frac{\sqrt{n}  \left( \wh{W}_n  -W_{\varepsilon}(\mu,\nu) \right)}{ \wh{\sigma}_n}$$   using each of  the four stochastic algorithms, where 
$$
\wh{\sigma}^{\,2}_n =  \frac{1}{n}\sum_{k=1}^n  h_{\varepsilon}^2(X_k, \wh{V}_{k-1}) - \wh{W}_n^{\,2},
$$
is a recursive estimator of the asymptotic variance of $ \wh{W}_n$ that has been introduced in \cite{Stochastic_Bigot_Bercu}. For all the algorithms, it can be seen in  Figure \ref{fig:TCL_Wn_eps_0_1} and  Figure \ref{fig:TCL_Wn_eps_0_01} that  $\wh{W}_n $ is  normally distributed. For the SGD and the SN algorithms, the histograms of $\widetilde{W}_n$ are very close to the standard Gaussian distribution, while the SGN is seen to be slightly biased. The bias is much more important for the ADAM algorithm.

\begin{figure}[htbp]
\centering
{\subfigure[$\varepsilon = 0.01$]{\includegraphics[width=0.45 \textwidth,height=0.35\textwidth]{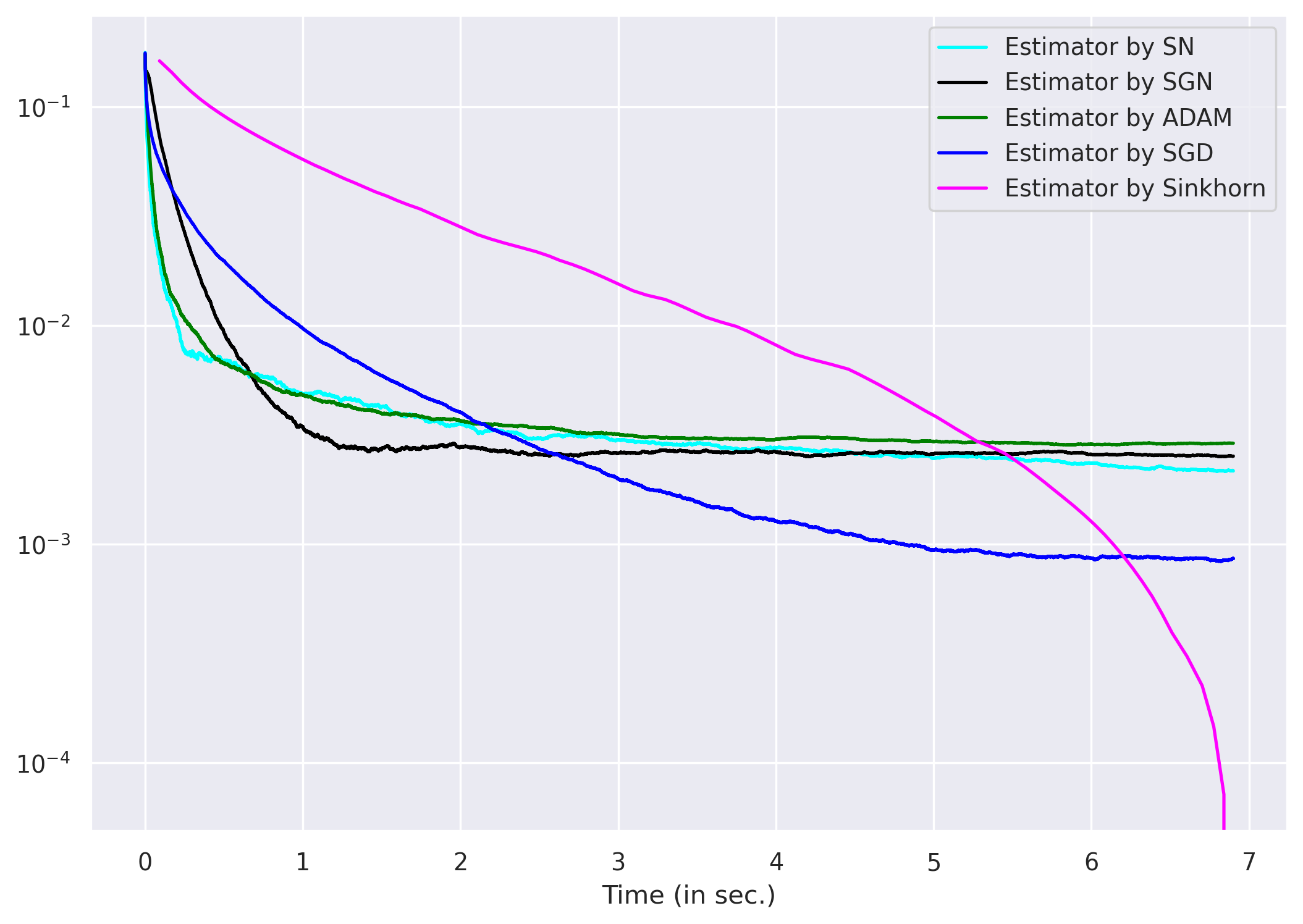}}}
{\subfigure[$\varepsilon = 0.005$]{\includegraphics[width=0.45 \textwidth,height=0.35\textwidth]{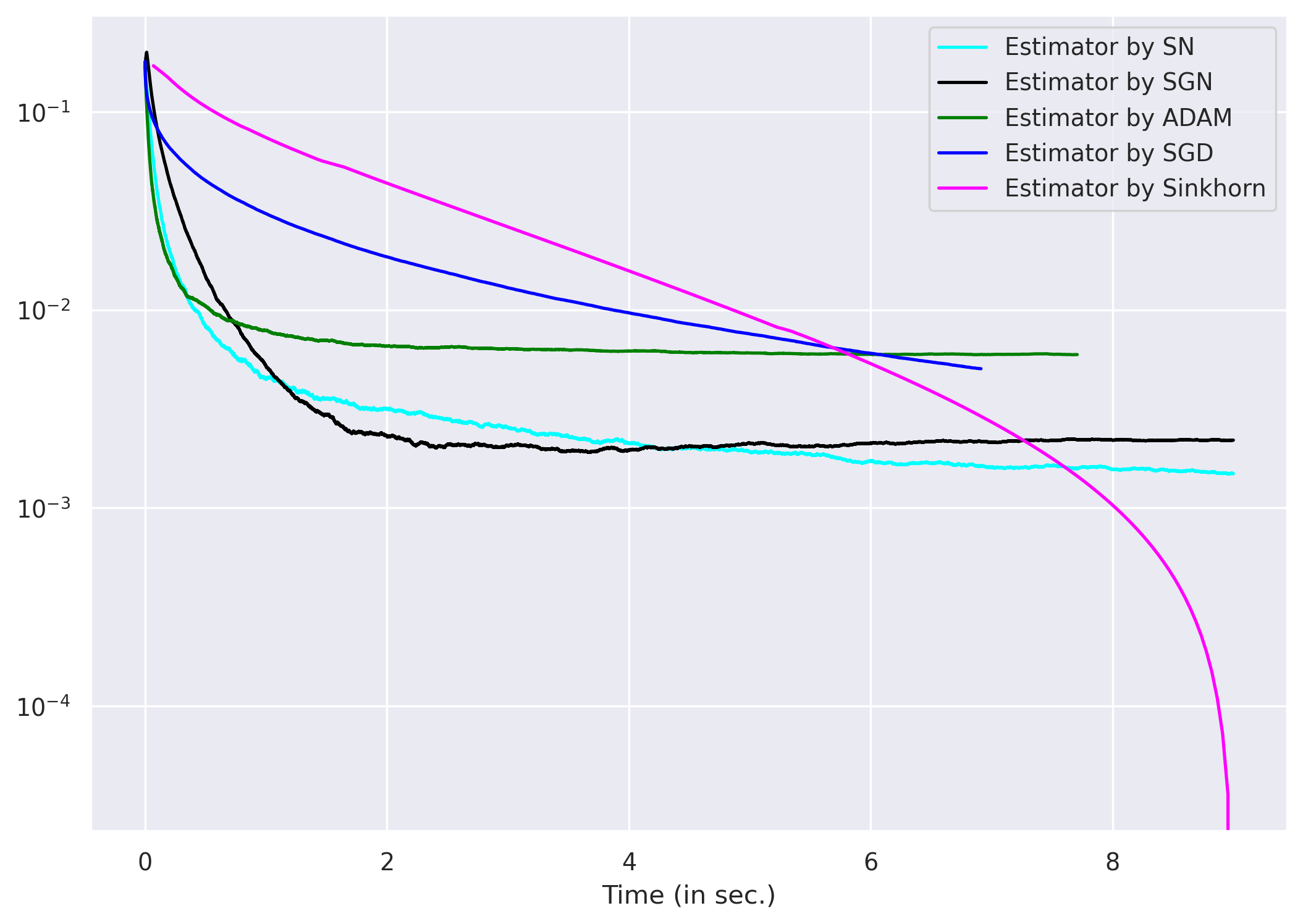}}}

\caption{Discrete setting with $I=10^4$ and $J=100$  with $n= 10^5$ iterations. Expected excess risk  (in logarithmic scale)  $\log( \E  \bigl[ \bigl|  \wh{W}_n  -W_{\varepsilon}(\mu,\nu)\bigr| \bigr])$   (resp.\ metric $\log( \bigl|  W_k  -W_{\varepsilon}(\mu,\nu)\bigr|) $) as a function of the averaged computational cost of the iterations of  the four stochastic algorithms (resp.\ the Sinkhorn algorithm) for different values of the regularization parameter $\varepsilon$.  \label{fig:excess_risk_d2_Wn_I_10p4_J100}}
\end{figure}

\begin{figure}[htbp]
\centering
{\subfigure[$\varepsilon = 0.01$]{\includegraphics[width=0.45 \textwidth,height=0.35\textwidth]{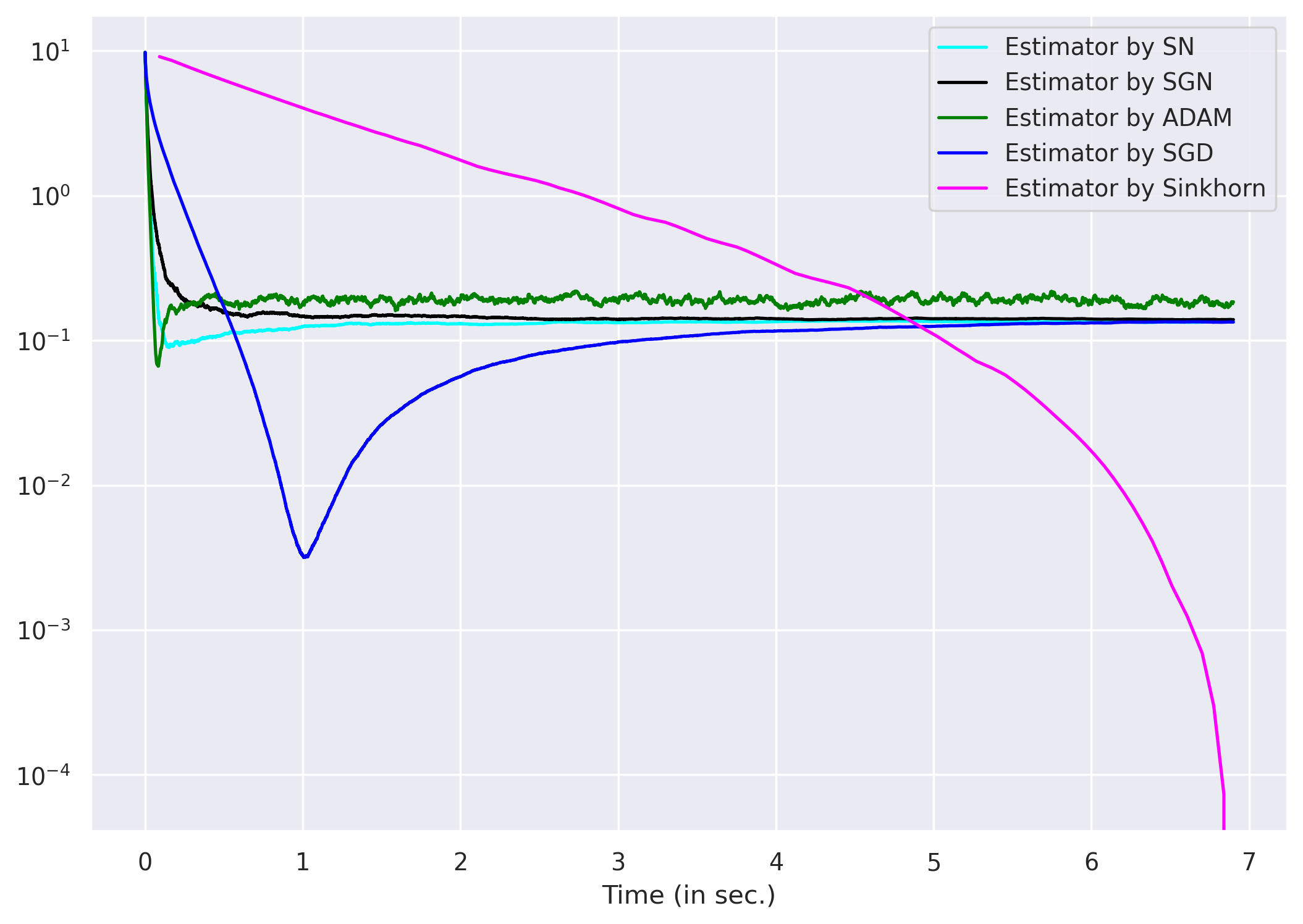}}}
{\subfigure[$\varepsilon = 0.005$]{\includegraphics[width=0.45 \textwidth,height=0.35\textwidth]{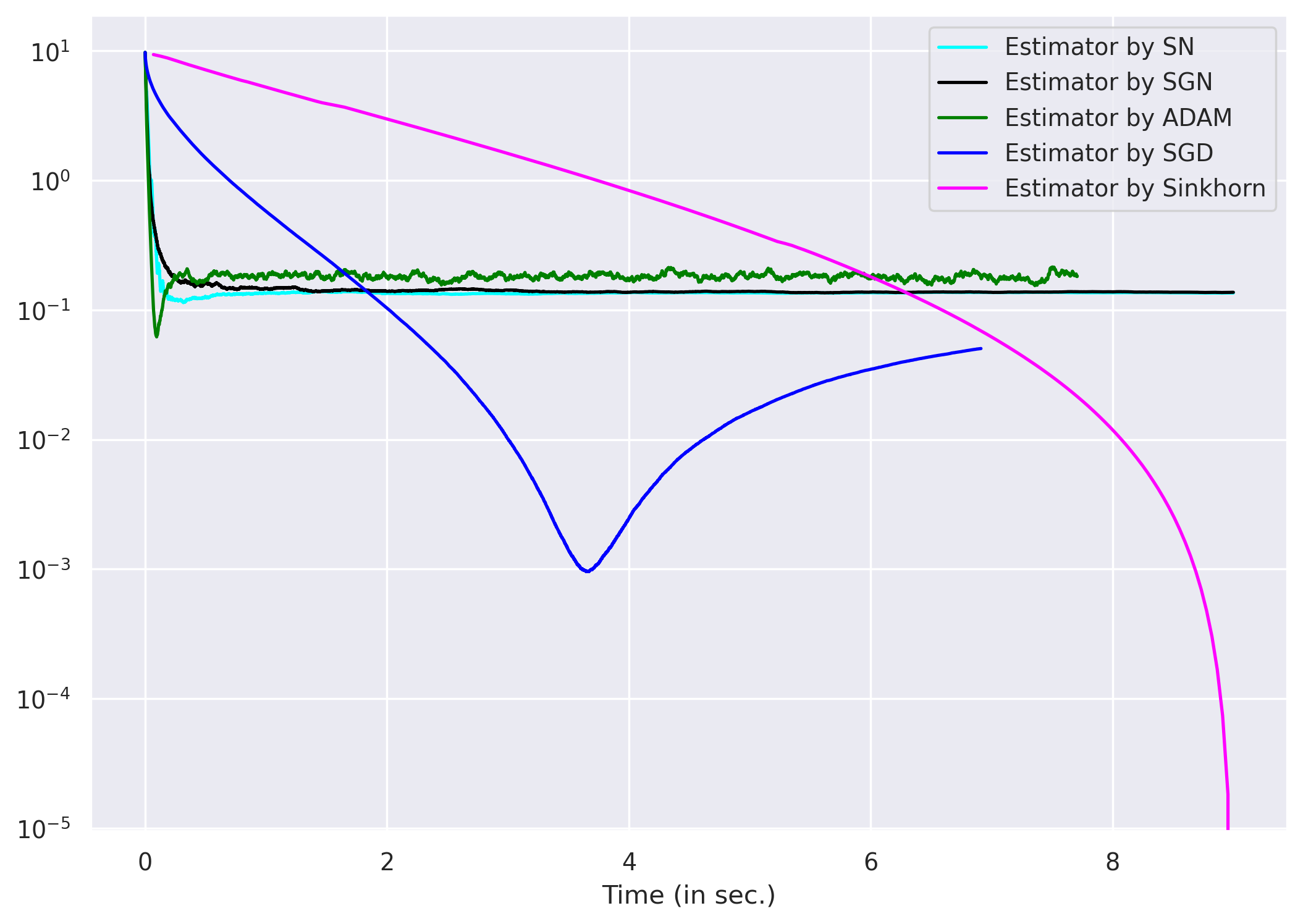}}}

\caption{Discrete setting with $I=10^4$ and $J=100$ with $n= 10^5$ iterations. Expected excess risk  (in logarithmic scale)  $ \log(\E   \bigl[  \bigl\| \wh{V}_n -v^\ast \bigr\|^2  \bigr])$  (resp.\ metric $ \log( \bigl\|  V_k  -v^\ast \bigr\|^2)$) as a function of the averaged computational cost of the iterations of  the four stochastic algorithms (resp.\ the Sinkhorn algorithm) for different values of the regularization parameter $\varepsilon$.  \label{fig:excess_risk_d2_Vn_I_10p4_J100}}
\end{figure}

\begin{figure}[htbp]
\centering
{\subfigure[$\varepsilon = 0.01$]{\includegraphics[width=0.45 \textwidth,height=0.35\textwidth]{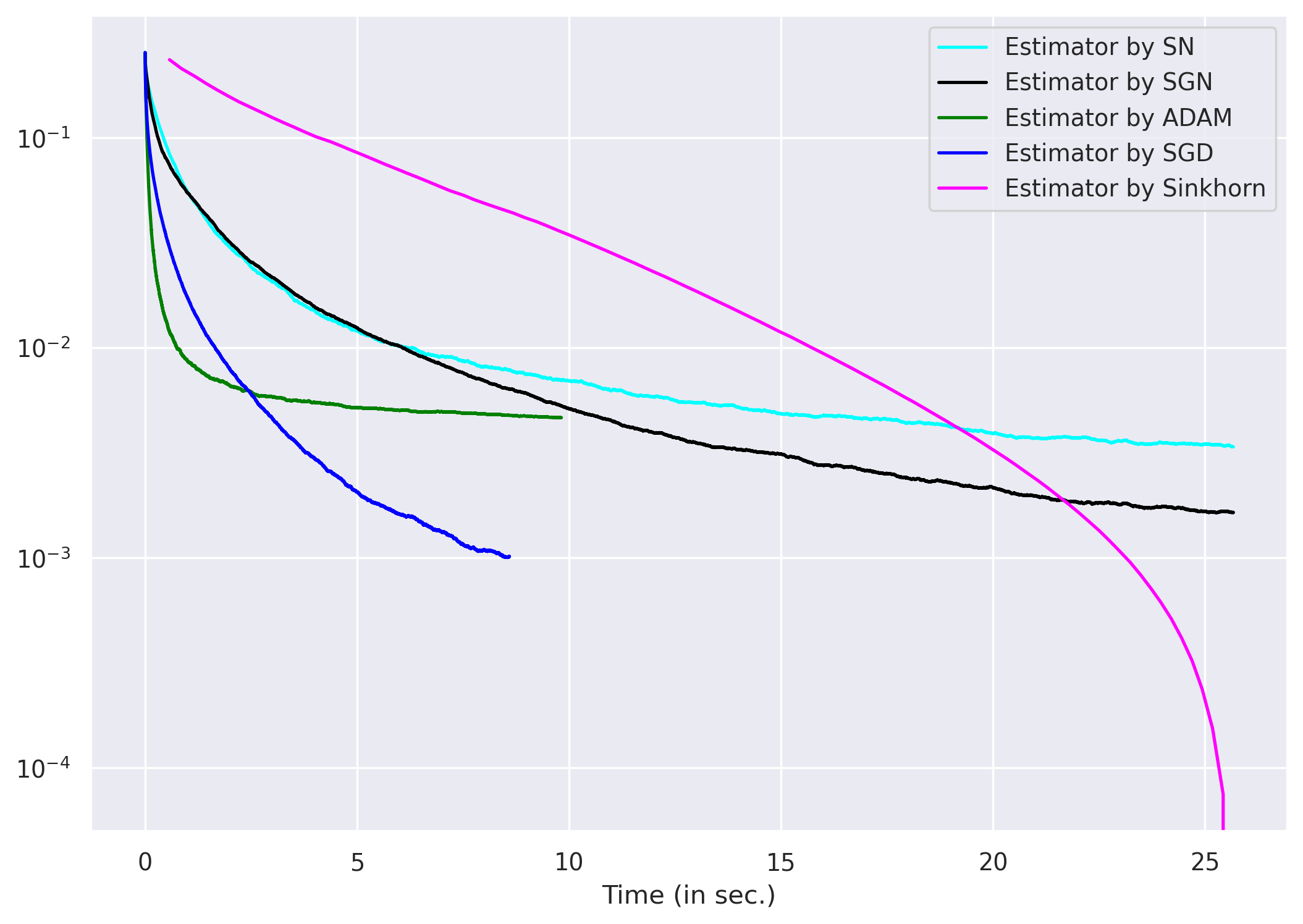}}}
{\subfigure[$\varepsilon = 0.005$]{\includegraphics[width=0.45 \textwidth,height=0.35\textwidth]{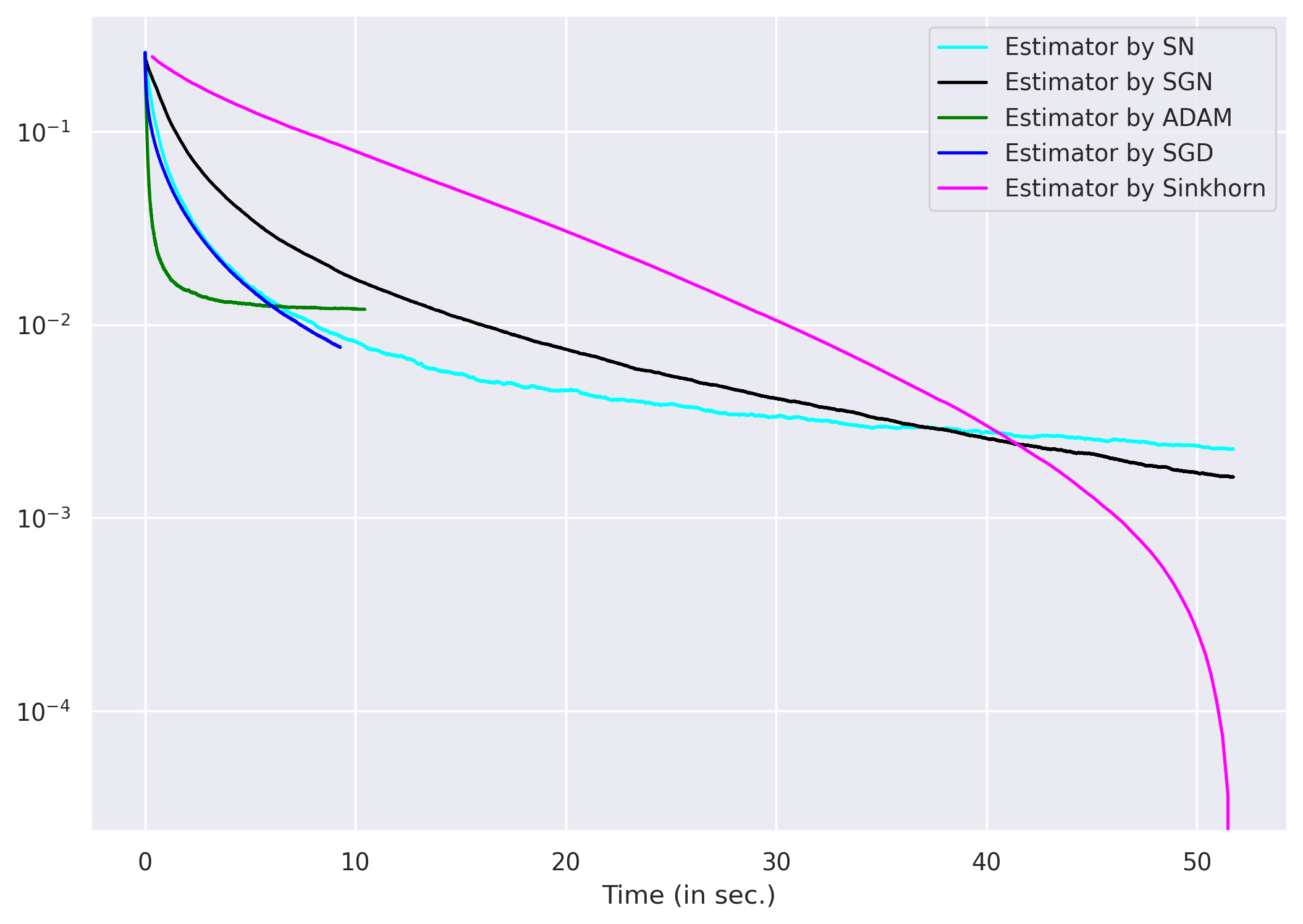}}}

\caption{Discrete setting with $I=10^4$ and $J=400$  with $n= 10^5$ iterations. Expected excess risk  (in logarithmic scale)  $\log( \E  \bigl[ \bigl|  \wh{W}_n  -W_{\varepsilon}(\mu,\nu)\bigr| \bigr])$   (resp.\ metric $\log( \bigl|  W_k  -W_{\varepsilon}(\mu,\nu)\bigr|) $) as a function of the averaged computational cost of the iterations of  the four stochastic algorithms (resp.\ the Sinkhorn algorithm) for different values of the regularization parameter $\varepsilon$.  \label{fig:excess_risk_d2_Wn_I_10p4_J400}}
\end{figure}

\begin{figure}[htbp]
\centering
{\subfigure[$\varepsilon = 0.01$]{\includegraphics[width=0.45 \textwidth,height=0.35\textwidth]{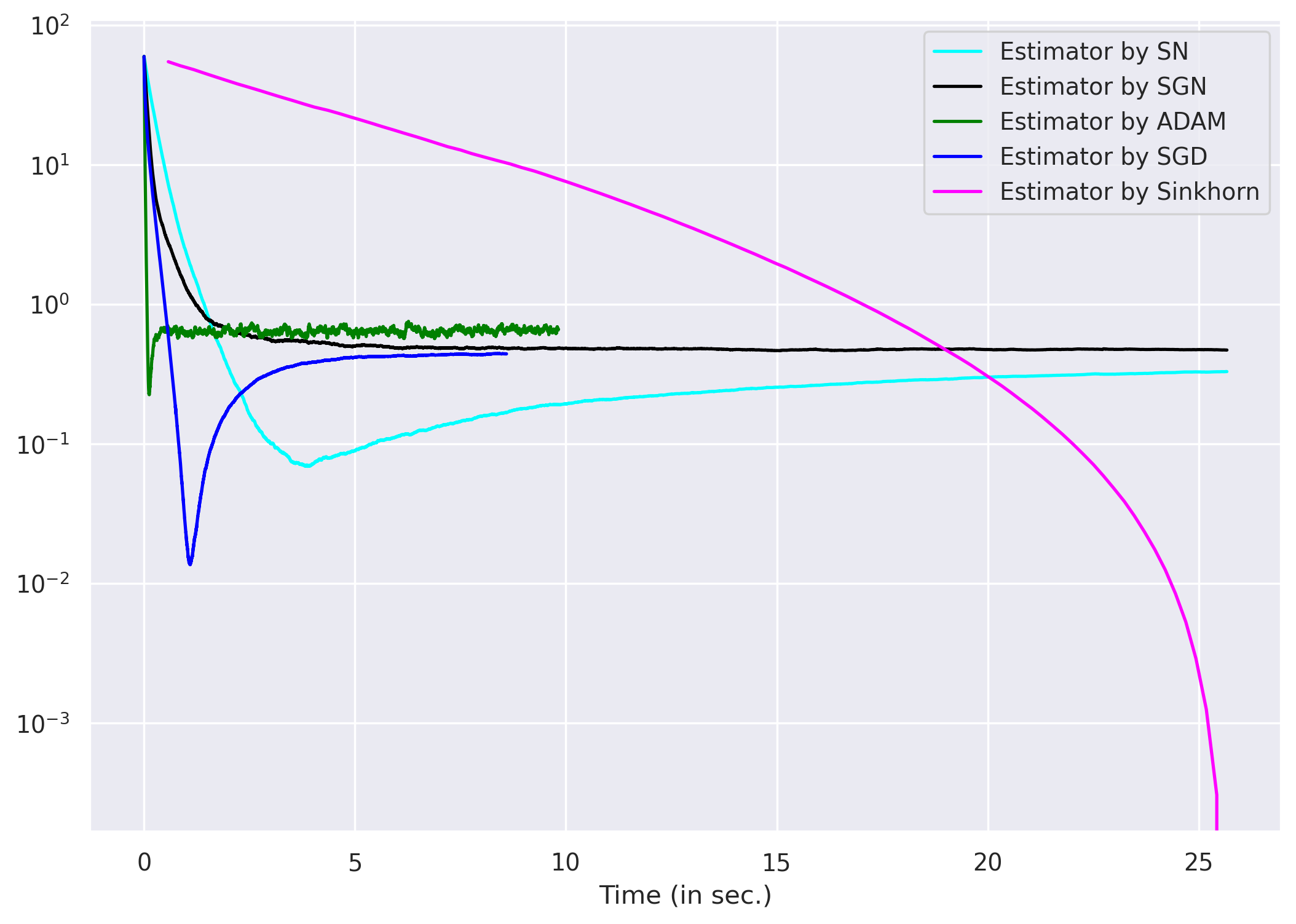}}}
{\subfigure[$\varepsilon = 0.005$]{\includegraphics[width=0.45 \textwidth,height=0.35\textwidth]{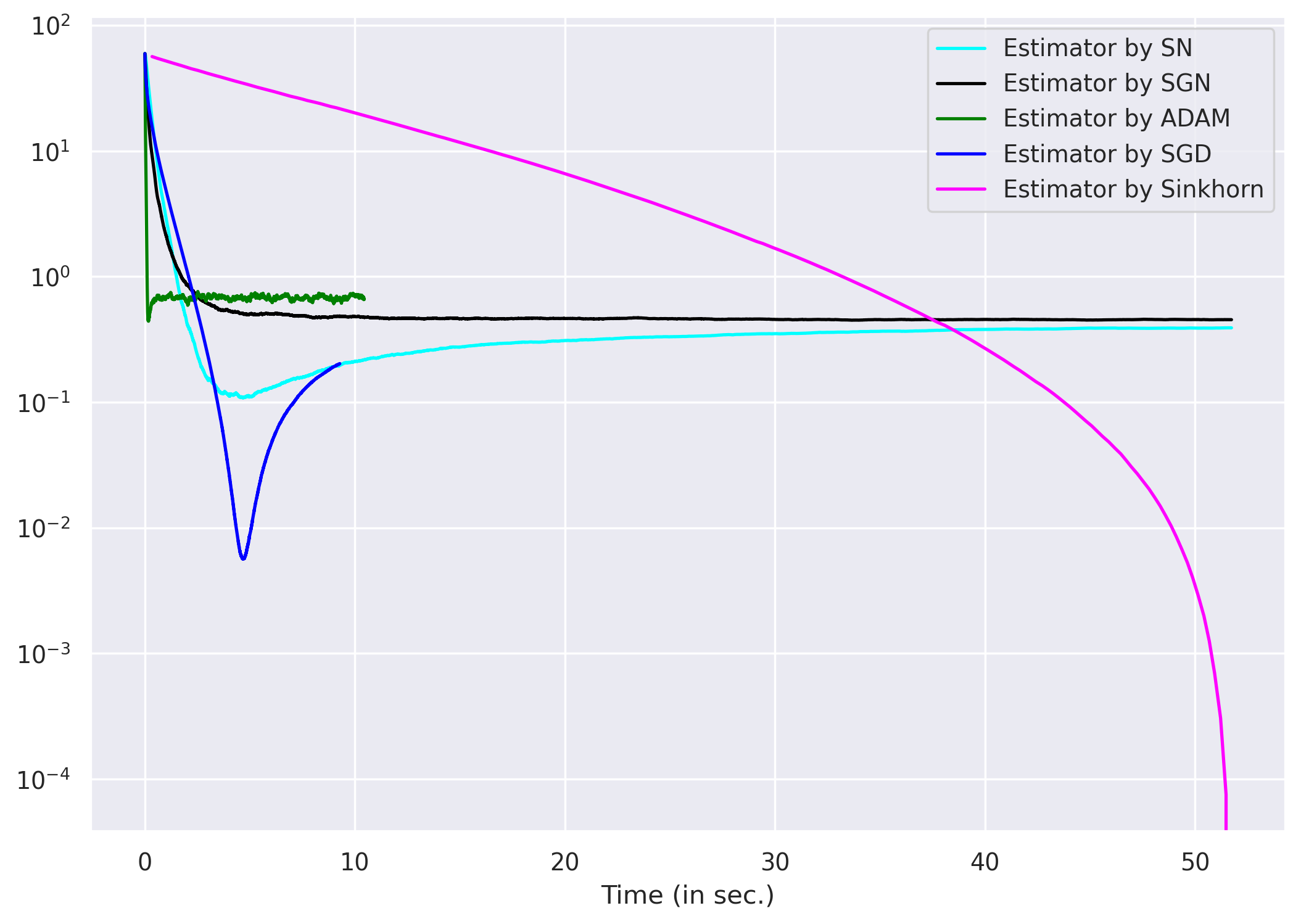}}}

\caption{Discrete setting with $I=10^4$ and $J=400$  with $n= 10^5$ iterations.  Expected excess risk  (in logarithmic scale)  $ \log(\E   \bigl[  \bigl\| \wh{V}_n -v^\ast \bigr\|^2  \bigr])$  (resp.\ metric $ \log( \bigl\|  V_k  -v^\ast \bigr\|^2)$) as a function of the averaged computational cost of the iterations of  the four stochastic algorithms (resp.\ the Sinkhorn algorithm) for different values of the regularization parameter $\varepsilon$.  \label{fig:excess_risk_d2_Vn_I_10p4_J400}}
\end{figure}

}

\CB{


\begin{figure}[htbp]
\centering
{\subfigure[SGD - $\varepsilon = 0.1$]{\includegraphics[width=0.45 \textwidth,height=0.23\textwidth]{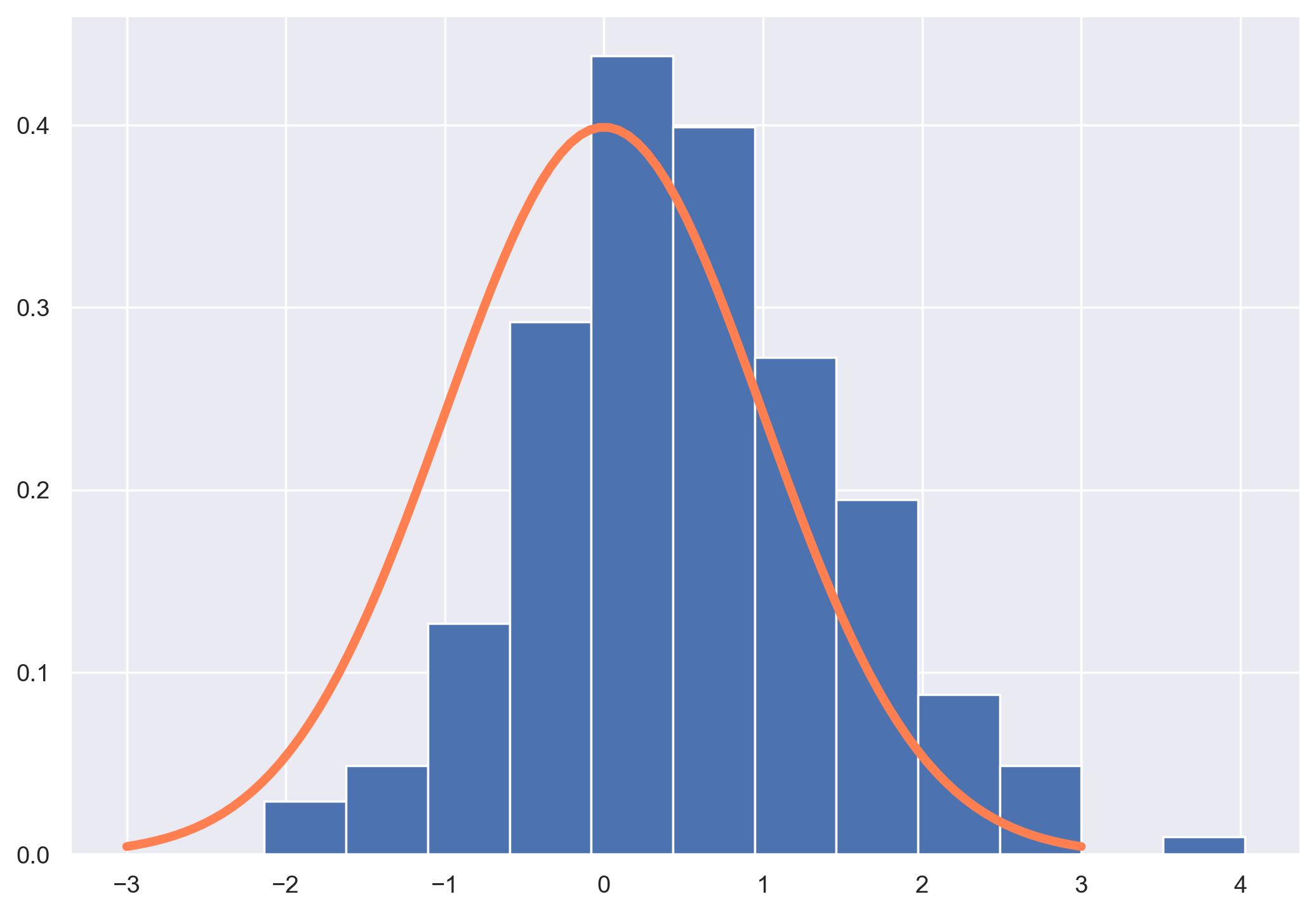}}}
{\subfigure[SGN - $\varepsilon = 0.1$]{\includegraphics[width=0.45 \textwidth,height=0.23\textwidth]{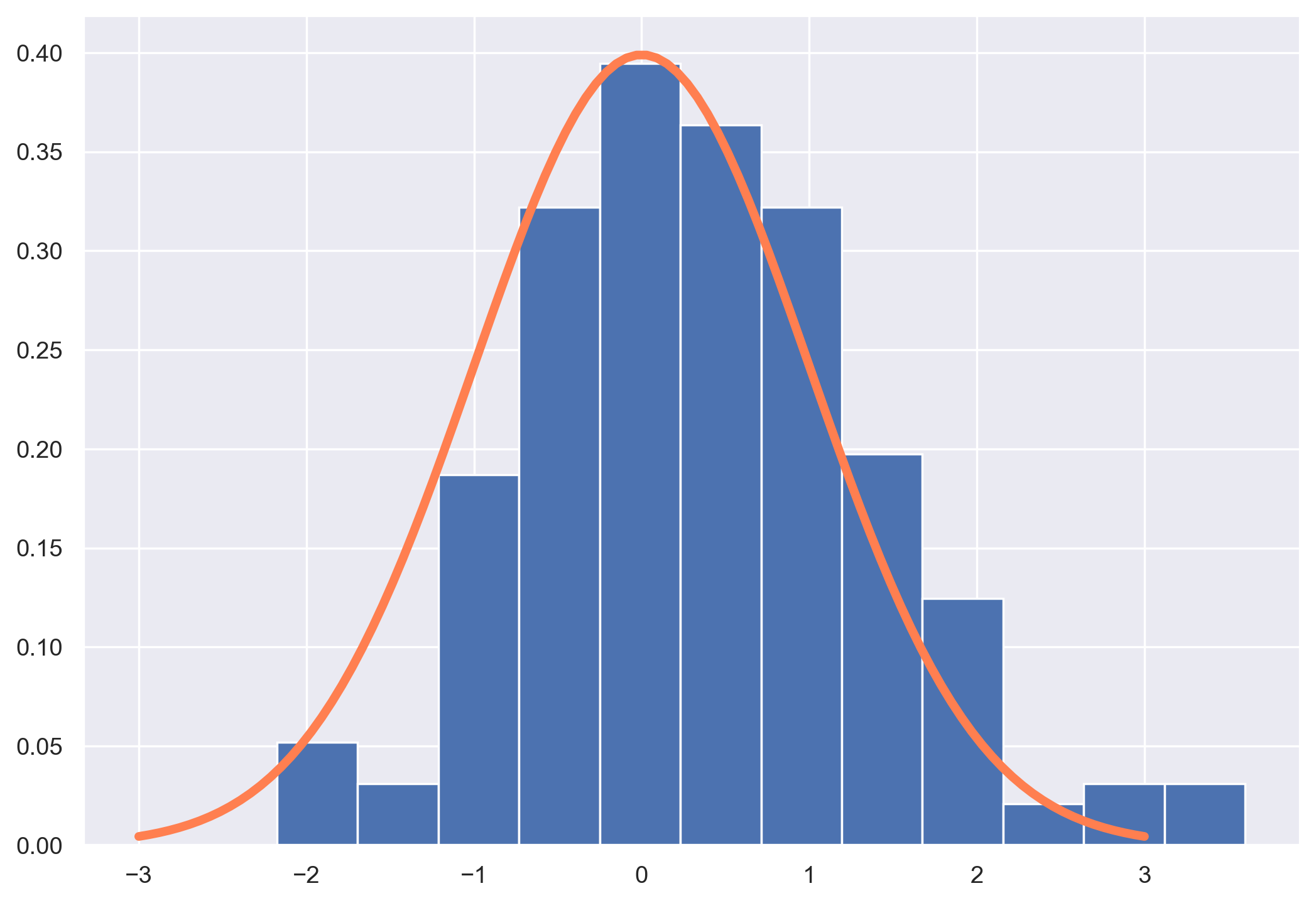}}}
{\subfigure[SN - $\varepsilon = 0.1$]{\includegraphics[width=0.45 \textwidth,height=0.23\textwidth]{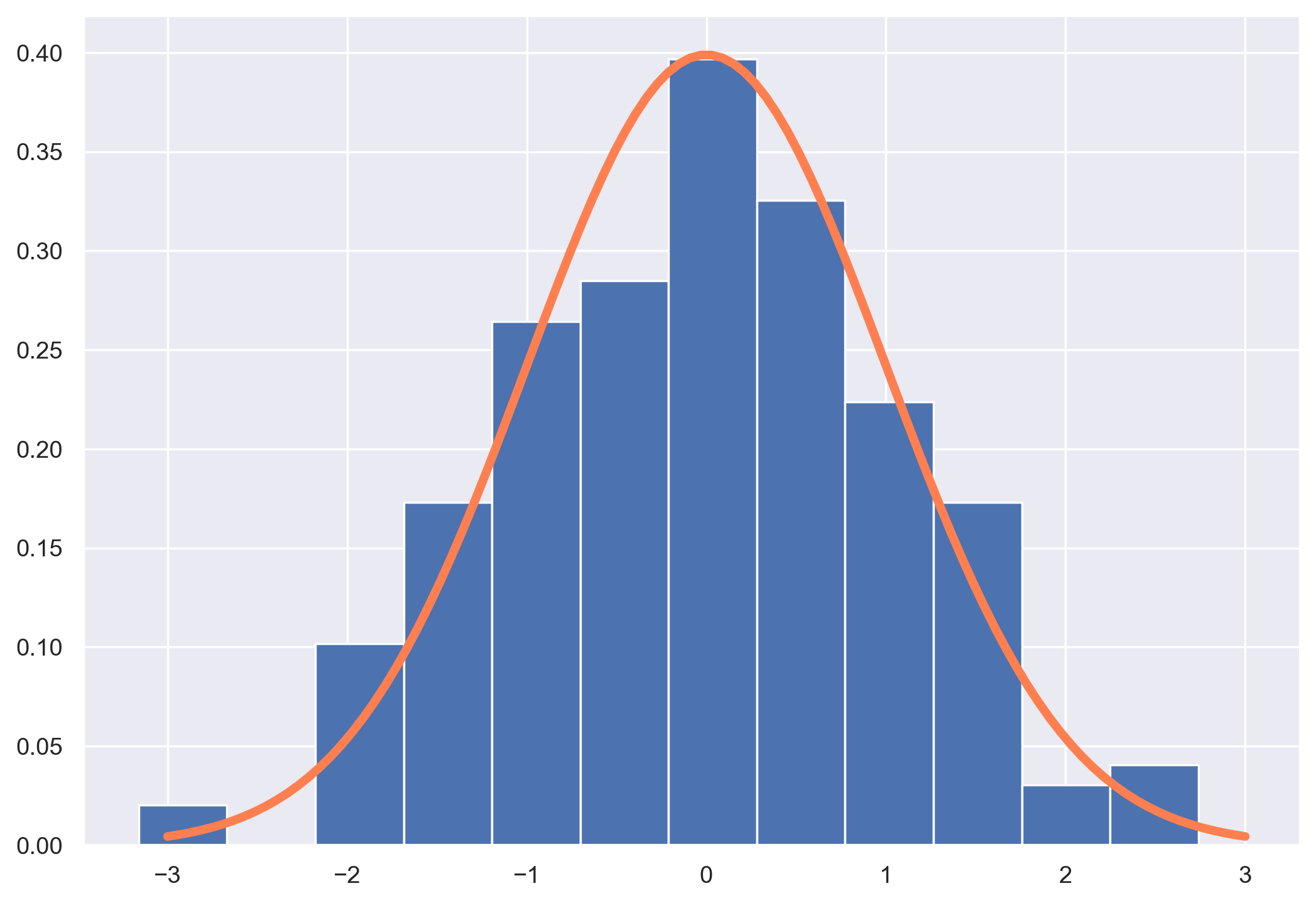}}}
{\subfigure[ADAM - $\varepsilon = 0.1$]{\includegraphics[width=0.45 \textwidth,height=0.23\textwidth]{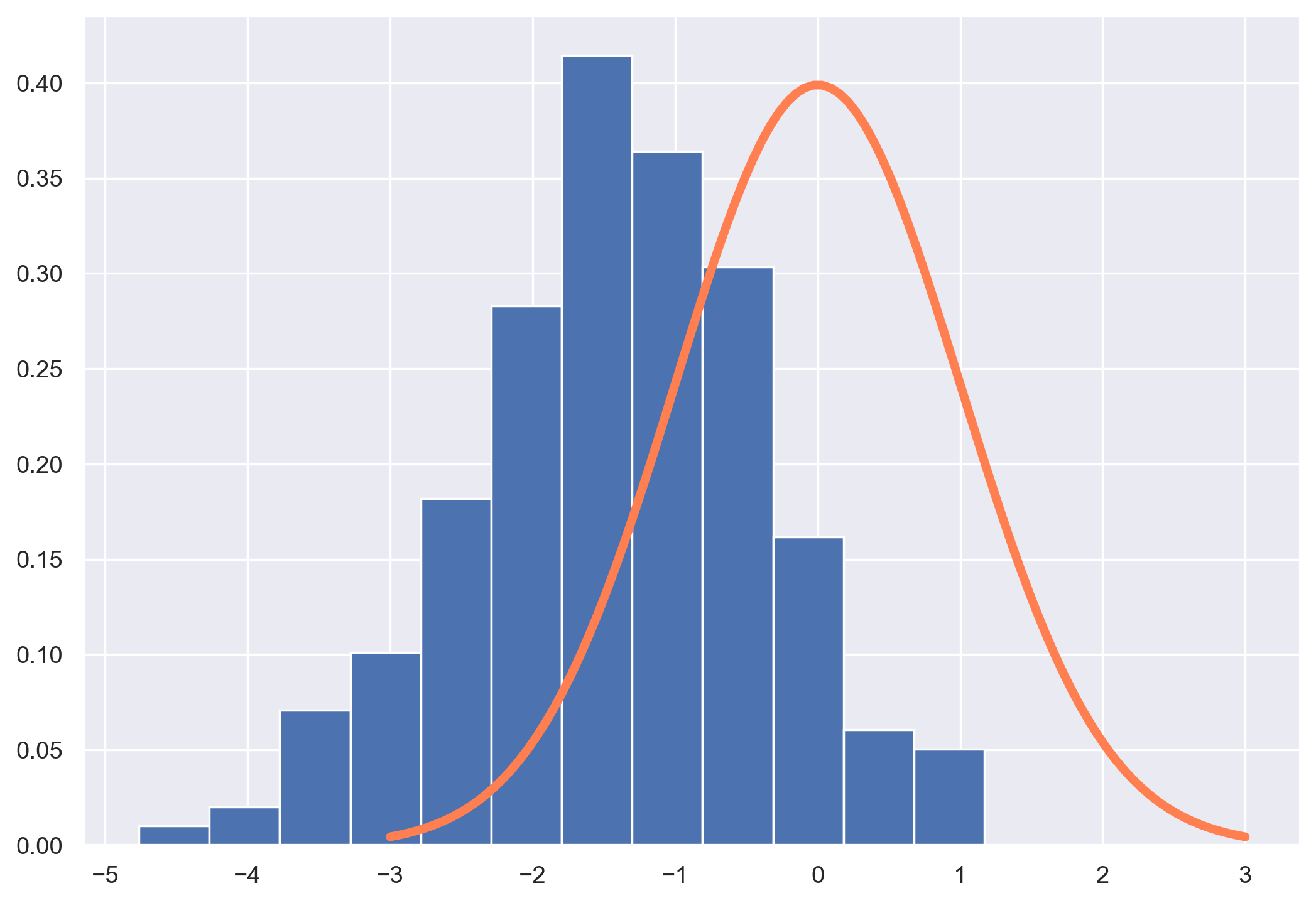}}}

\caption{Discrete setting with $I=10^3$, $J=50$ and $\varepsilon = 0.1$. Histogram of 200 independent realizations of $\frac{\sqrt{n}  \left( \wh{W}_n  -W_{\varepsilon}(\mu,\nu) \right)}{ \wh{\sigma}_n}$  with $n= 2 \times 10^5$  using each of  the four stochastic algorithms. The orange curve is the density of the standard Gaussian distribution.  \label{fig:TCL_Wn_eps_0_1}}
\end{figure}

\begin{figure}[htbp]
\centering
{\subfigure[SGD - $\varepsilon = 0.01$]{\includegraphics[width=0.45 \textwidth,height=0.23\textwidth]{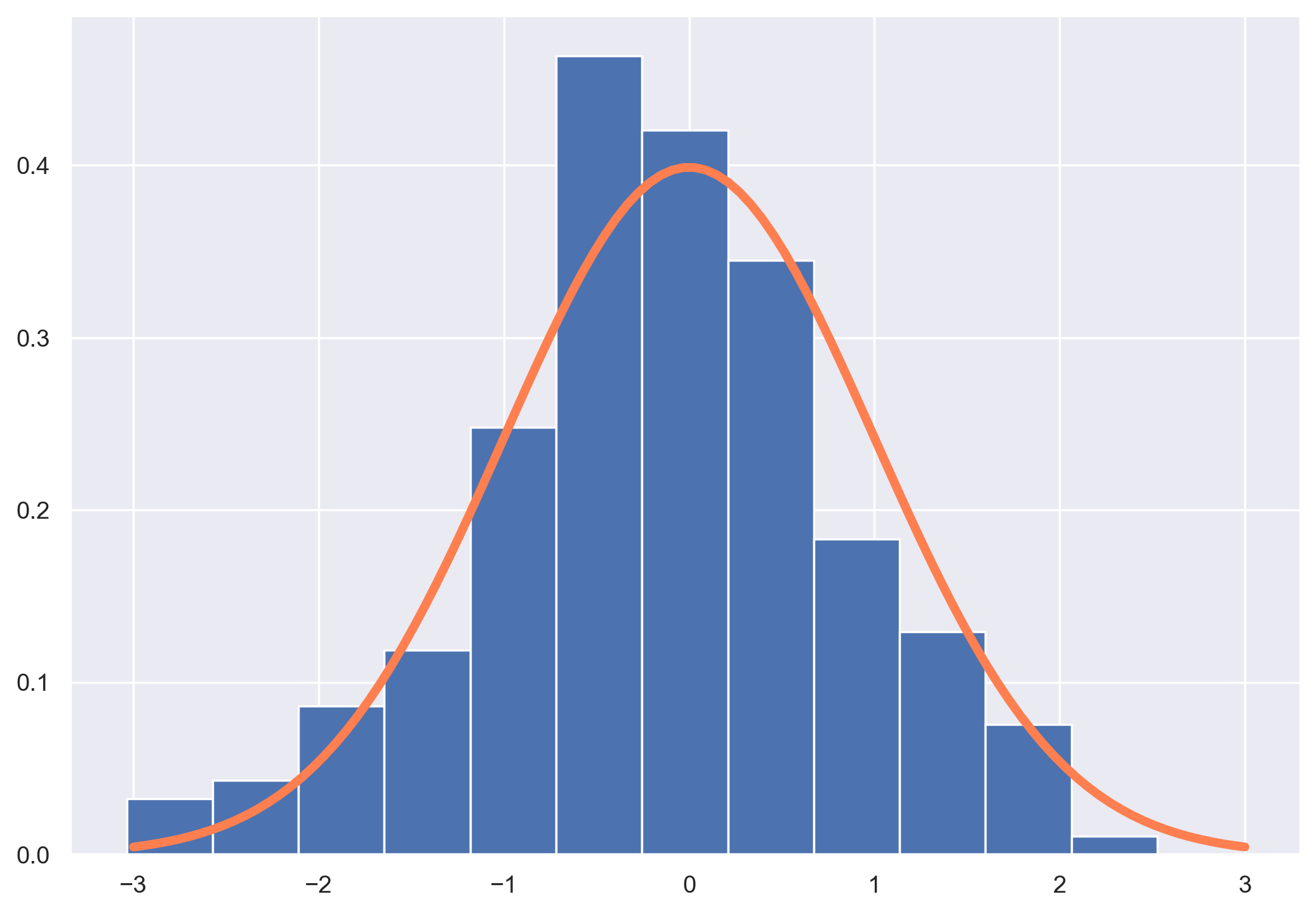}}}
{\subfigure[SGN - $\varepsilon = 0.01$]{\includegraphics[width=0.45 \textwidth,height=0.23\textwidth]{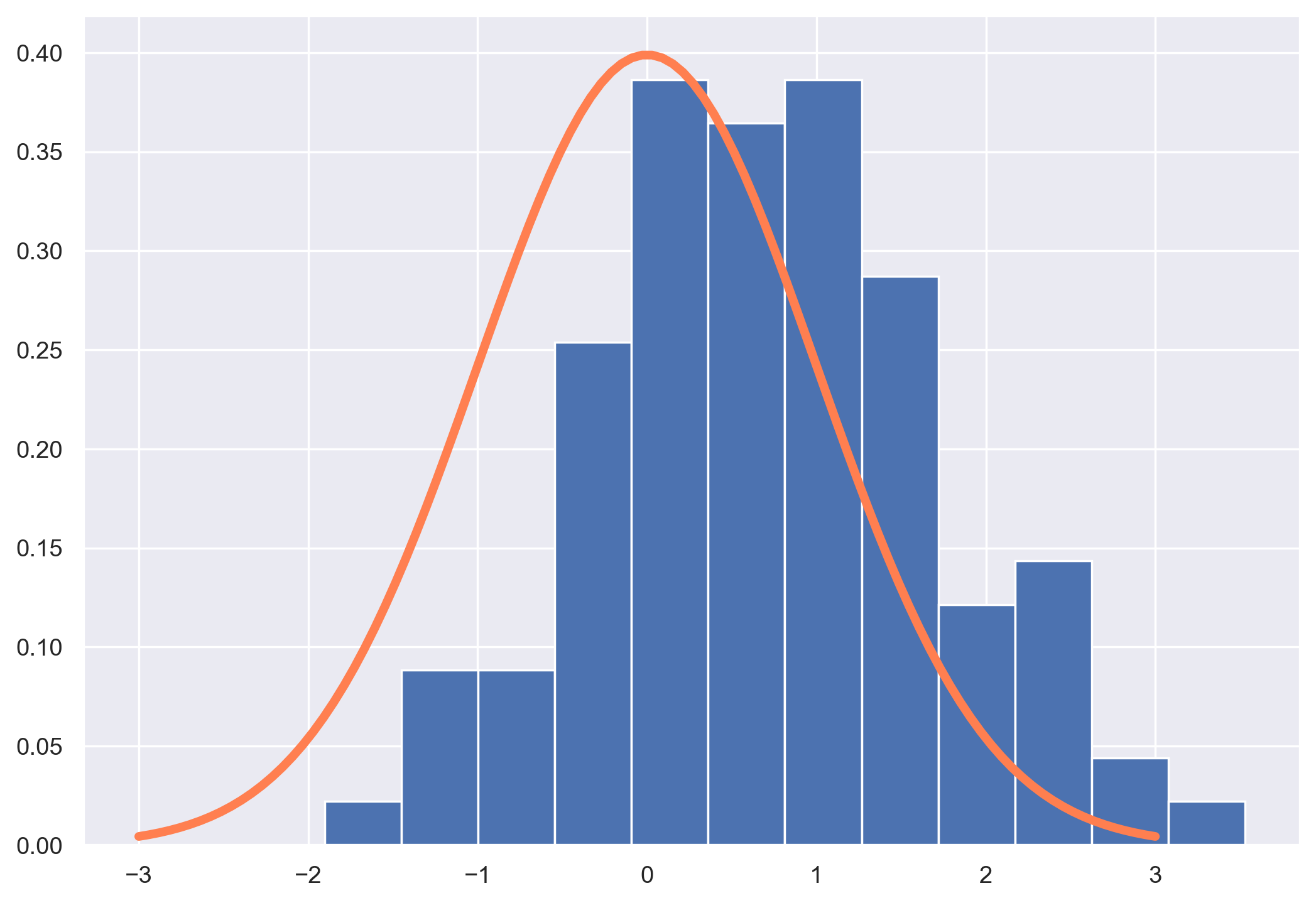}}}
{\subfigure[SN - $\varepsilon = 0.01$]{\includegraphics[width=0.45 \textwidth,height=0.23\textwidth]{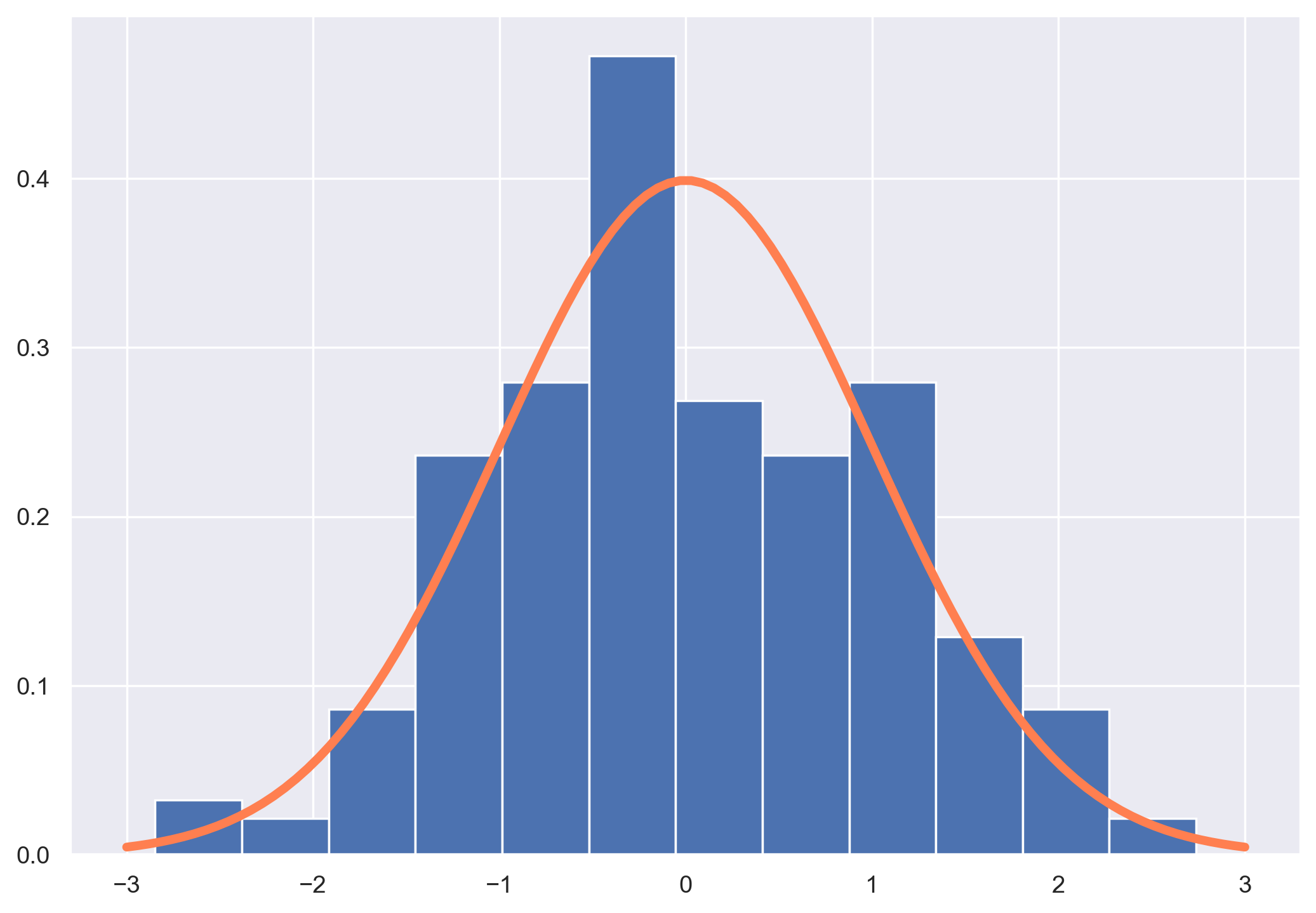}}}
{\subfigure[ADAM - $\varepsilon = 0.01$]{\includegraphics[width=0.45 \textwidth,height=0.23\textwidth]{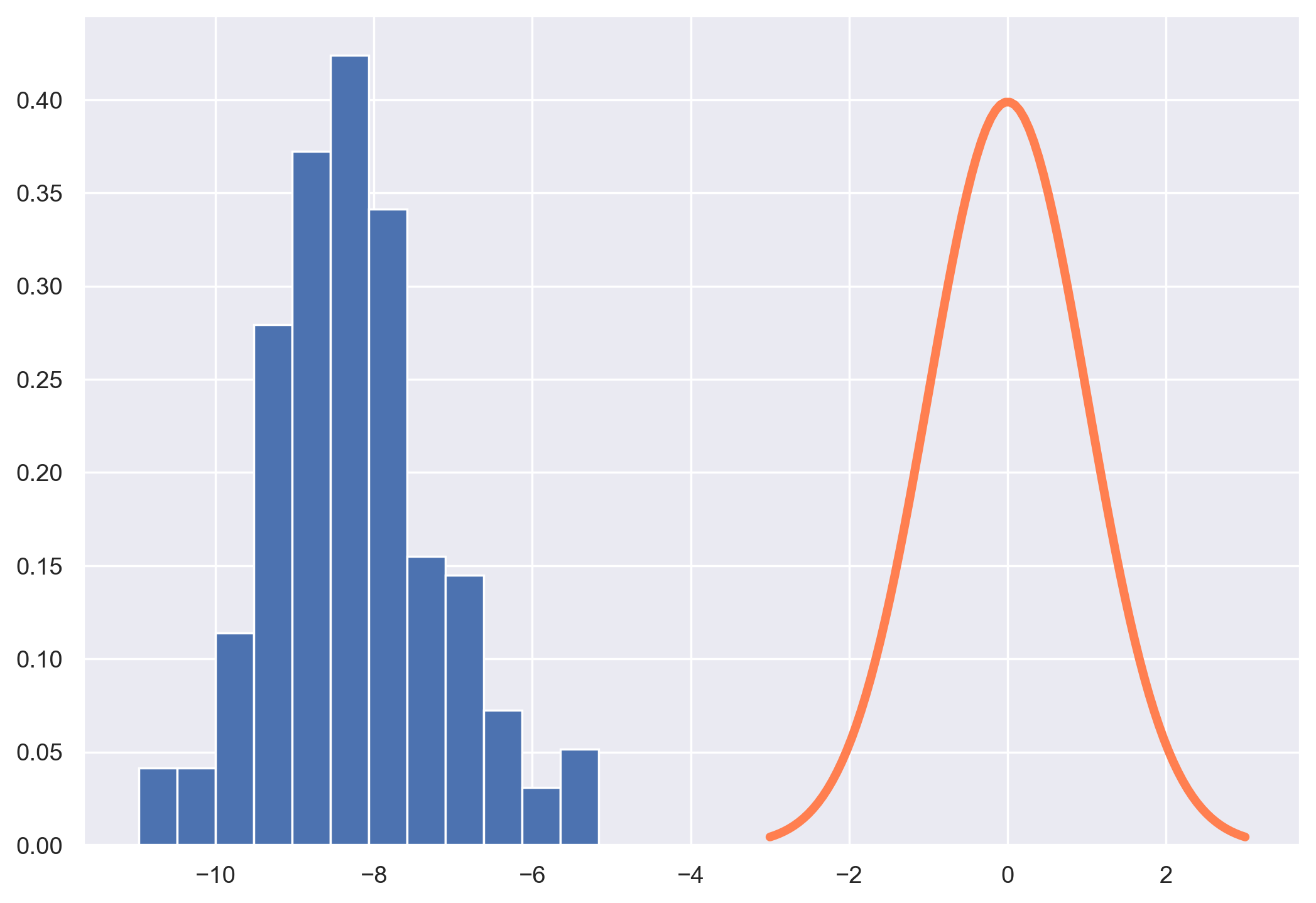}}}

\caption{Discrete setting with $I=10^3$, $J=50$ and $\varepsilon = 0.01$. Histogram of 200 independent realizations of $\frac{\sqrt{n}  \left( \wh{W}_n  -W_{\varepsilon}(\mu,\nu) \right)}{ \wh{\sigma}_n}$  with $n= 4 \times 10^5$ iterations using each of  the four stochastic algorithms. The orange curve is the density of the standard Gaussian distribution.  \label{fig:TCL_Wn_eps_0_01}}
\end{figure}

}

\CB{

In Figure \ref{fig:TCL_Vn_eps_0_1} and  Figure \ref{fig:TCL_Vn_eps_0_01}, we also display the histograms of 200 independent realizations of $\widetilde{V}_n =  n  \bigl\| \wh{V}_n -v^\ast \bigr\|^2$ for the four stochastic algorithms. The distribution of $\widetilde{V}_n$ has the shape of a $\chi^2$-distribution but the ``number of degrees of freedom'' is highly varying from one algorithm to the other. It can be seen from  Figure \ref{fig:TCL_Vn_eps_0_1} and  Figure \ref{fig:TCL_Vn_eps_0_01}, that  $\widetilde{V}_n$ reaches its smallest variance for the SN algorithm, and that the second smallest variance is obtained with the SGN algorithm. The SGD and the ADAM algorithms finally have a much larger variance.

 Therefore, these numerical experiments clearly show  that using the SGN algorithm has  interesting benefits as it outperforms SGD and ADAM for the estimation of $v^\ast$ since it yields an estimator  $\widetilde{V}_n$ with a smaller variance. 
 
\paragraph{Convergence of the pre-conditionning matrix for the SN algorithm.} Finally, we report in Figure \ref{fig:Sn} numerical results (with $I=10^4$ and $J\in\{100,200\}$) on the convergence of   $\overline{S}_n$ to $G_{\varepsilon}(v^\ast)$ as a function the computational time of the SGN algorithm (using $n=10^5$ iterations) for different values of the regularization parameter $\varepsilon$. We observe that the convergence becomes slower as $\varepsilon$ decreases.

}

\CB{


\begin{figure}[htbp]
\centering
{\subfigure[SGD - $\varepsilon = 0.1$]{\includegraphics[width=0.45 \textwidth,height=0.23\textwidth]{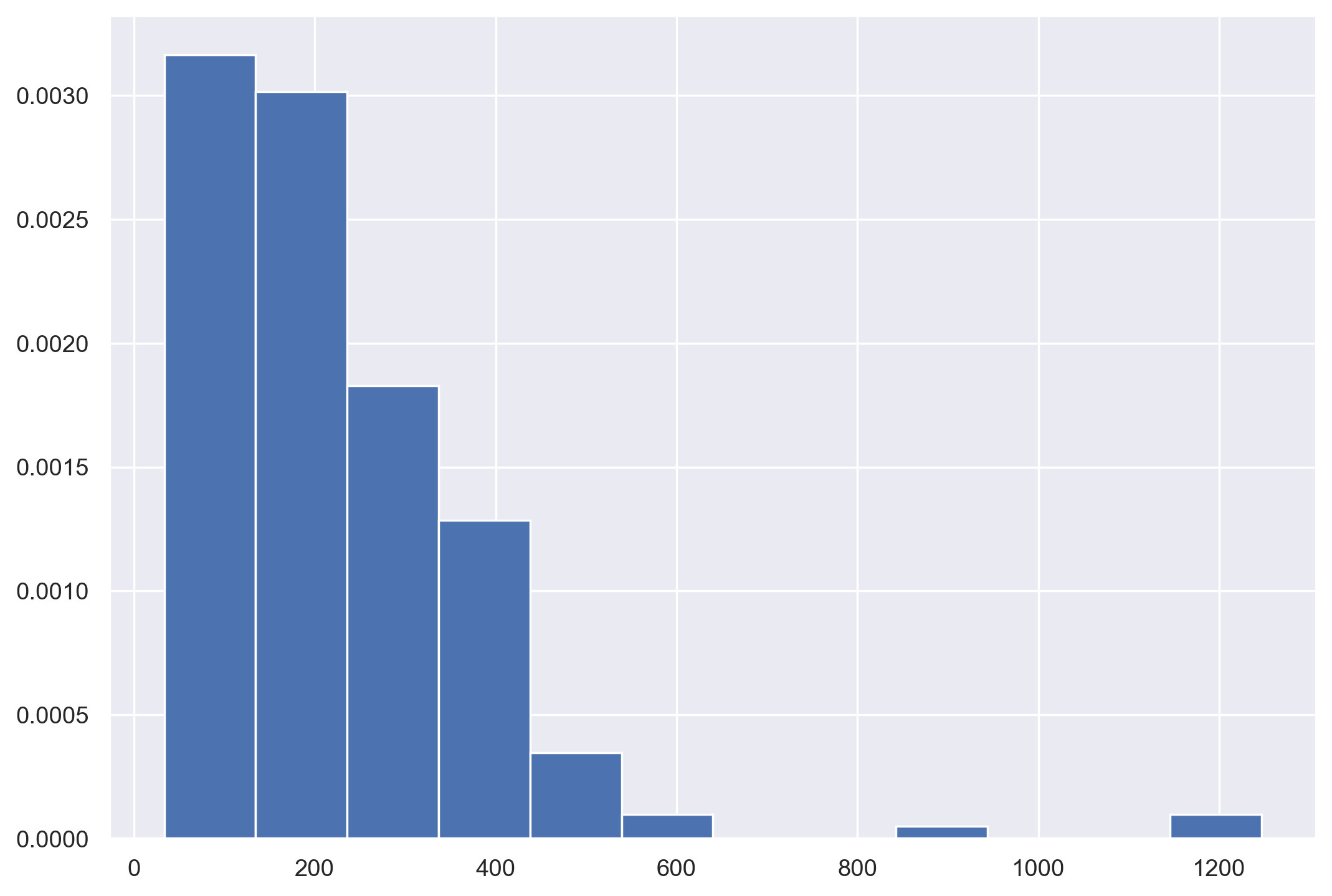}}}
{\subfigure[SGN - $\varepsilon = 0.1$]{\includegraphics[width=0.45 \textwidth,height=0.23\textwidth]{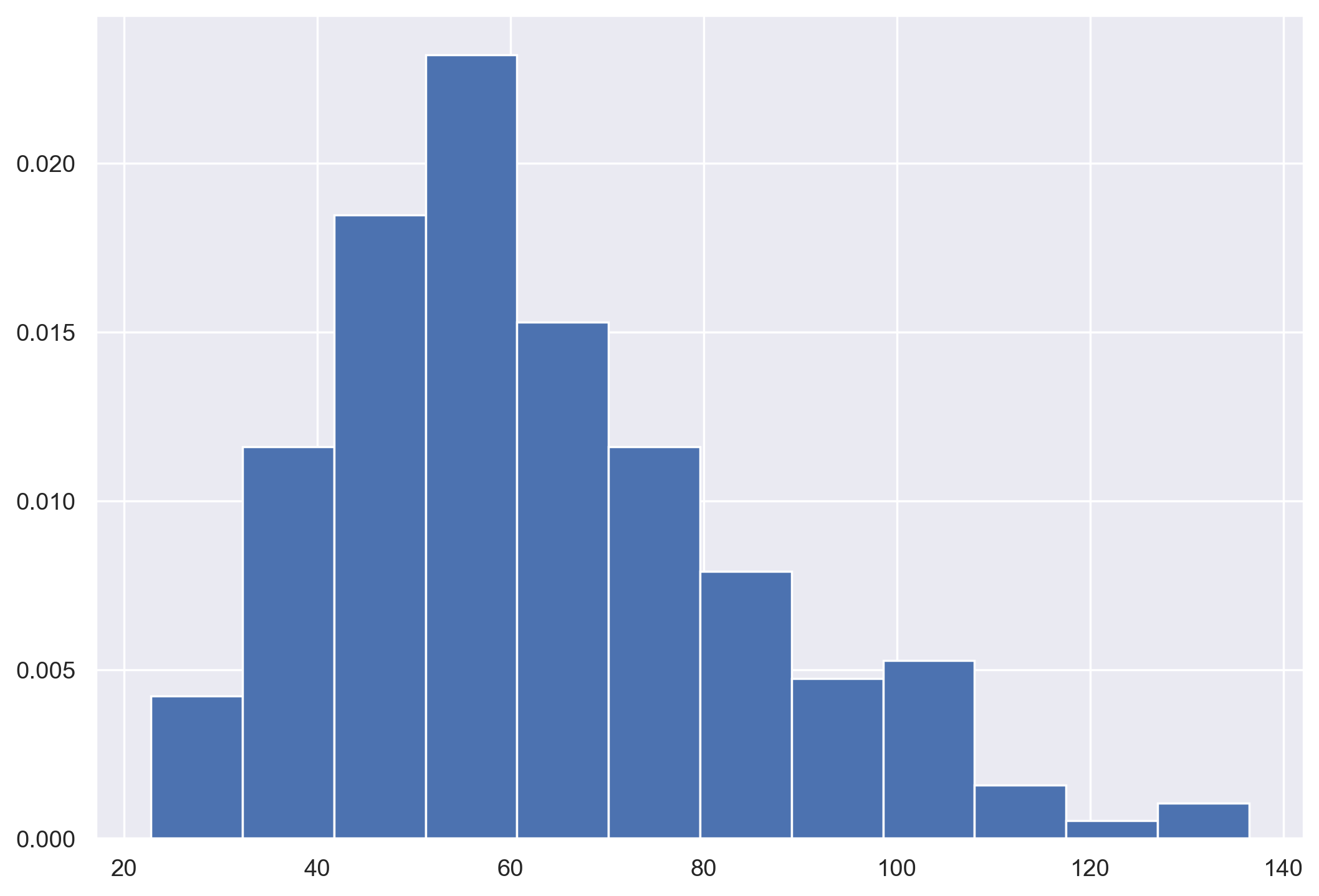}}}
{\subfigure[SN - $\varepsilon = 0.1$]{\includegraphics[width=0.45 \textwidth,height=0.23\textwidth]{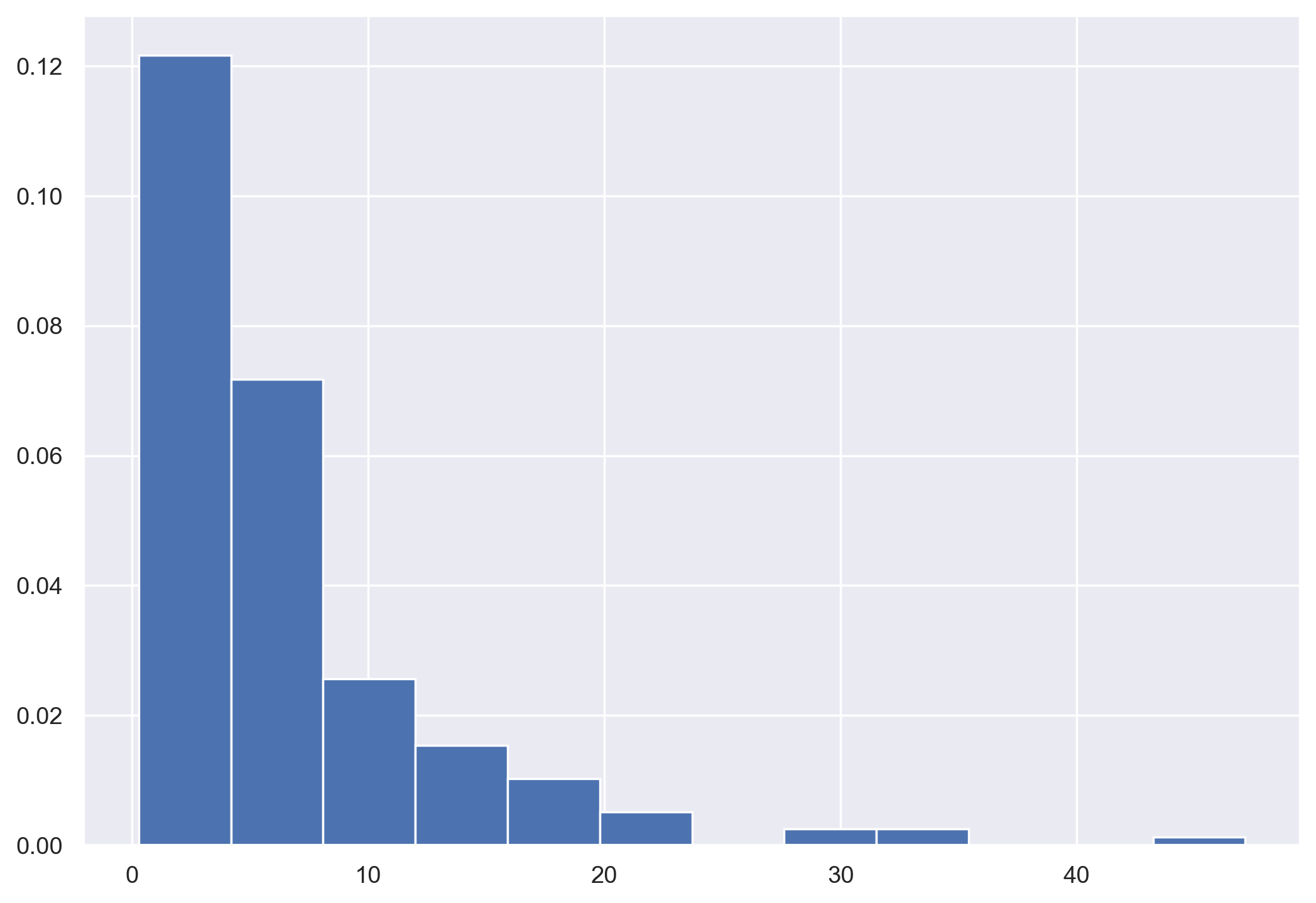}}}
{\subfigure[ADAM - $\varepsilon = 0.1$]{\includegraphics[width=0.45 \textwidth,height=0.23\textwidth]{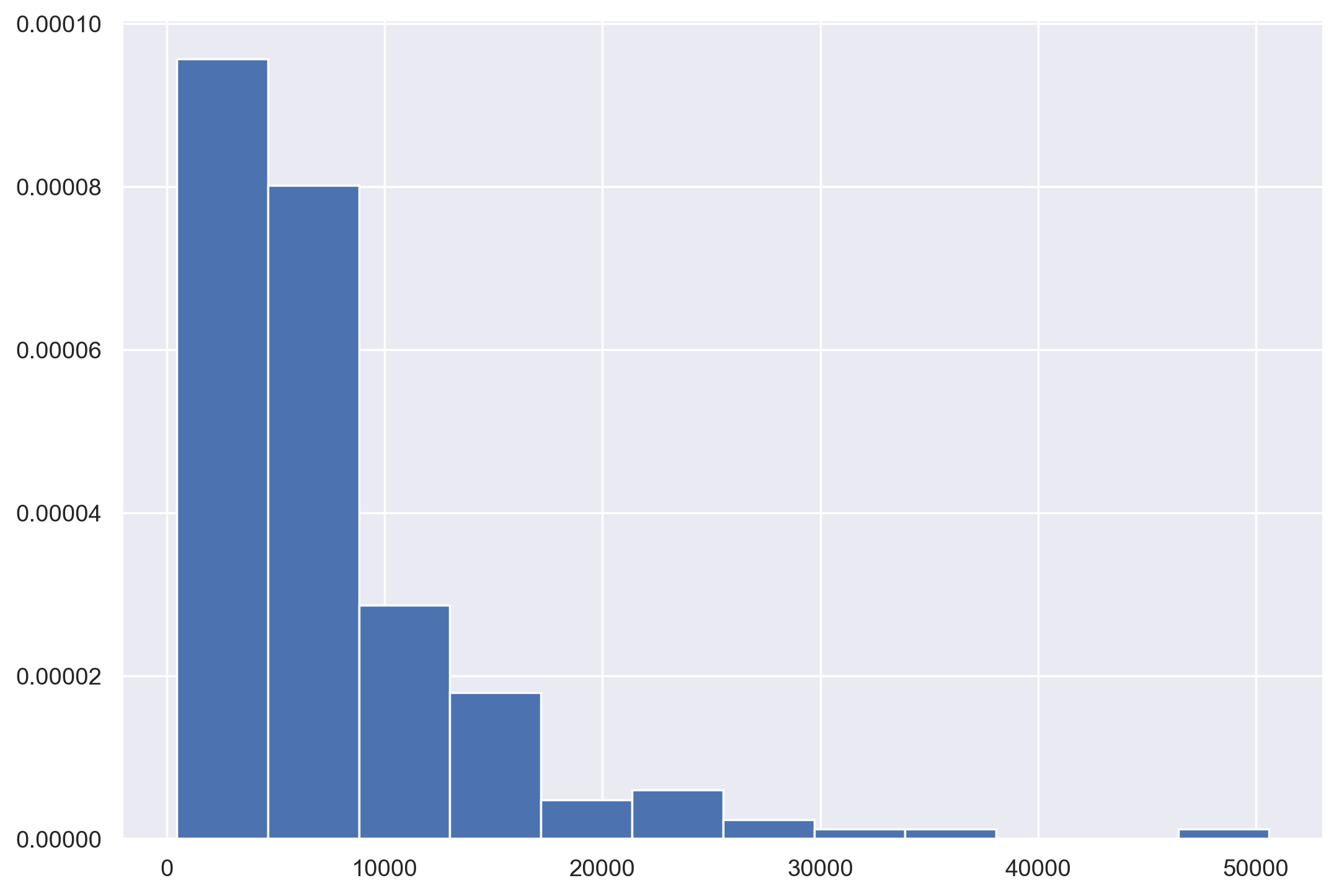}}}

\caption{Discrete setting with $I=10^3$,  $J=50$ and $\varepsilon = 0.1$. Histogram of 200 independent realizations of $ n  \bigl\| \wh{V}_n -v^\ast \bigr\|^2$  with $n= 2 \times 10^5$ iterations using each of  the four stochastic algorithms.  \label{fig:TCL_Vn_eps_0_1}}
\end{figure}

\begin{figure}[htbp]
\centering
{\subfigure[SGD - $\varepsilon = 0.01$]{\includegraphics[width=0.45 \textwidth,height=0.23\textwidth]{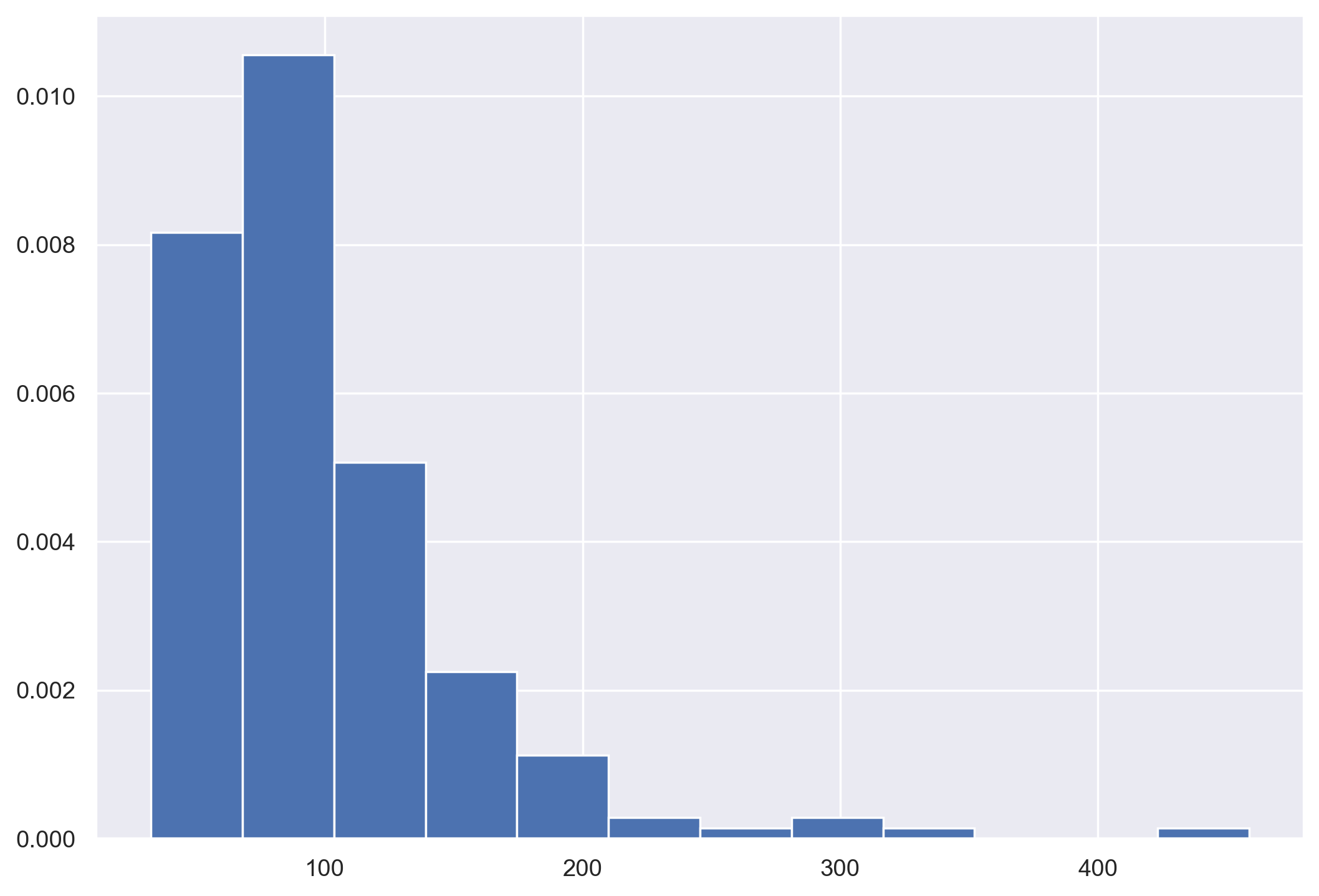}}}
{\subfigure[SGN - $\varepsilon = 0.01$]{\includegraphics[width=0.45 \textwidth,height=0.23\textwidth]{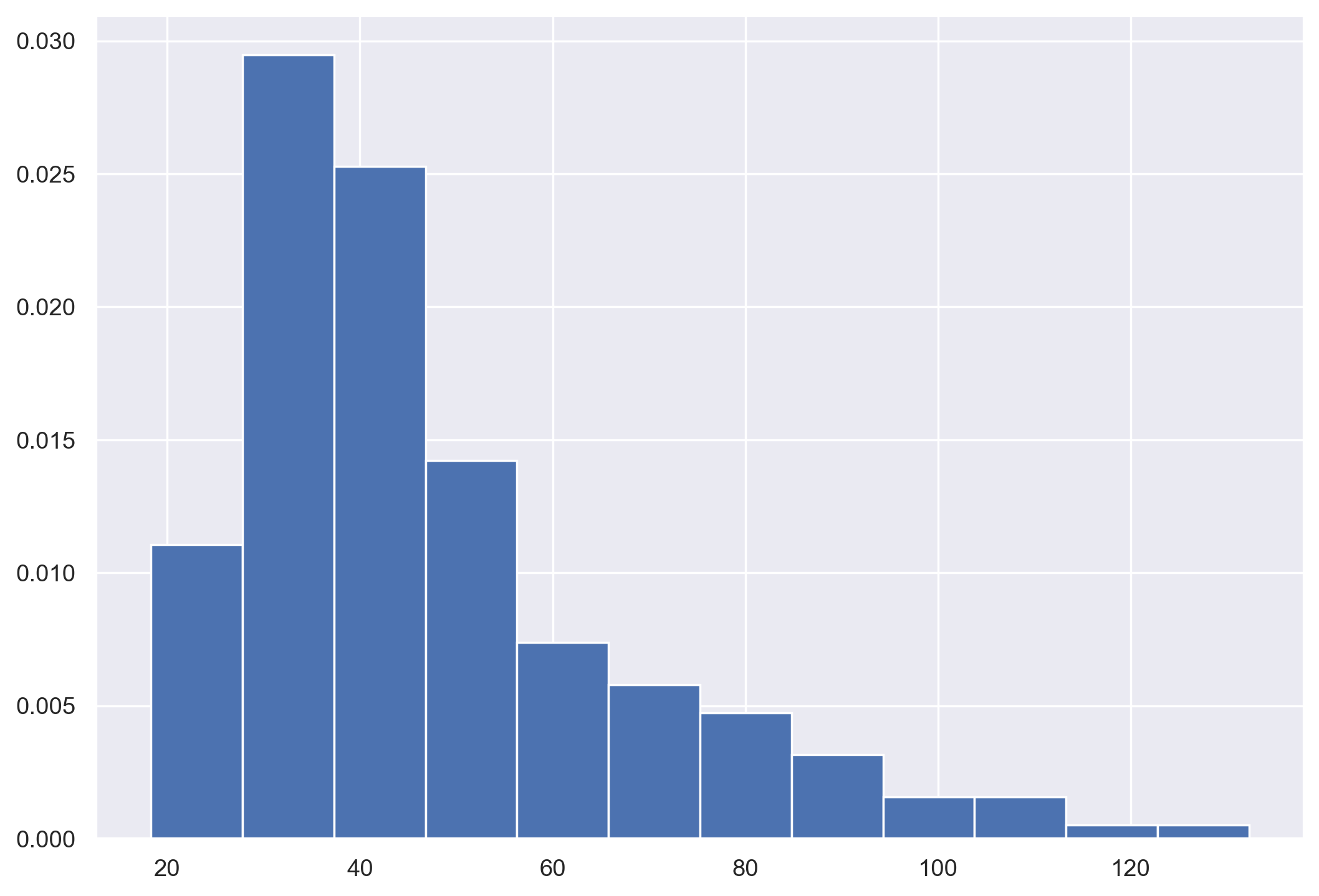}}}
{\subfigure[SN - $\varepsilon = 0.01$]{\includegraphics[width=0.45 \textwidth,height=0.23\textwidth]{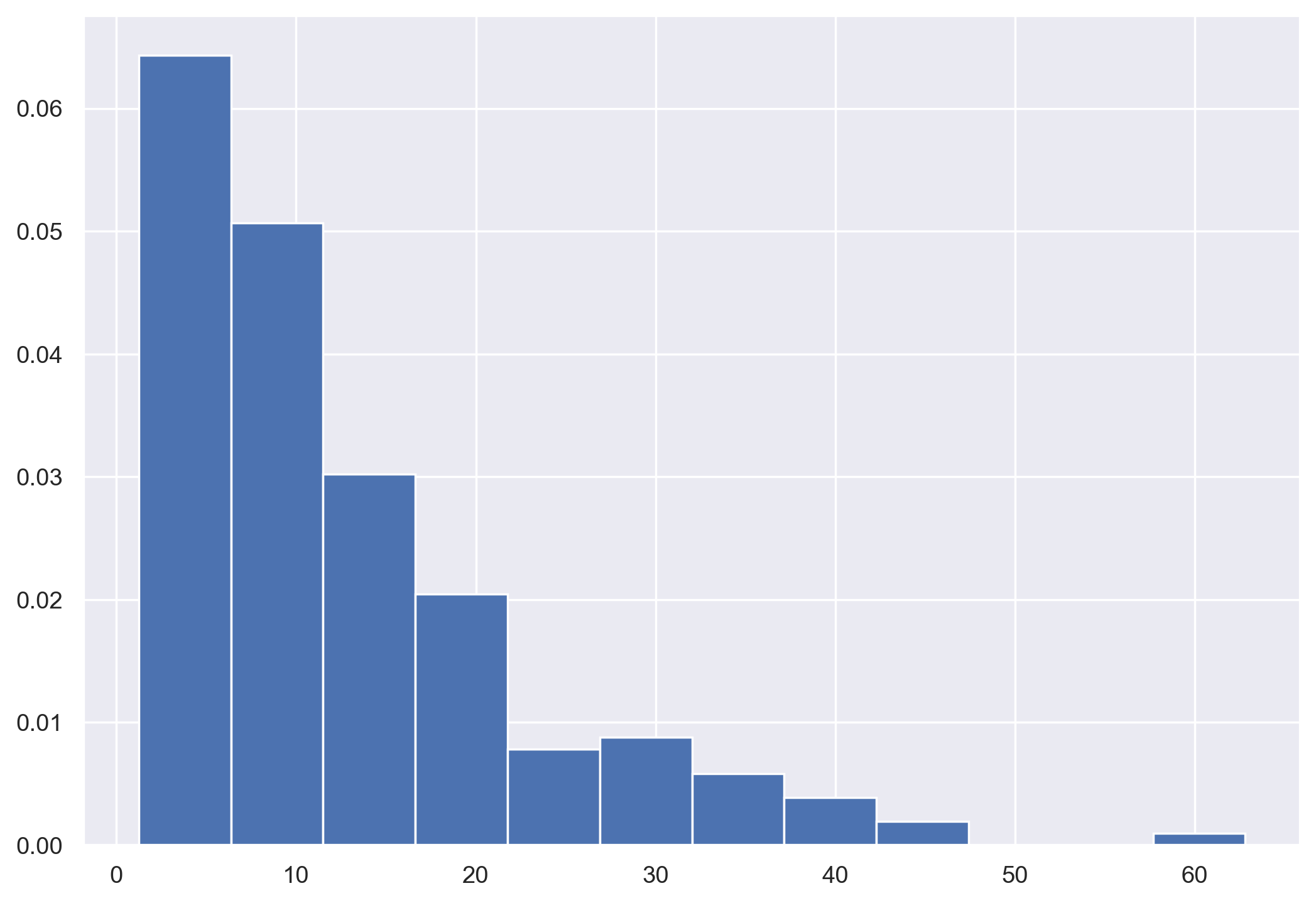}}}
{\subfigure[ADAM - $\varepsilon = 0.01$]{\includegraphics[width=0.45 \textwidth,height=0.23\textwidth]{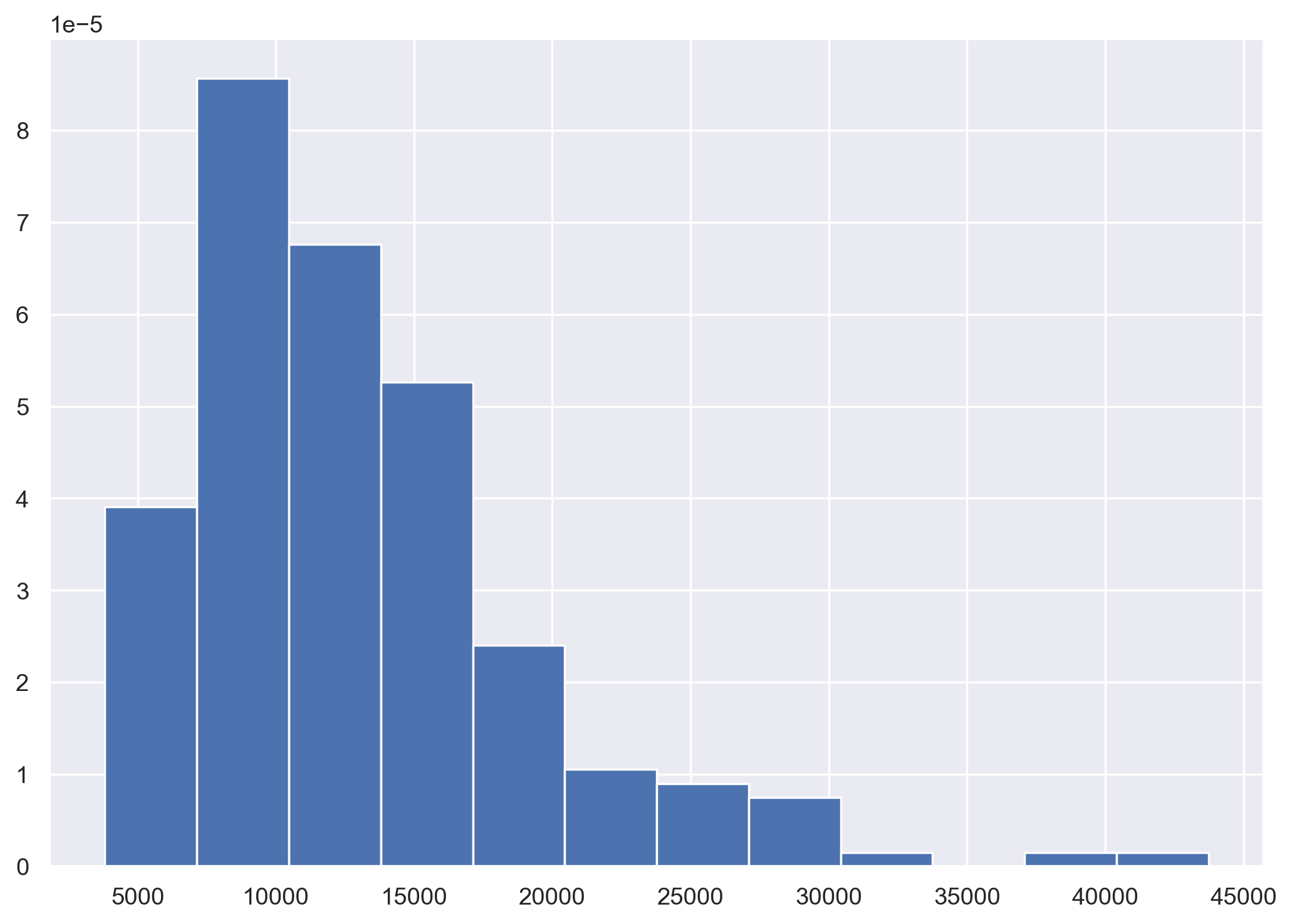}}}

\caption{Discrete setting with $I=10^3$, $J=50$ and $\varepsilon = 0.01$. Histogram of 200 independent realizations of $ n  \bigl\| \wh{V}_n -v^\ast \bigr\|^2$  with $n= 4 \times 10^5$ iterations using each of  the four stochastic algorithms.  \label{fig:TCL_Vn_eps_0_01}}
\end{figure}


\begin{figure}[htbp]
\centering
{\subfigure[$I=10^4$ and $J=100$]{\includegraphics[width=0.45 \textwidth,height=0.33\textwidth]{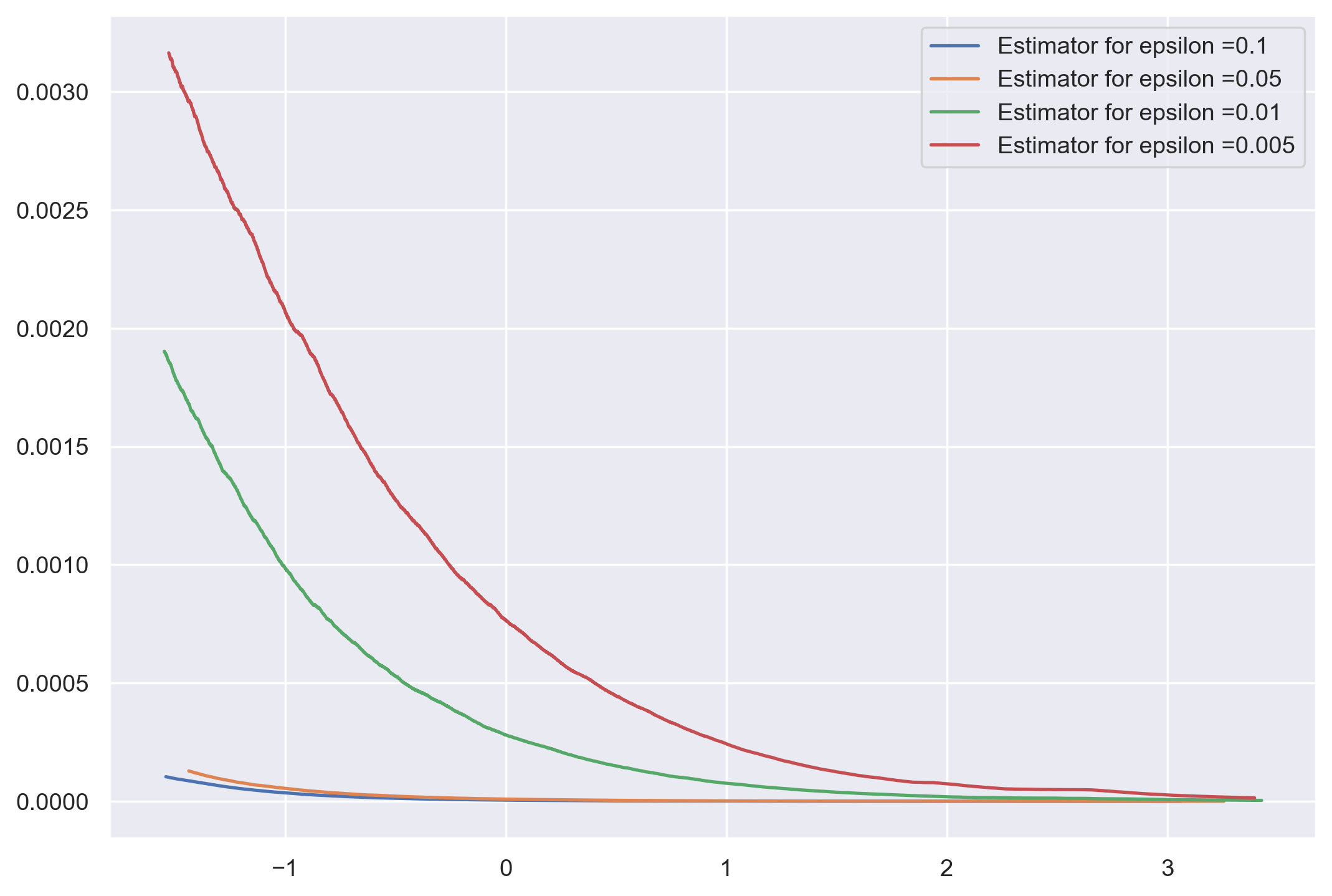}}}
{\subfigure[$I=10^4$ and $J=200$]{\includegraphics[width=0.45 \textwidth,height=0.33\textwidth]{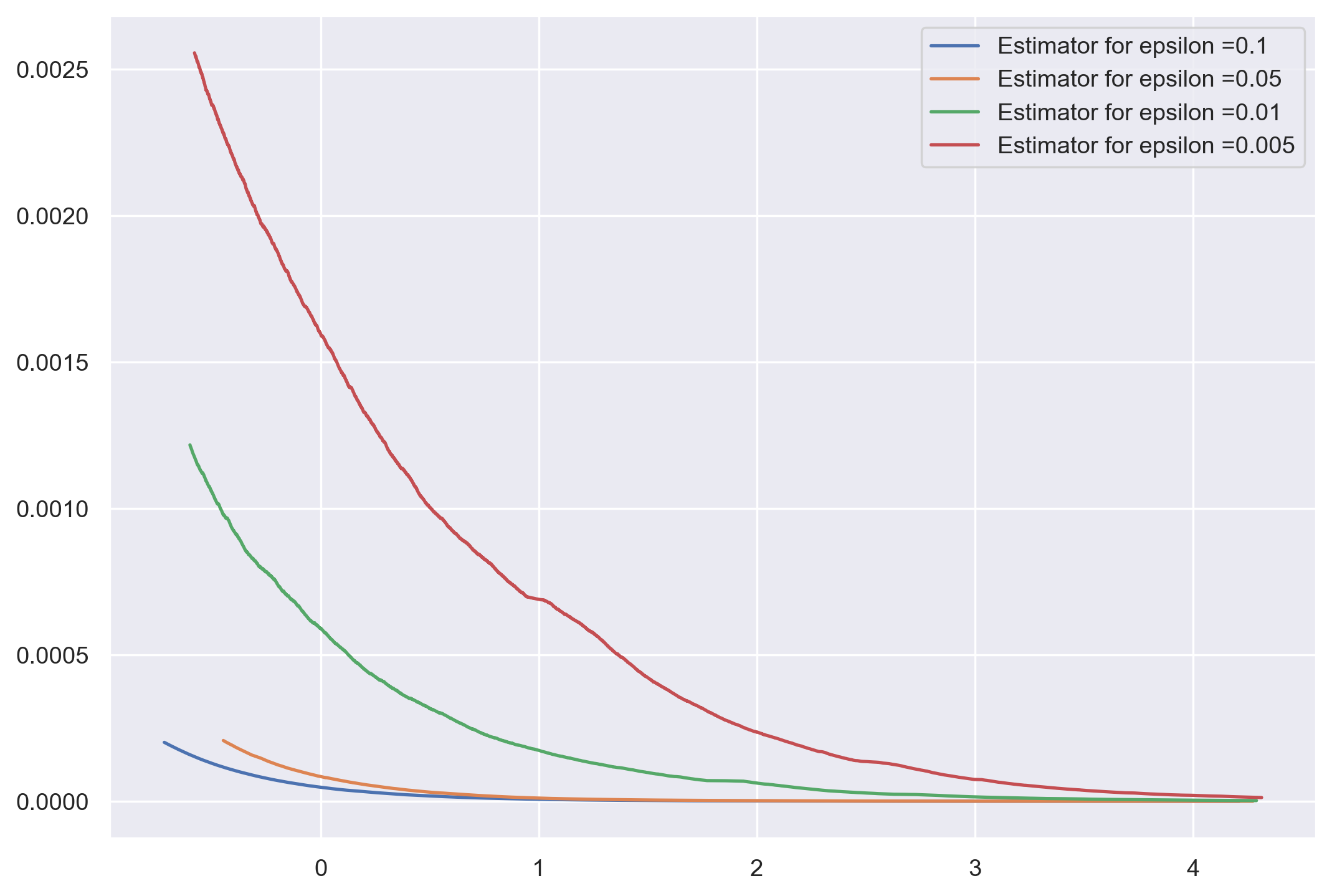}}}

\caption{Discrete setting and convergence of  $\log\left(\bigl \|  \overline{S}_n- G_{\varepsilon}(v^\ast) \bigr \|_F^2\right)$ as a function the computational time of the SGN algorithm (using $n=10^5$ iterations) for different values of the regularization parameter $\varepsilon$.  \label{fig:Sn}}
\end{figure}
}

\subsubsection{Semi-discrete setting}

\CB{

In this section, the cost function is chosen as the following normalized squared Euclidean distance $c(x,y) = \frac{1}{d} \|x-y\|^2$.
We now consider the framework where $\mu$ is  a mixture of three Gaussian densities in dimension $d$.   In these numerical experiments,  the size $J = 100$ of the support of $\nu$ is held fixed, and it is chosen as the uniform discrete probability measure supported on $J$ points drawn uniformly on the hypercube $[0,1]^d$. The value of the dimension $d$ is let growing, and we  analyze its influence on the performances of the stochastic algorithms with  either $n = 5 \times 10^5$ or $n=10^6$ iterations. We also study the performances of the Sinkhorn algorithm from a full-batch sample,  that is using the empirical measure
$$
\hat{\mu}_n = \frac{1}{n} \sum_{i= 1}^{n} \delta_{X_{i}}, \quad \mbox{where} \quad X_1,\ldots,X_n \sim_{iid} \mu,
$$
to compute a solution of the regularized OT problem $W_{\varepsilon}(\hat{\mu}_n,\nu)$ as an approximation of $W_{\varepsilon}(\mu,\nu)$.  At each iteration, the cost  of this Sinkhorn algorithm is thus  $\mathcal{O}(nJ)$ as $n$ is the size of the support of $\hat{\mu}_n$. We have chosen to display results for only one simulation as averaging over Monte-Carlo replications does not change our main conclusions from these numerical experiments. All  stochastic algorithms (resp.\ Sinkhorn algorithm) are  compared for the metric $ \bigl\| \wh{V}_n -v^\ast \bigr\|^2$ (resp.\ $\bigl\|  V_k  -v^\ast \bigr\|^2$), where the vector $v^\ast$ is preliminary approximated by running the SN algorithm with a  large value of  iterations $n_{\max} = 10^6$.

In Figure \ref{fig:excess_risk_semidiscrete_Vn} (for $n=5 \times 10^5$ iterations) and Figure \ref{fig:excess_risk_semidiscrete_Vn_large} (for $n=n_{\max}=10^6$ iterations) ,  we display these metrics (in logarithmic scale)    as functions of the computational time of the stochastic and the Sinkhorn algorithms  for different values of dimension $d \in \{5, 10, 50\}$ and $\varepsilon \in \{0.01,0.005\}$. First, for either $n = 5 \times 10^5$ or $n=10^6$,  it can be observed that the SN algorithm has always the best performances, while ADAM has the worst ones.  Moreover, the SGN algorithm has slightly better performances than  SGD. In  Figure \ref{fig:excess_risk_semidiscrete_Vn_large}, the very fast decay of the error after 450 seconds for  SN is due to the fact that the ground truth value $v^\ast$ has been preliminary computed with the SN algorithm with $n=n_{\max}=10^6$ iterations. For either $n = 5 \times 10^5$ or $n=10^6$, we also remark that the SGN  outperforms Sinkhorn for $\varepsilon = 0.005$ and $d \in \{5,10\}$ in the sense that when the SGN stops the value reached by $ \bigl\| \wh{V}_n -v^\ast \bigr\|^2$ is smaller than  $\bigl\|  V_k  -v^\ast \bigr\|^2$ obtained with Sinkhorn. In larger dimension $d=50$, the Sinkhorn algorithm appears to have better performances than SGN and SGD, but we recall that Sinkhorn uses the full sample at each iteration.

Therefore, these numerical experiments show that SGN has interesting benefits over   Sinkhorn in small dimension $d$ when combined with small values of the regularization parameter $\varepsilon$ and large values of $n$.  Moreover, we observe that the best results are always obtained with the SN algorithm.

\begin{figure}[htbp]
\centering
{\subfigure[$d = 5$ and $\varepsilon = 0.01$]{\includegraphics[width=0.45 \textwidth,height=0.3\textwidth]{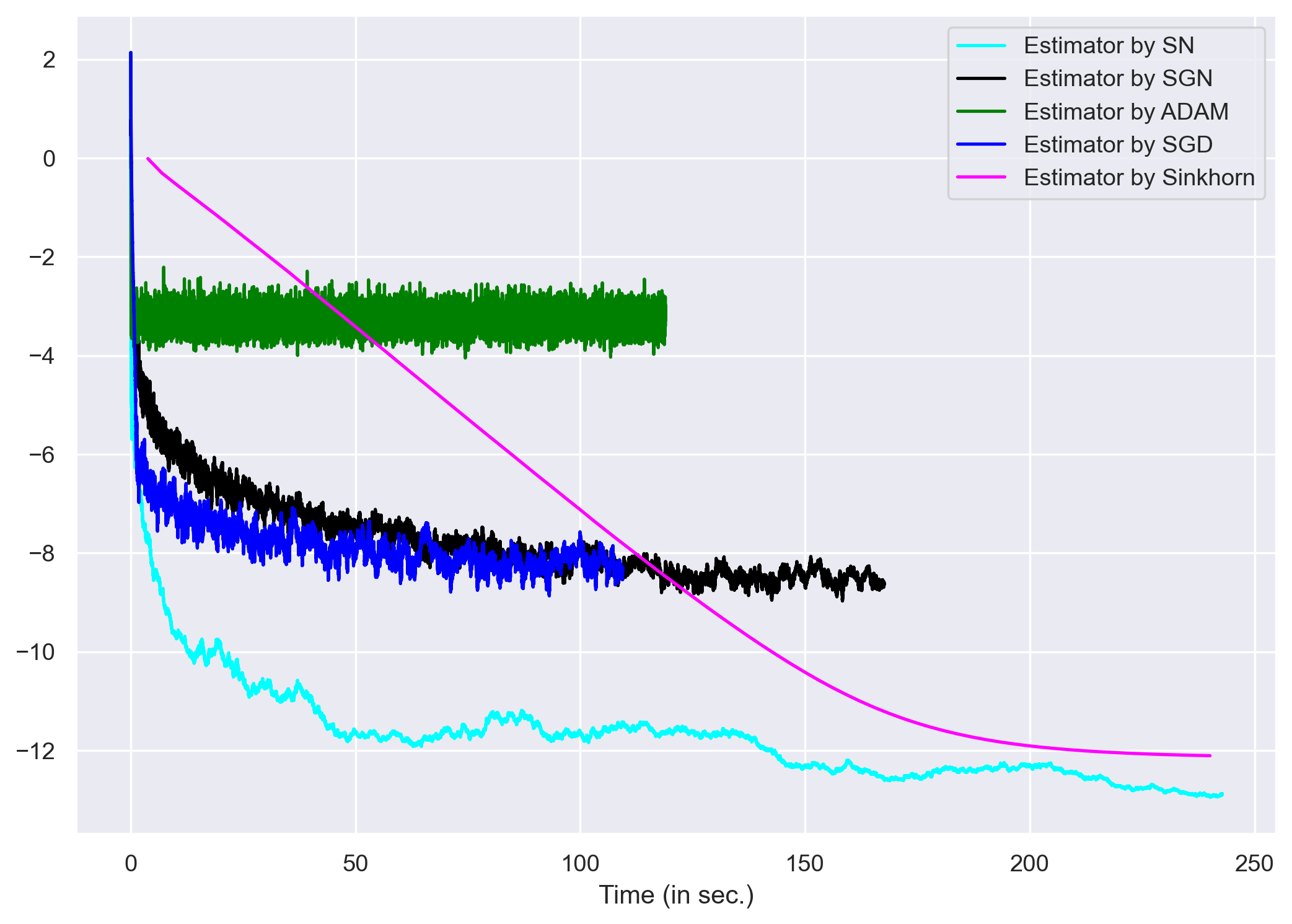}}}
{\subfigure[$d = 5$ and $\varepsilon = 0.005$]{\includegraphics[width=0.45 \textwidth,height=0.3\textwidth]{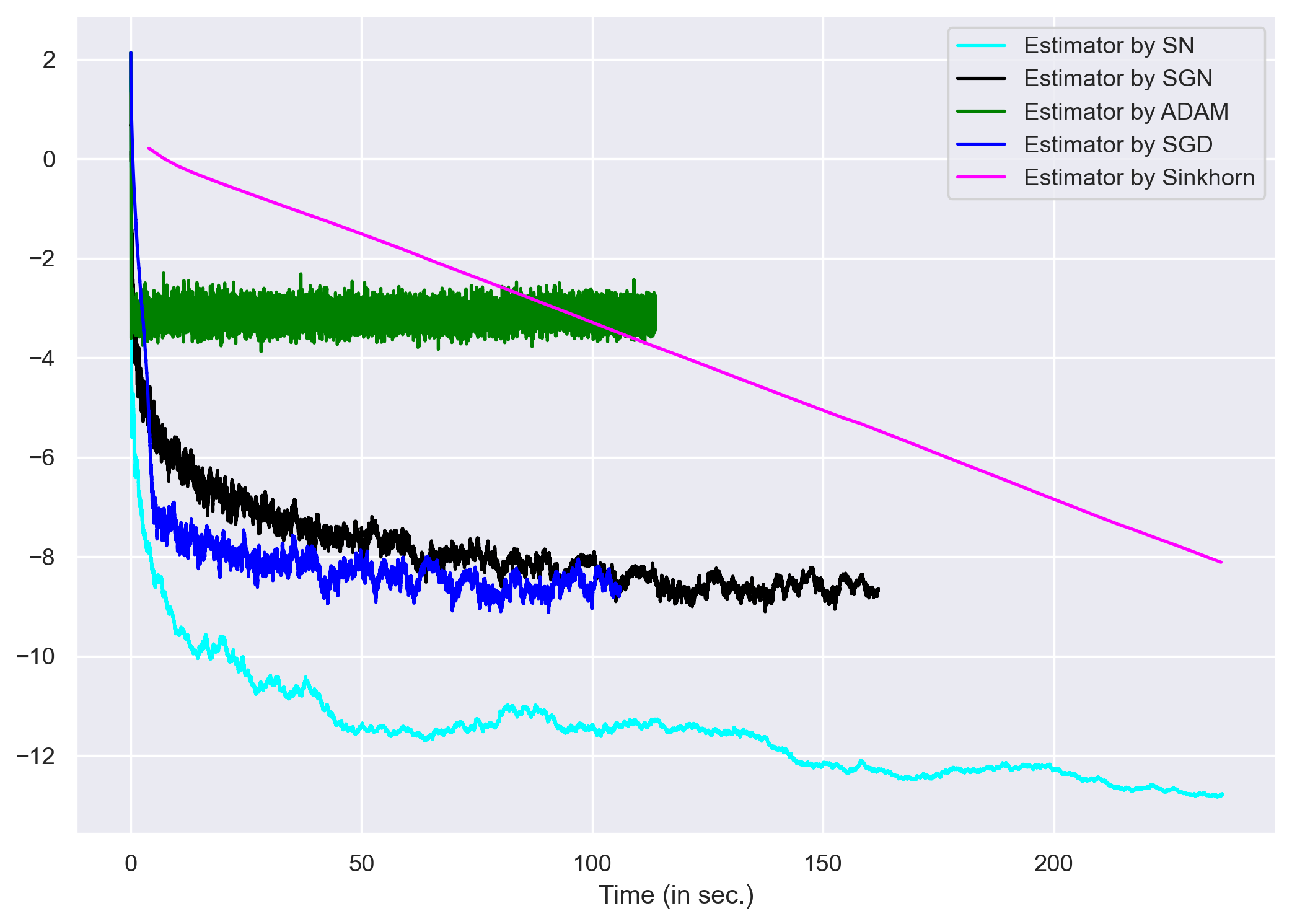}}}
{\subfigure[$d = 10$ and $\varepsilon = 0.01$]{\includegraphics[width=0.45 \textwidth,height=0.3\textwidth]{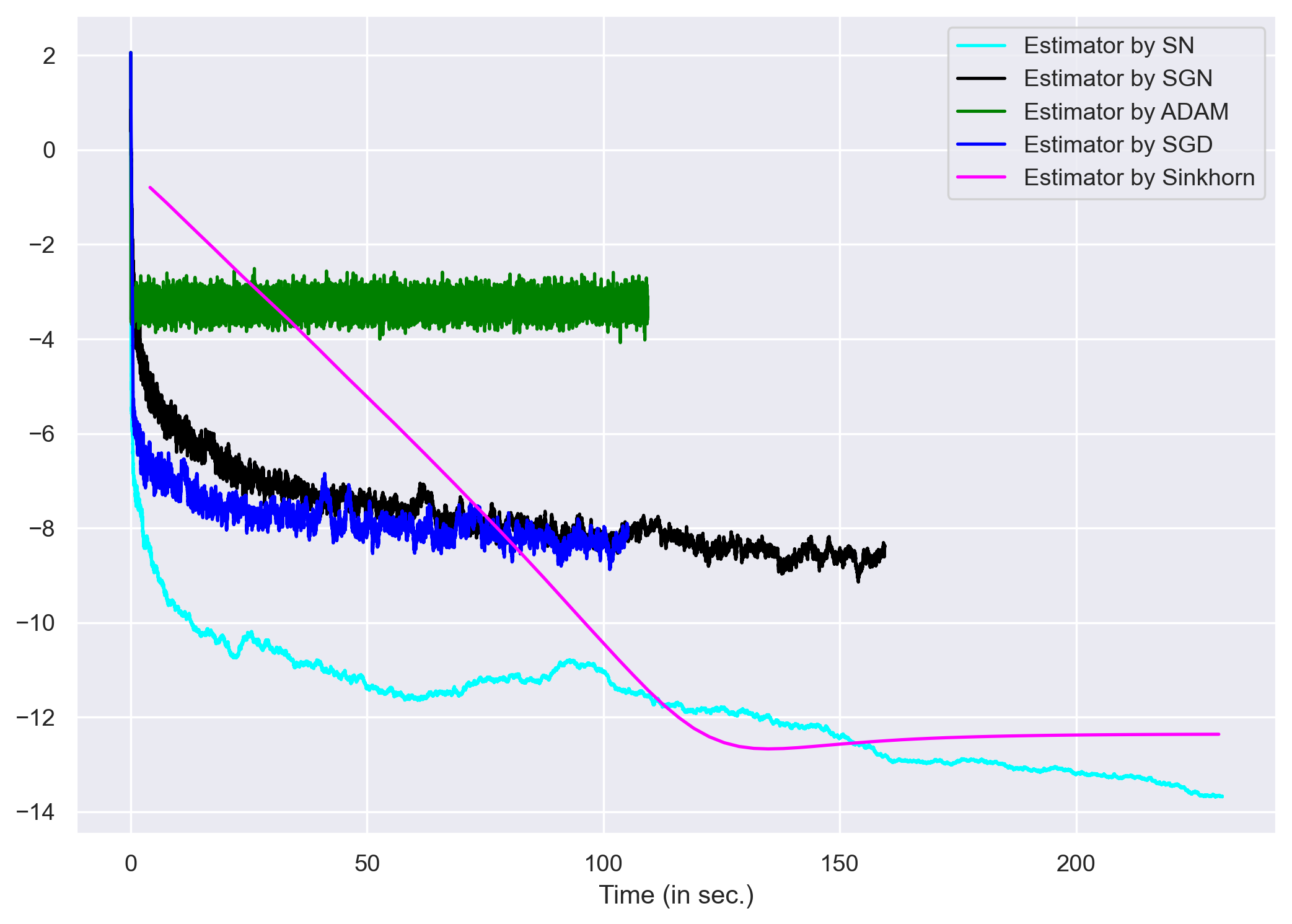}}}
{\subfigure[$d = 10$ and $\varepsilon = 0.005$]{\includegraphics[width=0.45 \textwidth,height=0.3\textwidth]{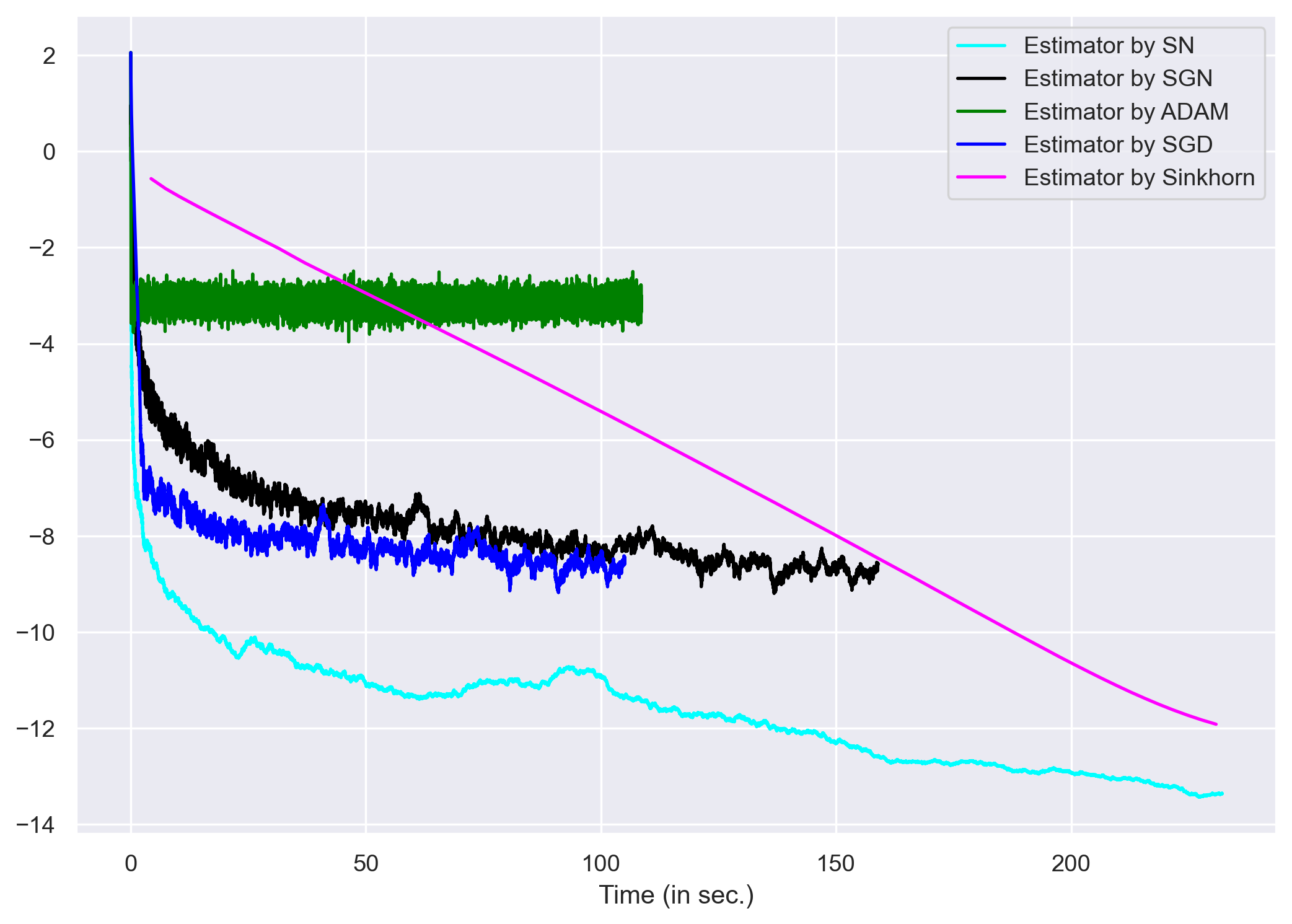}}}
{\subfigure[$d = 50$ and $\varepsilon = 0.01$]{\includegraphics[width=0.45 \textwidth,height=0.3\textwidth]{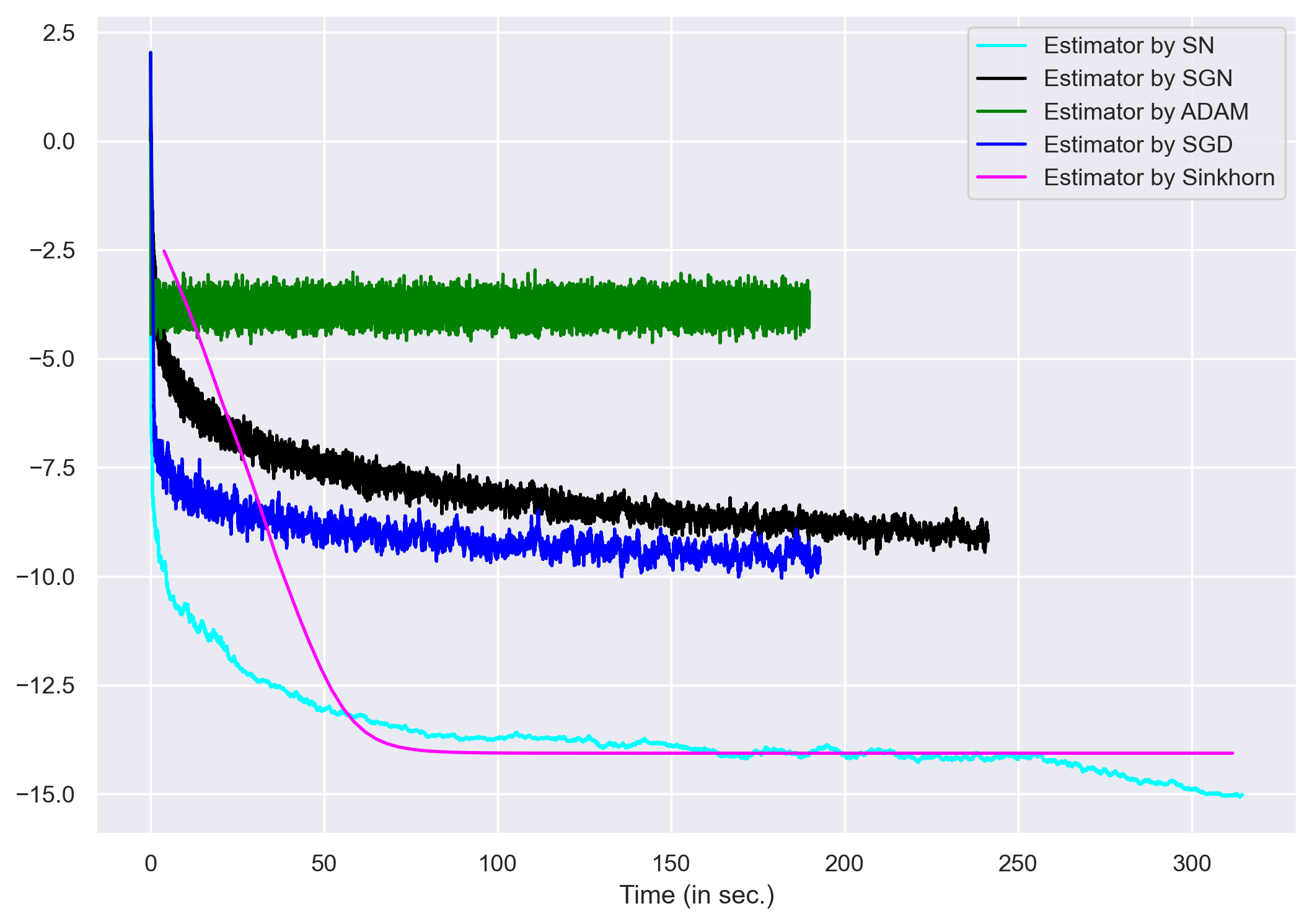}}}
{\subfigure[$d = 50$ and $\varepsilon = 0.005$]{\includegraphics[width=0.45 \textwidth,height=0.3\textwidth]{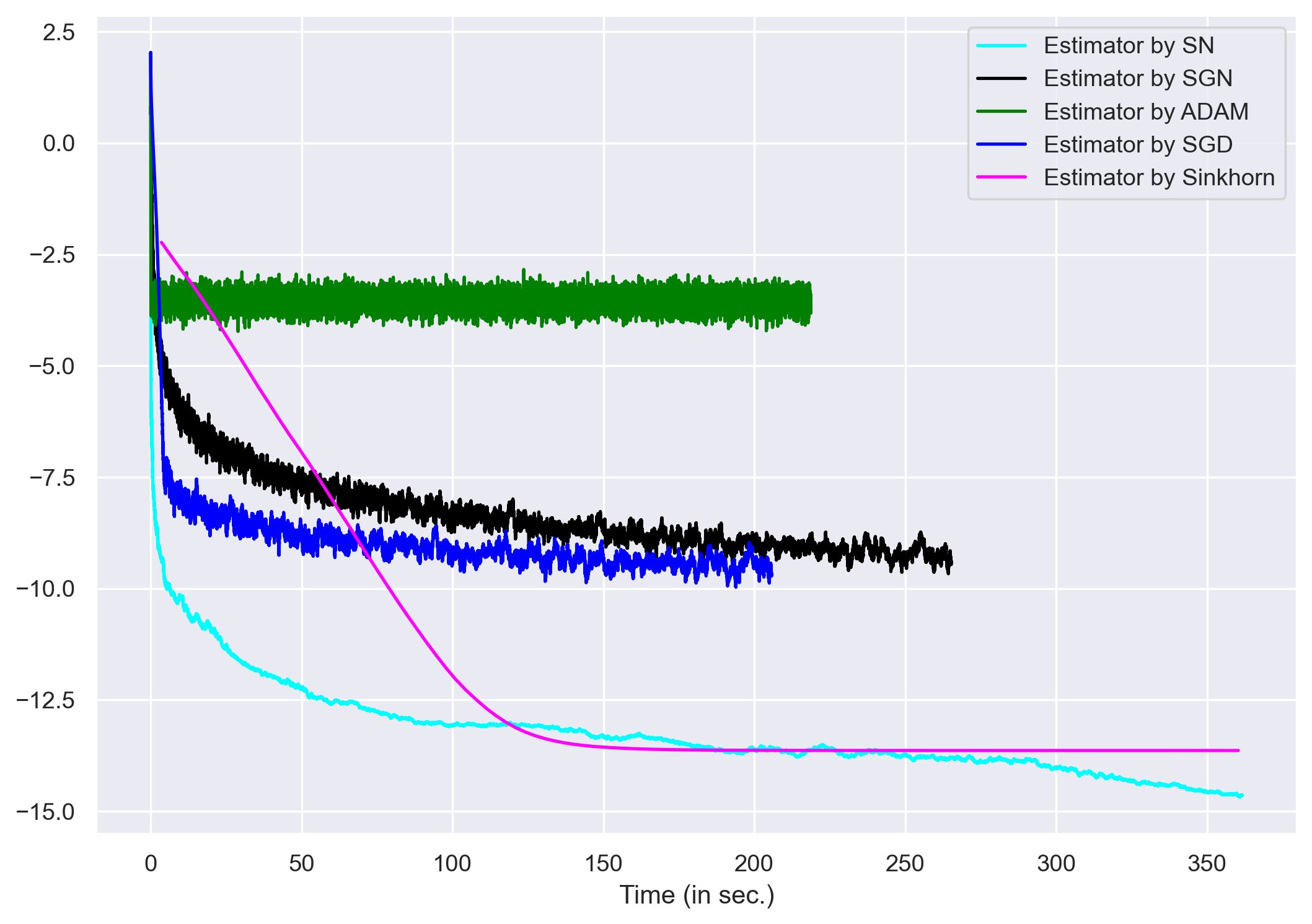}}}

\caption{Semi-discrete setting where $\mu$ is  a mixture of three Gaussian densities,   $J=100$ and $n= 5\times10^5$ iterations. Excess risk  (in logarithmic scale)  $ \log(    \bigl\| \wh{V}_n -v^\ast \bigr\|^2 )$  (resp.\ metric $ \log( \bigl\|  V_k  -v^\ast \bigr\|^2)$) as a function of the  computational cost of the iterations of  the four stochastic algorithms (resp.\ the Sinkhorn algorithm) for different values of the dimension $d$ and the regularization parameter $\varepsilon$. \label{fig:excess_risk_semidiscrete_Vn}}
\end{figure}

\begin{figure}[htbp]
\centering
{\subfigure[$d = 5$ and $\varepsilon = 0.01$]{\includegraphics[width=0.45 \textwidth,height=0.3\textwidth]{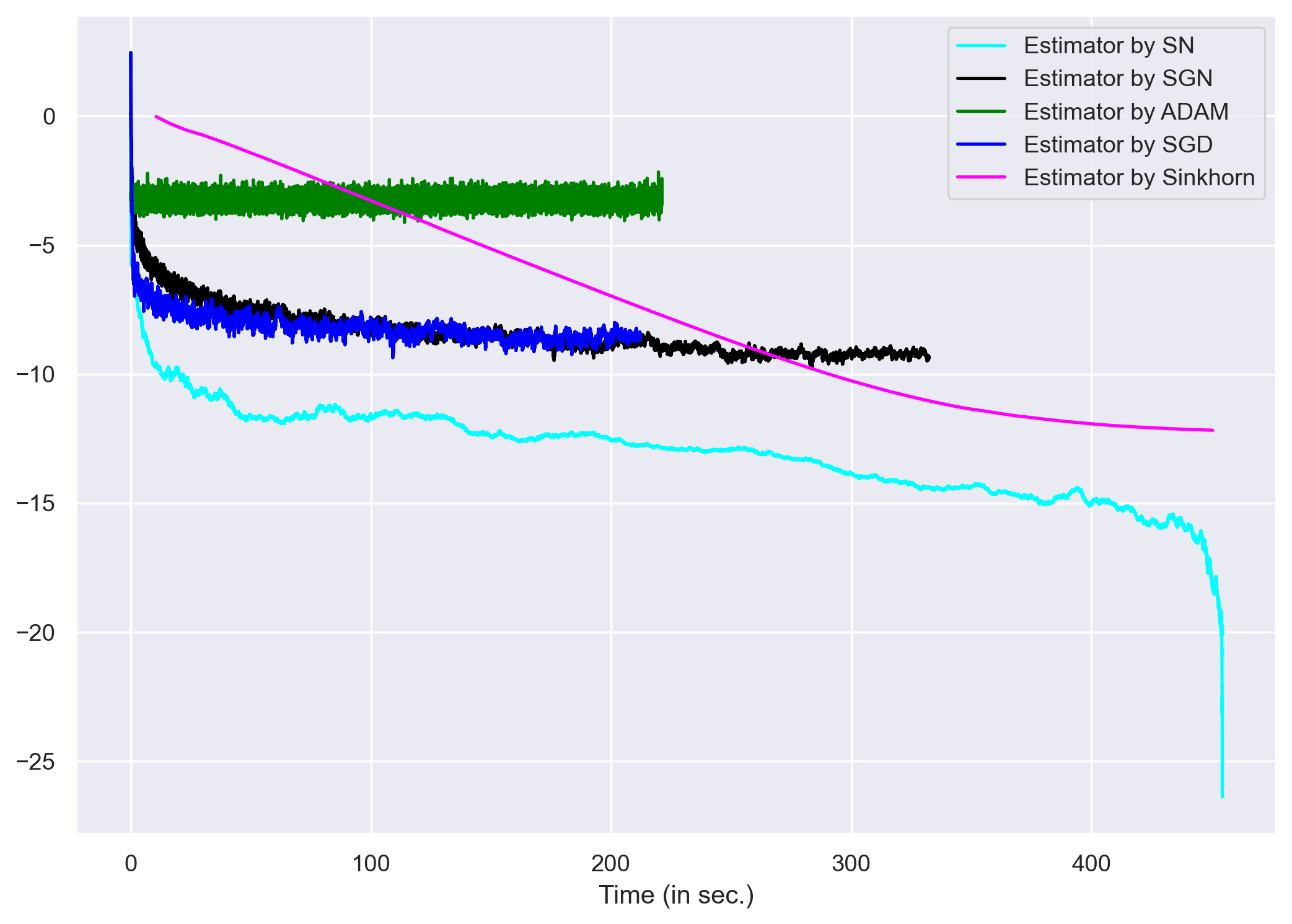}}}
{\subfigure[$d = 5$ and $\varepsilon = 0.005$]{\includegraphics[width=0.45 \textwidth,height=0.3\textwidth]{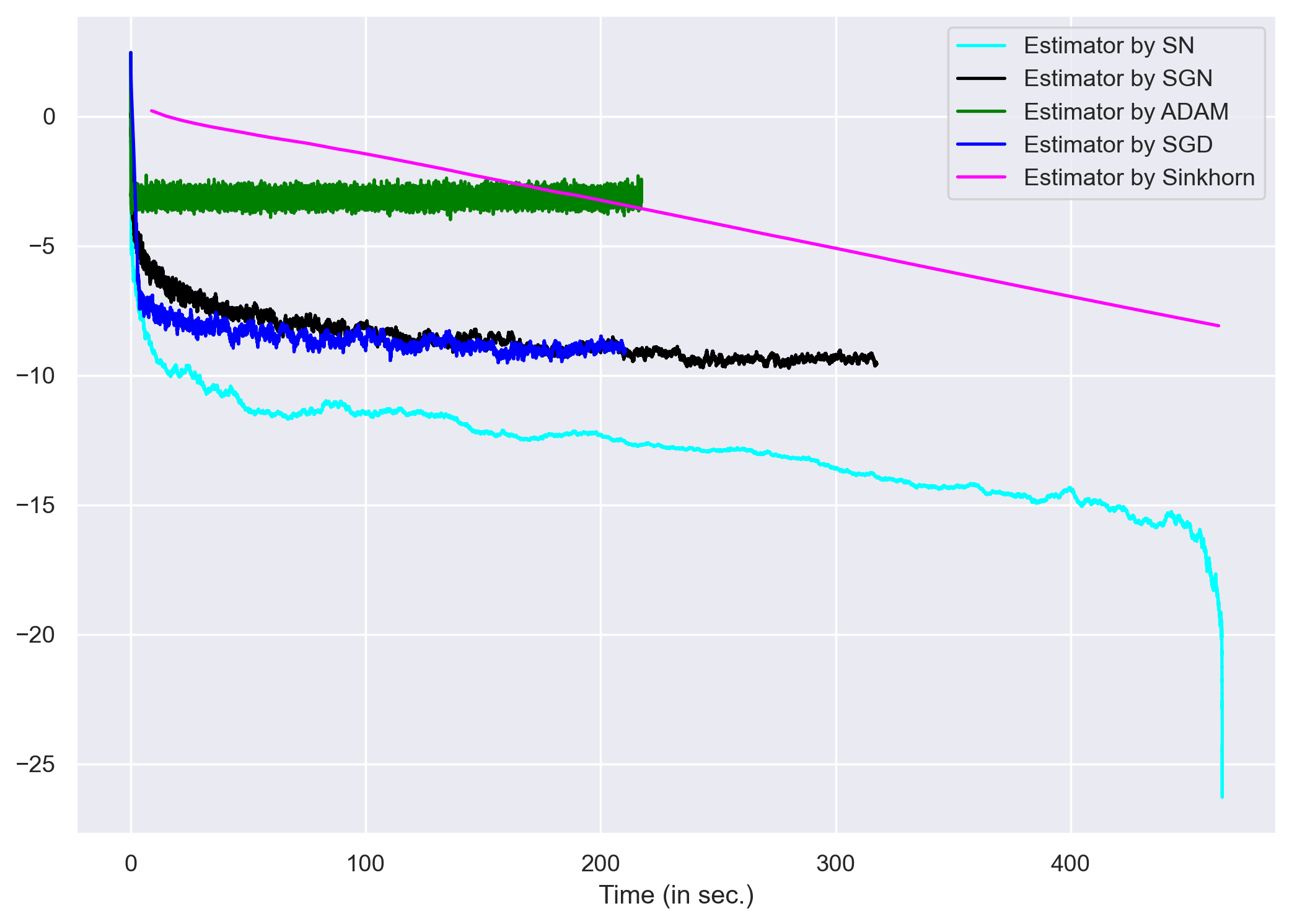}}}
{\subfigure[$d = 10$ and $\varepsilon = 0.01$]{\includegraphics[width=0.45 \textwidth,height=0.3\textwidth]{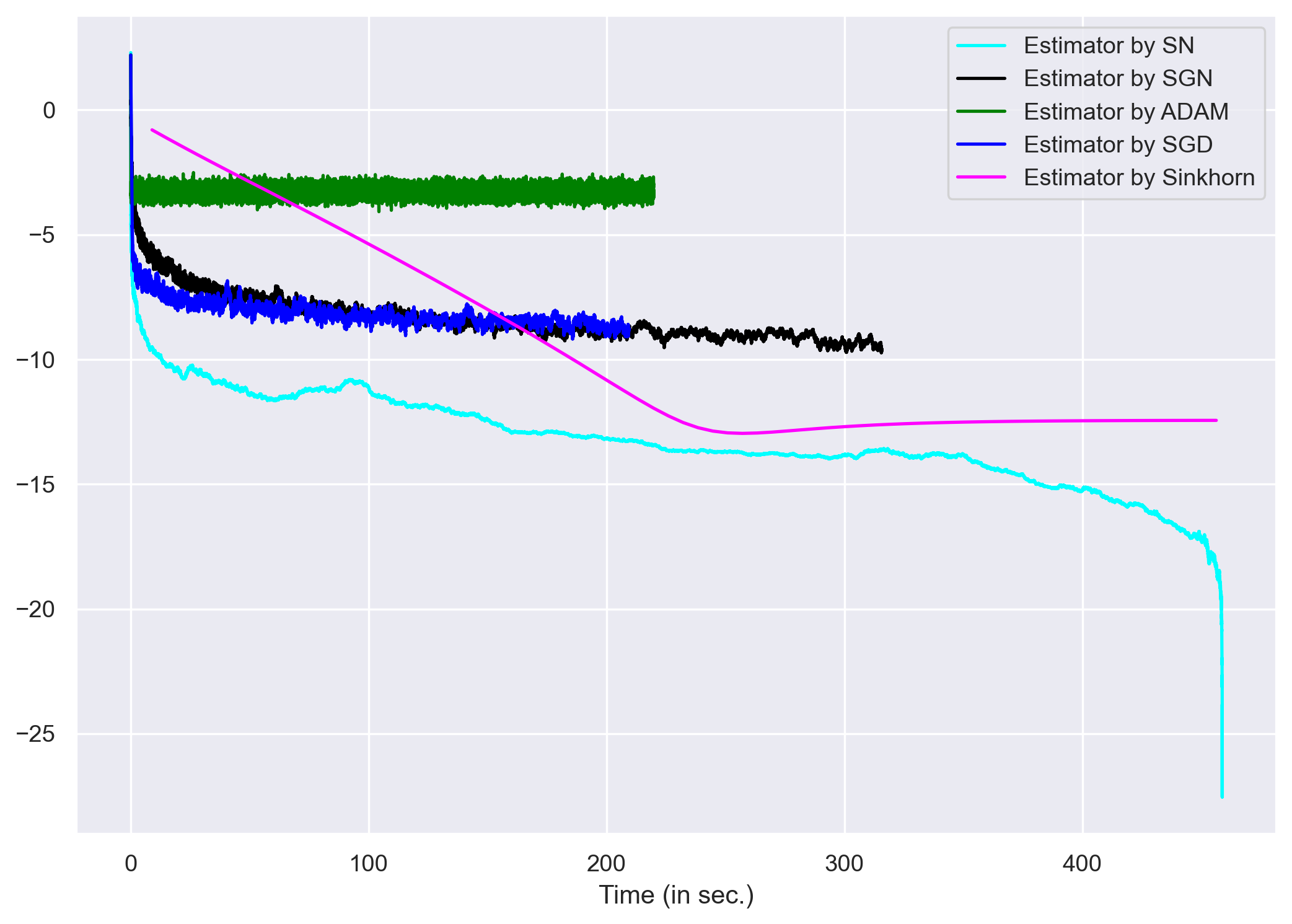}}}
{\subfigure[$d = 10$ and $\varepsilon = 0.005$]{\includegraphics[width=0.45 \textwidth,height=0.3\textwidth]{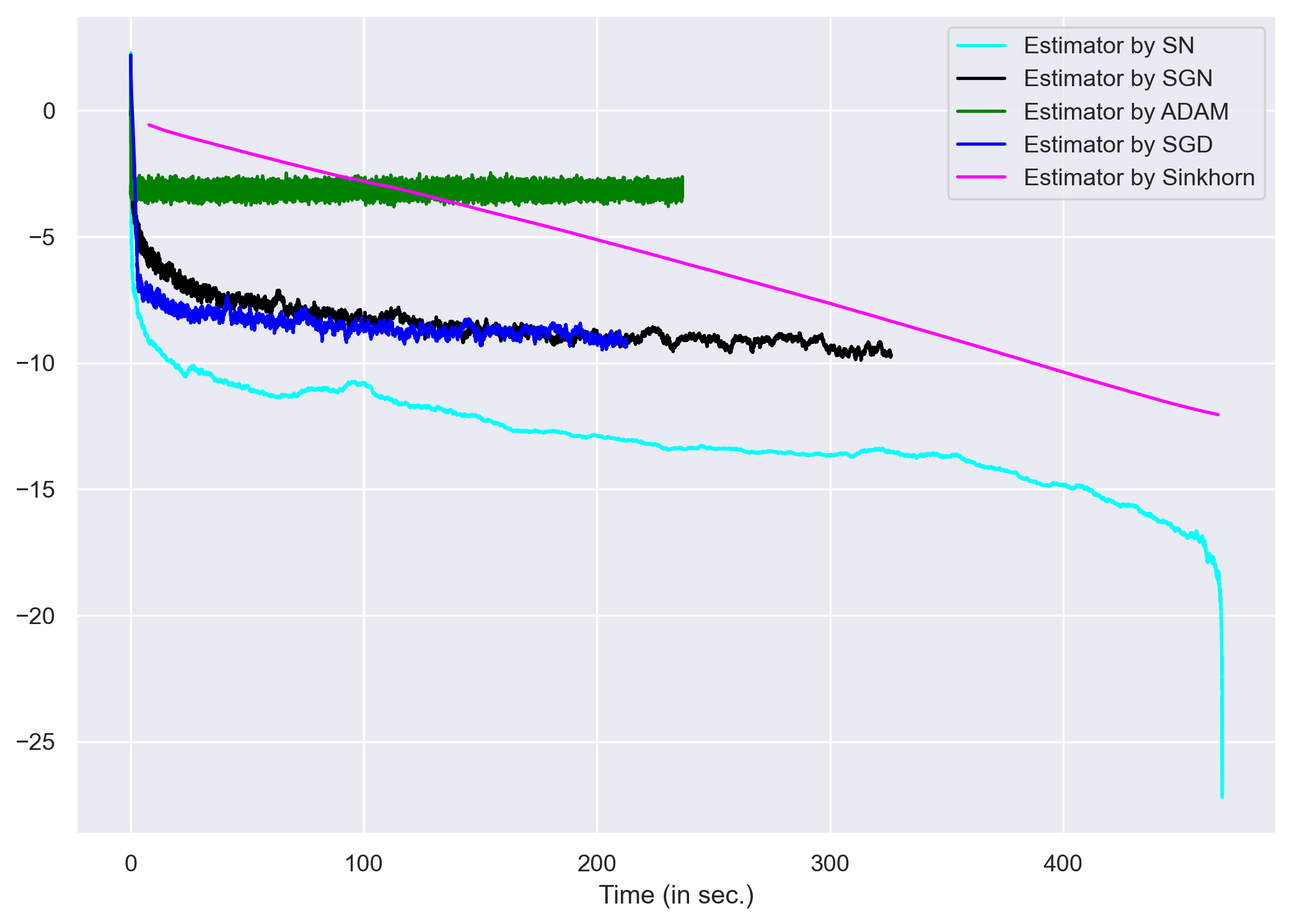}}}
{\subfigure[$d = 50$ and $\varepsilon = 0.01$]{\includegraphics[width=0.45 \textwidth,height=0.3\textwidth]{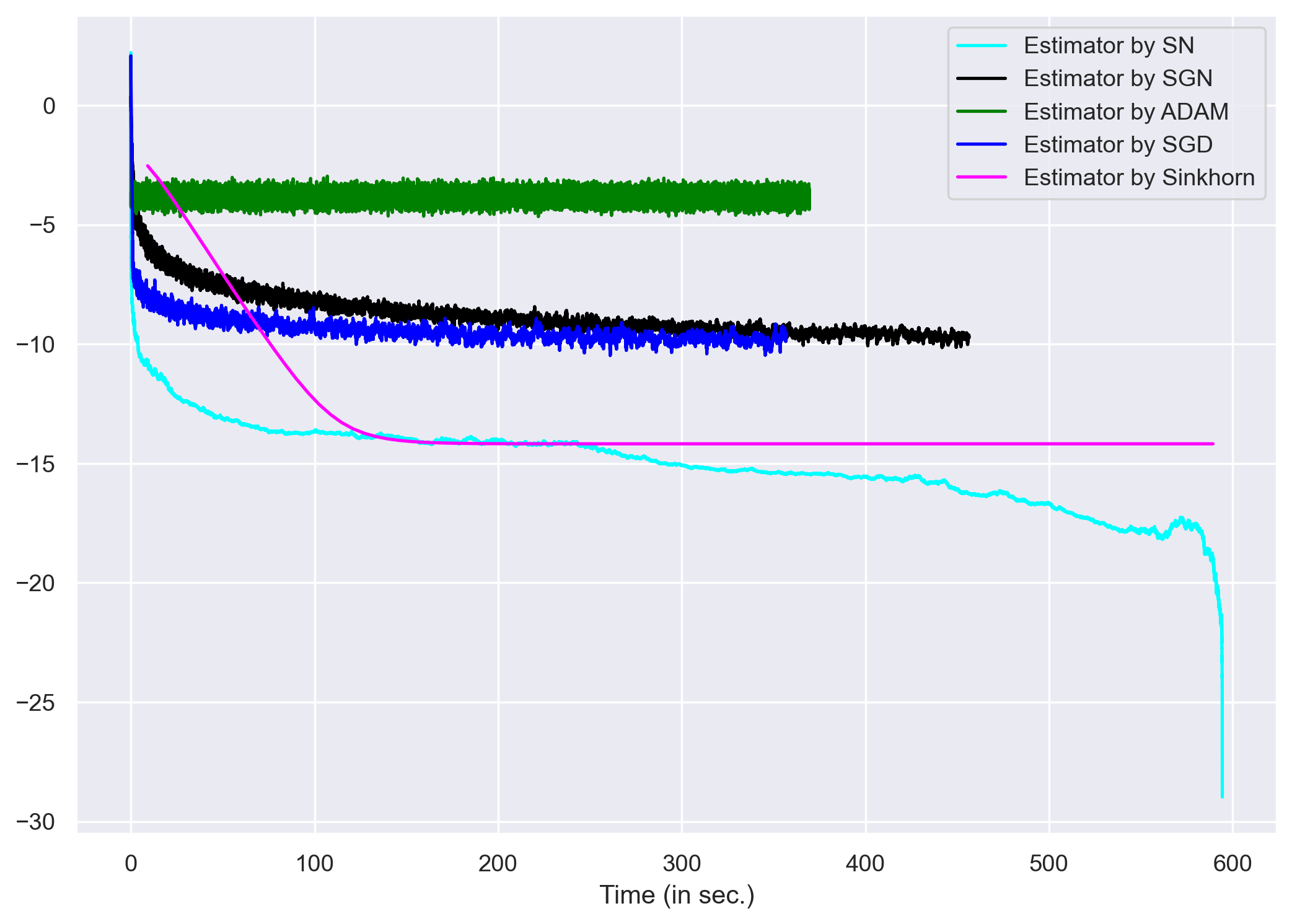}}}
{\subfigure[$d = 50$ and $\varepsilon = 0.005$]{\includegraphics[width=0.45 \textwidth,height=0.3\textwidth]{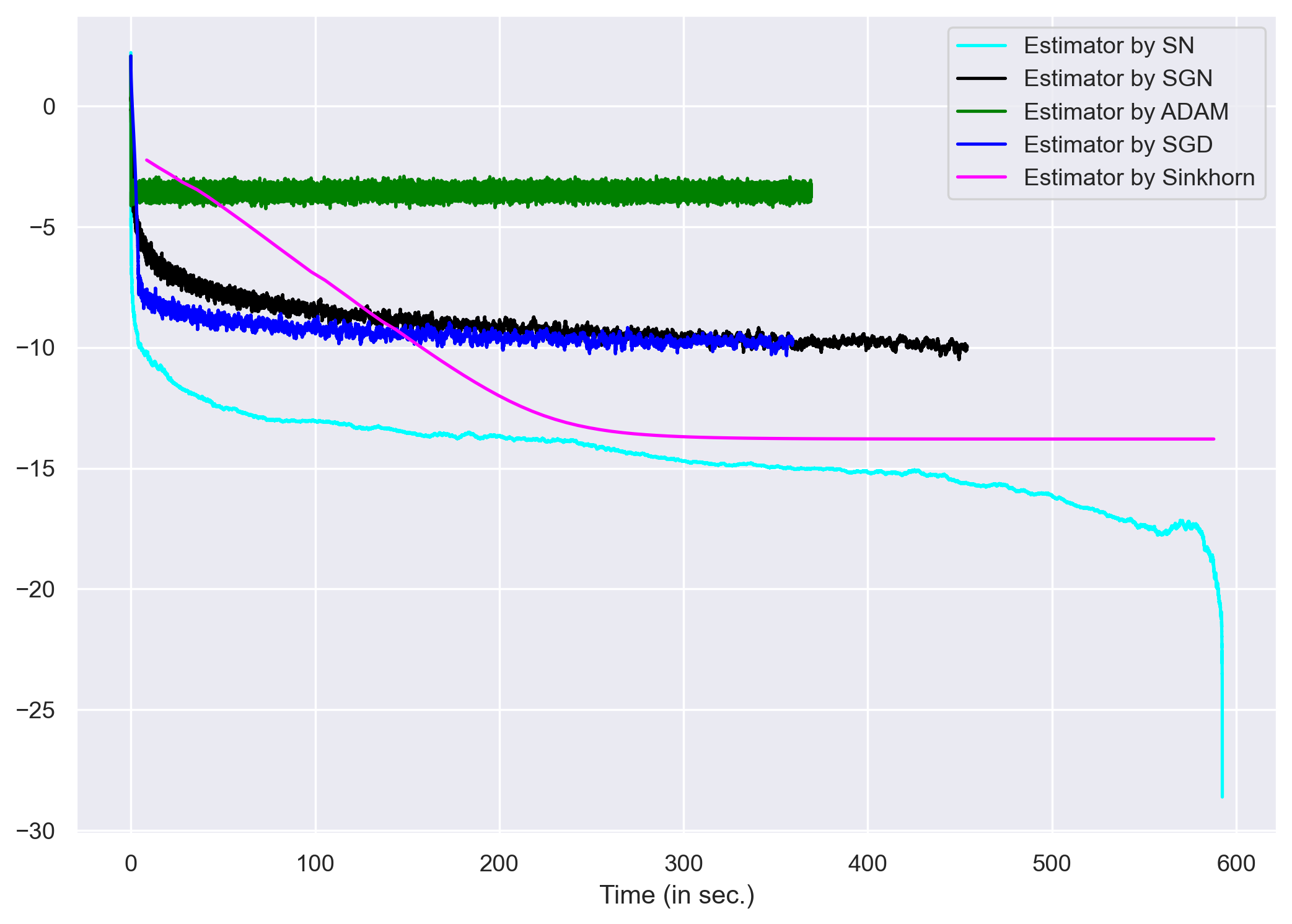}}}

\caption{Semi-discrete setting where $\mu$ is  a mixture of three Gaussian densities,   $J=100$ and $n = n_{\max}= 10^6$ iterations. Excess risk  (in logarithmic scale)  $ \log(    \bigl\| \wh{V}_n -v^\ast \bigr\|^2 )$  (resp.\ metric $ \log( \bigl\|  V_k  -v^\ast \bigr\|^2)$) as a function of the  computational cost of the iterations of  the four stochastic algorithms (resp.\ the Sinkhorn algorithm) for different values of the dimension $d$ and the regularization parameter $\varepsilon$. \label{fig:excess_risk_semidiscrete_Vn_large}}
\end{figure}
}


\section{Properties of the objective  function $H_{\varepsilon}$} \label{sec:useful}

The purpose of this  section is to discuss various keystone properties of the objective function $H_{\varepsilon}$  that are needed to establish our main results.

\subsection{Gradient properties}

Let us first remark that, for any $x \in \XX$, the function $v \mapsto h_\varepsilon(x,v)$, defined by \eqref{Defh}, is twice differentiable.  For a fixed $x \in \XX$, the gradient vector and Hessian matrix of the function $h_\varepsilon$, with respect to its second argument, are given by
 \begin{equation}
 \label{Grad.h}
\nabla_{v} h_{\varepsilon}(x, v) = \pi(x,v) - \nu,
 \end{equation}
and  
\begin{equation}
 \label{Hessian.h}
\nabla^{2}_{v} h_{\varepsilon}(x,v) = \frac{1}{\varepsilon}
\Bigl( \diag(\pi(x,v)) - \pi(x,v) \pi(x,v)^T  \Bigr),
 \end{equation} 
 where the $j^{th}$  component of the vector  $\pi(x,v) \in \R^{J}$ is such that
 \begin{equation*}
 \label{Vectpi}
 \pi_j(x,v) = \Bigl(\sum_{k=1}^{J}\nu_k \exp \Bigl(  \dfrac{v_k - c(x,y_k) }{\varepsilon} \Bigr)\Bigr)^{-1} \nu_j \exp \Bigl( \dfrac{v_j - c(x,y_j) }{\varepsilon} \Bigr).
 \end{equation*}
Consequently, the gradient vector and the Hessian matrix of the function $H_\varepsilon$, defined by \eqref{Semi-dualdisc}, are as follows
 \begin{equation}
 \label{Grad.H}
\nabla H_{\varepsilon}(v) = \E[ \nabla_{v} h_{\varepsilon}(X,v) ]=  \E[ \pi(X,v) ]  -  \nu,
 \end{equation}
 and 
\begin{equation}
\label{Hessian.H}
\nabla^{2} H_{\varepsilon}(v) =  \E[ \nabla_{v}^2 h_{\varepsilon}(X,v) ] = \frac{1}{\varepsilon} \E
 \bigl[\diag(\pi(X,v))  -  \pi(X,v) \pi(X,v)^T \bigr]  .
\end{equation}
Note that the minimizer $v^{\ast}$ satisfies $\nabla H_{\varepsilon}(v^{\ast}) = 0$, leading to
$$
\E[ \pi(X,v^{\ast}) ] = \nu,
$$
which allows us to simplify the expression for the Hessian of $H_{\varepsilon}$ at $v^{\ast}$,
\begin{equation}
\nabla^{2} H_{\varepsilon}(v^{\ast}) = \frac{1}{\varepsilon}\left(\diag(\nu) - \E
 \bigl[ \pi(X,v^{\ast}) \pi(X,v^{\ast})^T \bigr] \right).\label{eq:Hess_vast}
\end{equation}
\noindent
We now discuss some properties of the above gradient vectors and Hessian matrices that will be of interest to study the SGN algorithm.

\subsection{Convexity of $H_{\varepsilon}$ and related properties}

First of all, the baseline remark is that $\nabla^{2}_{v}H_{\varepsilon}(v)$ is a positive semi-definite matrix for any $v \in \R^J$, which entails the convexity of $H_{\varepsilon}$.

\paragraph{Minimizers and rank of the Hessian.} 
It is clear from \eqref{Hessian.H}  that for any $v \in \R^J$,  the smallest eigenvalue of the Hessian matrix $\nabla^{2}_{v}H_{\varepsilon}(v)$ 
associated to the eigenvector $\bv_J$
is equal to zero. Therefore, as indicated in the end of Section \ref{subsec:defOT}, 
for any $t \in \R$, the vector $v^\ast + t  \bv_J$ is also a minimizer of \eqref{DefH}. 
Nevertheless, it is well-known \cite{cuturi} that the minimizer $v^{\ast}$ of \eqref{DefH} is unique up to a scalar translation of its coordinates. 
We shall thus denote by $v^{\ast}$ the minimizer of  \eqref{Semi-dualdisc} satisfying 
$\langle v^{\ast} , \bv_J \rangle = 0$.  It means that
$v^{\ast}$ belongs to $\langle \bv_J \rangle^\perp$, and that the function $H_\varepsilon$ admits a unique minimizer over the $J-1$ dimensional subspace $\langle \bv_J \rangle^\perp$. However, as already shown in \cite{NIPS2016_6566} and further discussed in  \cite[Section 3.3]{Stochastic_Bigot_Bercu}, the objective function $H_{\varepsilon}$    is {\it not strongly convex},
even by restricting the maximization problem  \eqref{Semi-dualdisc}   to the subspace $\langle \bv_J \rangle^\perp$ 
since it may be shown that $v \mapsto H_{\varepsilon}(v)$ may have some \tcb{vanishing curvature, leading to a  flat landscape, \textit{i.e.} to eigenvalues of the Hessian matrix that are arbitrarily close to $0$  for large values of $\|v\|$ in  $\langle \bv_J \rangle^\perp$}.

  Moreover, for any $(x,v) \in \XX \times \R^J$, it follows from \cite[Lemma A.1]{Stochastic_Bigot_Bercu} that the matrices $\nabla^{2}_{v} h_{\varepsilon}(x,v)$ and $\nabla^{2}_{v} H_{\varepsilon}(v)$  are of rank $J-1$, and therefore, all their eigenvectors associated to non-zero eigenvalues belong to $\langle \bv_J \rangle^\perp$. Finally, one also has that $\nabla_{v} h_{\varepsilon}(x, v) \in \langle \bv_J \rangle^\perp$  for any $(x,v) \in \XX \times \R^J$.

\paragraph{Useful upper and lower bounds.} We conclude this section by stating a few inequalities that we repeatedly use in the proofs of our main results. Since $\nu$ and $\pi(x,v)$ are vectors with positive entries that sum up to one, it follows from \eqref{Grad.h} that for any $(x,v) \in \XX \times \R^J$,
\begin{equation}
\|\nabla_v h_{\varepsilon}(x, v)\| \leq \|\nu\| + \| \pi(x,v) \| \leq 2, \label{eq:boundgradient}
\end{equation}
 and that the gradient of $H_{\varepsilon}$ is always bounded for any $v \in \R^J$,
\begin{equation}
\|\nabla H_{\varepsilon}(v)\|  \leq 2. \label{eq:boundHepsgradient}
\end{equation}
Moreover,  thanks to the property that
$$
\lambda_{\max}\Bigl( \frac{1}{\varepsilon} \bigl( \diag(\pi(x,v))  -  \pi(x,v) \pi(x,v)^T \bigr)\Bigr) \leq   \frac{1}{\varepsilon}    \Tr \bigl(  \diag(\pi(x,v))    \bigr) = \frac{1}{\varepsilon},   
$$
we obtain that for any $(x,v) \in \XX \times \R^J$,
\begin{equation} 
\lambda_{\max}\left(\nabla^{2}_{v} h_{\varepsilon}(x,v) \right) \leq \frac{1}{\varepsilon} \hspace{1cm} \text{and} \hspace{1cm} \lambda_{\max}\left(\nabla^{2}_{v} H_{\varepsilon}(v) \right) \leq \frac{1}{\varepsilon}. \label{eq:boundlambdamax}
\end{equation}
Finally, by  \cite[Lemma A.1]{Stochastic_Bigot_Bercu}, the second smallest eigenvalue of $\nabla^{2} H_{\varepsilon}(v^\ast)$ is positive, and one has that
\begin{equation} 
\lambda_{\min}\left( \nabla^{2} H_{\varepsilon}(v^\ast) \right) = \min_{v \in  \langle \bv_J \rangle^\perp}
\Bigl \{\frac{v^T \nabla^{2} H_{\varepsilon}(v^\ast) v }{\|v\|^2} \Bigr \}
 \geq \frac{1}{\varepsilon} \min(\nu).\label{eq:boundlambdamin}
\end{equation}

\subsection{Generalized self-concordance for regularized semi-discrete OT}

Let us now introduce the so-called notion of generalized self-concordance proposed in Bach \cite{Bach14} for the purpose of obtaining fast rates of convergence for stochastic algorithms with non-strongly convex objective functions.  Generalized self-concordance has been shown to hold for regularized semi-discrete OT in \cite{Stochastic_Bigot_Bercu}, and we discuss below its implications of some key properties  for the analysis of the SGN algorithm studied in this paper. 
To this end, for any $v \in \langle \bv_J \rangle^\perp$  and for all  $t$ in the interval $[0,1]$, we denote
$v_t=v^\ast +t(v-v^\ast)$, and we define the function $\varphi$, for all $t \in [0,1]$, as
$$
\varphi(t) = H_\varepsilon(v_t).
$$
The second-order Taylor expansion of $\varphi$ with integral remainder is given by
\begin{equation}
\label{PTaylor1}
\varphi(1)=\varphi(0)+\varphi^\prime(0) - \int_0^1 (t-1) \varphi^{\prime \prime}(t)\,dt.
\end{equation}
Using that  $\varphi(1)=H_\varepsilon(v)$, $\varphi(0)=H_\varepsilon(v^\ast)$ and
$\varphi^\prime(0)=\langle  v - v^\ast , \nabla H_{\varepsilon}(v^\ast) \rangle=0$, it has been first remarked in \cite{Stochastic_Bigot_Bercu} that inequality \eqref{eq:boundlambdamax} implies that
\begin{equation}
\label{TAYLOR1}
 H_{\varepsilon}(v)-H_{\varepsilon}(v^\ast)  \leq \frac{1}{2 \varepsilon} \| v-v^{\ast}\|^2 
\end{equation}
Moreover, it is shown in the proof of \cite[Lemma A.2]{Stochastic_Bigot_Bercu} that the following inequality holds
\begin{equation}
\label{Selfcvarphi}
\bigl| \varphi^{\prime \prime \prime}(t) \bigr| \leq  \frac{\sqrt{2}}{\varepsilon} \varphi^{\prime \prime}(t) 
\| v-v^{\ast} \|.
\end{equation}
It means that the function $\varphi$ satisfies the so-called generalized 
self-concordance property with constant $s_\varepsilon=\sqrt{2}/\varepsilon$ as defined in Appendix B of \cite{Bach14}. As a consequence of inequality \eqref{Selfcvarphi} and thanks to the arguments in the proof of \cite[Lemma A.2]{Stochastic_Bigot_Bercu}, the error of linearizing the gradient $\nabla \He(v) \approx \nabla^2 \He(v^\ast) (v-v^\ast) $ is controlled as follows,
\begin{equation}
\| \nabla \He(v) -\nabla^2 \He(v^\ast) (v-v^\ast) \| \leq 2 s_\varepsilon \| v-v^\ast \|^2. \label{eq:lingrad1}
\end{equation}

\noindent
Moreover, generalized self-concordance also implies the following result  (which is a consequence of the arguments in the proof of Lemma A.2 in \cite{Stochastic_Bigot_Bercu}) that may be interpreted as a local strong convexity property of the function $ H_{\varepsilon}$ in the neighborhood of $v^{\ast}$.

\begin{lem} \label{lem:LocalConvexity}
For any $v \in  \langle \bv_J \rangle^\perp$, we have
\begin{equation} \label{eq:LocalConvexity}
\langle   \nabla H_{\varepsilon}(v) ,v-v^{\ast}  \rangle \geq 
{\displaystyle  \frac{1 - \exp ( - \delta(v) ) }{\delta(v)}  (v - v^{\ast})^T  \nabla^2 H_{\varepsilon}(v^{\ast}) (v - v^{\ast})},
\end{equation}
where $\delta(v) =  s_\varepsilon  \|v - v^{\ast}\|$.
\end{lem}

\noindent
Finally, if we now consider  the matrix-valued function $G_{\varepsilon}(v)$ 
introduced in equation \eqref{eq:defG}, we have the following result which can be interpreted as a local Lipschitz property of $G_{\varepsilon}(v)$ around $v = v^{\ast}$.
The proof of this lemma is postponed to Appendix A.

\begin{lem} \label{lem:LipGv}
For any $v \in  \langle \bv_J \rangle^\perp$, we have that 
\begin{equation}
- \frac{4}{\varepsilon} \|v  - v^\ast\| \Id \leq G_{\varepsilon}(v) - G_{\varepsilon}(v^{\ast}) \leq \frac{4}{\varepsilon} \|v  - v^\ast\| \Id, \label{ineq:Gstar}
\end{equation}
in the sense of partial ordering between positive semi-definite matrices.
\end{lem}


\section{Proofs of the main results} \label{sec:prop}


This section contains the proofs of our main results that are stated in Section \ref{sec:main}. 
\tcb{Our results are based on previous important contributions on self-concordance functions (see \textit{e.g.} \cite{Bach14,Stochastic_Bigot_Bercu}), regularization of second order algorithms \cite{Bercu_Godichon_Portier2020} and on the Kurdyka-\L ojasiewicz inequality adapted to stochastic algorithms \cite{gadat:hal-01623986}. 
More specifically, almost sure convergence and almost sure convergence rates crucially depend on the adaptive property \eqref{ineq:eiglower}, which induces a contraction rate of the sequence $G_{\varepsilon}^{1/2}(v^\ast)\bigl(\wh{V}_n 
- v^\ast\bigr)$ (see Equation \eqref{eq:decVhatn3} below). The non-asymptotic study is then based on both the self-concordance property, stated in Lemma \ref{lem:LocalConvexity} and on the KL inequality stated in Proposition \ref{prop:KL}. The combination of these two properties is an essential novelty brought by our work in order to  build a key Lyapunov function in Equation \eqref{def:Phip}. 
We emphasize that to obtain the results stated below, we have derived quantitative computations that are specific  to the regularized OT problem. In particular, if the use of the KL inequality is borrowed from \cite{gadat:hal-01623986}, the exact values of the constant $m_\epsilon$ and 
$M$ used in Proposition \ref{prop:KL} crucially depend on the self-concordance property of the regularized OT problem.
}

\subsection{Keystone property}
We start with the proof of inequality \eqref{ineq:eiglower} that states the adaptivity of the SGN algorithm to the local geometry of $H_{\varepsilon}$.

\begin{proof}[Proof of Proposition \ref{prop:keystone}]

First of all, one can remark that $G_{\varepsilon}(v^\ast)$ is a positive semi-definite matrix whose smallest eigenvalue is equal to zero and associated to the eigenvector $\bv_J$.  Thus, all the eigenvectors of $G_{\varepsilon}(v^\ast)$ associated to non-zero eigenvalues belong to $\langle \bv_J \rangle^\perp$. 
We already saw from \eqref{eq:Hess_vast} that
$$
\nabla^2 \He(v^\ast)= \frac{1}{\varepsilon}\left(\diag(\nu) - \E
 \bigl[ \pi(X,v^{\ast}) \pi(X,v^{\ast})^T \bigr] \right).
$$
Consequently, it follows from \eqref{eq:defG} and \eqref{Grad.h} that
$$
G_{\varepsilon}(v^\ast)  = 
\E \bigl[ \pi(X,v^{\ast}) \pi(X,v^{\ast})^T \bigr]  - \nu\nu^T = \diag(\nu) - \nu\nu^T - \varepsilon \nabla^2 \He(v^\ast),
$$
which implies that
$$
 G_{\varepsilon}(v^\ast) = \nabla^2 \He(v^\ast) + \Sigma_{\varepsilon}^\ast
$$
where
$$
\Sigma_{\varepsilon}^\ast =   \diag(\nu) - \nu\nu^T - (1+\varepsilon) \nabla^2 \He(v^\ast).
$$
On the one hand, it is easy to see that $ \bv_J^T \Sigma_{\varepsilon}^\ast   \bv_J= 0$. On the other hand, we deduce from inequality \eqref{eq:boundlambdamin} that
for all  $v \in \langle \bv_J \rangle^\perp$,
$$
v^T \Sigma_{\varepsilon}^\ast v \leq   \max(\nu) - \Bigl(\frac{1+\varepsilon}{\varepsilon} \Bigr)  \min(\nu).
$$
Finally, condition \eqref{eq:condeps} on  the regularization parameter $\varepsilon$ leads to $G_{\varepsilon}(v^\ast) \leq \nabla^2 H_{\varepsilon}(v^\ast)$, which completes the proof of Proposition \ref{prop:keystone}.
\end{proof}

\subsection{Proofs of the most sure convergence results}\label{subsec:asproof}

\begin{proof}[Proof of Theorem \ref{theo:asconvVn}]
In what follows,  we borrow some arguments from the proof of \cite[Theorem 4.1]{Bercu_Godichon_Portier2020} to establish the almost sure 
convergence of the regularized versions of the  SGN algorithm as an application of the Robbins-Siegmund Theorem \cite{robbins1971convergence}. \\
\newline
$\bullet$
We already saw that for all $n \geq 0$, $\wh{V}_n$  belongs to  $\langle \bv_J \rangle^\perp$. We clearly have from \eqref {Grad.h} that for all $n \geq 0$,
$\nabla_v \he(X_{n+1}, \wh{V}_n)$ also belong to $\langle \bv_J \rangle^\perp$.
Hence,  we have from \eqref{eq:SNgen} that for all $n \geq 0$,
\begin{equation}
\wh{V}_{n+1} =   \wh{V}_n -  n^{\alpha} \proj S_n^{-1} \proj \bigl(\nabla \He(\wh{V}_n) + \varepsilon_{n+1} \bigl),    \label{SNmar}
\end{equation}
where the martingale increment $\varepsilon_{n+1}$ is given by
\begin{equation*}
\varepsilon_{n+1} = \nabla_v \he(X_{n+1}, \wh{V}_n) - \mathbb{E}[ \nabla (\he(X_{n+1}, \wh{V}_n))  |\cF_n ] =\nabla_v \he(X_{n+1}, \wh{V}_n) - 
\nabla \He(\wh{V}_n)
\label{Defvarepsilon}
\end{equation*}
with $\cF_n=\sigma(X_1,\ldots,X_n)$. Moreover, it follows from the Taylor-Lagrange formula that
\begin{equation} \label{Taylorexp}
    \He(\wh{V}_{n+1}) = \He(\wh{V}_n) + \nabla \He(\wh{V}_n)^T(\wh{V}_{n+1} - \wh{V}_n) + \frac{1}{2}(\wh{V}_{n+1} - \wh{V}_n)^T\nabla^2 \He(\xi_{n+1})(\wh{V}_{n+1} - \wh{V}_n),
\end{equation}
where $\xi_{n+1}= \wh{V}_n+ t( \wh{V}_{n+1} - \wh{V}_n )$ with $t \in ]0,1[$.
Consequently, we deduce from  \eqref{SNmar} and \eqref{Taylorexp} that for all $n \geq 0$,
\begin{eqnarray*}
    \He(\wh{V}_{n+1}) & = & \He(\wh{V}_n) - n^{\alpha}  \bigl\langle \nabla \He(\wh{V}_n), \proj S_n^{-1} \proj \bigl(\nabla \He(\wh{V}_n) + \varepsilon_{n+1} \bigr) \bigr\rangle      \label{Taylor} \\ 
    &   & \hspace{-2cm}+ \frac{n^{2 \alpha} }{2}\bigl(\proj S_n^{-1} \proj \bigl(\nabla \He(\wh{V}_n) + \varepsilon_{n+1}\bigr)\bigr)^{T}\nabla^2 \He(\xi_{n+1})\bigl(\proj S_n^{-1} \proj \bigl(\nabla \He(\wh{V}_n) + \varepsilon_{n+1}\bigr)\bigr). \nonumber
\end{eqnarray*}
Taking the conditional expectation with respect to $\cF_n$ on both sides of the previous equality, we obtain that for all $n \geq 0$,
\begin{eqnarray}
    \E[\He(\wh{V}_{n+1}) | \cF_n] &  = &\He(\wh{V}_n)  - n^{\alpha}  \nabla \He(\wh{V}_n)^T \proj S_n^{-1} \proj \nabla\He(\wh{V}_n) \nonumber \\ & & 
    \hspace{-2cm}+ \frac{n^{2\alpha} }{2}\E\Bigl[\nabla_v \he(X_{n+1}, \wh{V}_n)^{T} 
 \proj S_n^{-1} \proj \nabla^2 \He(\xi_{n+1}) \proj S_n^{-1} \proj \nabla_v \he(X_{n+1}, \wh{V}_n)  \big| \cF_n\Bigr] \nonumber \\
    &  = &\He(\wh{V}_n)  - n^{\alpha}  \nabla \He(\wh{V}_n)^T \proj S_n^{-1} \proj \nabla\He(\wh{V}_n) \nonumber \\ & & 
    \hspace{-2cm} + \frac{n^{2 \alpha} }{2}\E\Bigl[\nabla_v \he(X_{n+1}, \wh{V}_n)^{T} S_n^{-1}  \nabla^2 \He(\xi_{n+1})  S_n^{-1}  \nabla_v \he(X_{n+1}, \wh{V}_n)  \big| \cF_n\Bigr] \label{eq:Taylor2} 
\end{eqnarray}
using the elementary fact that $\proj \nabla_v \he(X_{n+1}, \wh{V}_n) = 
\nabla_v \he(X_{n+1}, \wh{V}_n)$ 
as well as $\nabla^2 \He(\xi_{n+1}) v_J=0$ which implies that $\proj \nabla^2 \He(\xi_{n+1}) \proj = \nabla^2 \He(\xi_{n+1})$.
On the one hand, we have from inequality \eqref{eq:boundgradient} that $\|\nabla_v \he (X_{n+1}, \wh{V}_n)\| \leq 2$. On the other hand,
it follows from inequality \eqref{eq:boundlambdamax} that $\lambda_{\max}\bigl(\nabla^2 \He(\xi_{n+1})\bigr) \leq 1/\varepsilon$. Therefore, we deduce from \eqref{eq:Taylor2} that
for all $n \geq 0$,
\begin{equation*}
   \E[\He(\wh{V}_{n+1}) | \cF_n]  \leq \He(\wh{V}_n) +A_n-B_n \hspace{1cm} \text{a.s.}
   \label{eq:Taylor4}
\end{equation*}
where the two positive random variables $A_n$ and $B_n$ are given by
$$
A_n= \frac{2 n^{2 \alpha} }{\varepsilon (\lambda_{\min}\bigl(S_n)\bigr)^2} 
$$
and $B_n= n^{\alpha}  \nabla \He(\wh{V}_n)^T \proj S_n^{-1} \proj \nabla\He(\wh{V}_n)$. Our purpose is now to show that
$$
\sum_{n=1}^\infty  A_n < +\infty \hspace{1cm} \text{a.s.}
$$
We already saw from \eqref{eq:SnRSGN} that for all $n \geq 1$,
\begin{equation*}
\label{eq:RSnRSGN}
S_n =  \Id +\sum\limits_{k=1}^{n}  \nabla_v \he(X_{k}, \wh{V}_{k-1}) \nabla_v \he(X_{k}, \wh{V}_{k-1})^T + R_n
\end{equation*}
where
$$
R_n =  \sum\limits_{k=1}^{n} \gamma   \Bigl( 1+  \bigl\lfloor \frac{k}{J} \bigr\rfloor \Bigr)^{-\beta}  Z_k Z_k^T.
$$
We clearly have 
$
 \lambda_{\min}(S_n) \geq  \lambda_{\min}(R_n).
$
Let $p_n$ be the largest integer such that $p_n J \leq n$. One can remark 
that
\begin{equation}
\label{eq:DecRn}
R_n = \Bigl(\sum\limits_{m=1}^{p_n}  m^{-\beta} \Bigr)  \gamma  \diag(\nu) +  \sum\limits_{k=p_n+1}^{n} \gamma   \Bigl( 1+  
\bigl\lfloor \frac{k}{J} \bigr\rfloor \Bigr)^{-\beta}  Z_k Z_k^T,
\end{equation}
which implies that
$$
 \lambda_{\min}(S_n) \geq \gamma \min(\nu)  \Bigl(\sum\limits_{m=1}^{p_n}  m^{-\beta} \Bigr).
$$
However, 
for any $0< \beta < 1/2$ and for all $p_n \geq 4$,
$$
\sum_{m=1}^{p_n}  \frac{1}{m^\beta} \geq \frac{p_n^{1-\beta}}{2(1-\beta)}.
$$
Consequently, using that $p_n \geq n J^{-1} -1 $, we obtain that
$$
\sum_{n=1}^\infty  A_n  \leq  \frac{8(1-\beta)^2}{\varepsilon (\gamma \min(\nu))^2} \sum_{n=1}^\infty 
\frac{n^{2 \alpha}}{p_n^{2(1-\beta)}}  \leq  \frac{16(1-\beta)^2J^{2(1-\beta)}}{\varepsilon (\gamma 
\min(\nu))^2} \sum_{n=1}^\infty 
\frac{1}{n^{2(1-\alpha-\beta)}}
< +\infty \hspace{1cm} \text{a.s.}
$$
since the assumption $0 <\alpha +\beta < 1/2$  implies that $2(1-\alpha-\beta) > 1$. Therefore, we can
apply the Robbins-Siegmund Theorem \cite{robbins1971convergence}
to conclude that the sequence $(\He(\wh{V}_n))$ converges almost surely to a finite random variable and that the series
$$
\sum_{n=1}^\infty B_n < +\infty \hspace{1cm} \text{a.s.}
$$
leading to
\begin{equation}
\label{eq:series}
\sum_{n=1}^\infty n^{\alpha} \frac{\|\nabla H_{\varepsilon}(\wh{V}_n)\|^2}{\lambda_{\max}(S_n)} < + \infty   \hspace{1cm} \text{a.s.}
\end{equation}
One can very from inequality \eqref{eq:boundgradient} and \eqref{eq:SnRSGN} that for all $n \geq 1$,
$$
\lambda_{\max}(S_n)  \leq 1 + 4n + \gamma \max(\nu)  \sum_{k=1}^n   \Bigl( 1+  \bigl\lfloor \frac{k}{J} \bigr\rfloor \Bigr)^{-\beta} \leq 1 + (4 
+ \gamma \max(\nu)) n.
$$
Since $\alpha \geq 0$, it implies that
\begin{equation}\label{eq:series2}
  \sum_{n=1}^\infty \frac{n^{\alpha}}{\lambda_{\max}(  S_n )}=+\infty \hspace{1cm} \text{a.s.}
\end{equation}
The rest of the proof proceeds from standard arguments combining \eqref{eq:series} and  \eqref{eq:series2}. Let
$$
H_{\varepsilon,\infty}= \lim_{n \to + \infty} H_{\varepsilon}(\wh{V}_n) 
\hspace{1cm} \text{a.s.}
$$
and assume by contradiction that $H_{\varepsilon,\infty} > H_{\varepsilon}(v^\ast)$ where
$$
H_{\varepsilon}(v^\ast)= \min_{v \in \langle \bv_J \rangle^\perp} H_{\varepsilon}(v).
$$
Since $H_{\varepsilon}$ is a convex function with a unique minimizer $v^\ast$ on $\langle \bv_J \rangle^\perp$, we necessarily have
$$\lim_{\|v\| \to + \infty} H_{\varepsilon}(v) = + \infty.$$
It means that $(\wh{V}_n)$ is almost surely bounded   since $H_{\varepsilon,\infty}$ is finite. Therefore, we can find a compact set $K$ such that $v^{\ast} = \arg \min_{v \in \langle \bv_J \rangle^\perp} H_{\varepsilon}(v) \notin K$ and $\wh{V}_n  \in K$ for all $n$ large enough. Using 
the continuity of  $\|\nabla  H_{\varepsilon}\|$ and the compactness of $K$, we conclude that $\|\nabla  H_{\varepsilon}\|$ attains its lower bound, which is strictly positive on $K$. It ensures the existence of a constant $c > 0$, such that, for all $n$ large enough,
$$
\|\nabla H_{\varepsilon}(\wh{V}_n )\|\ge c >0.
$$
The above lower bound associated with \eqref{eq:series} and  \eqref{eq:series2} yields a contradiction. Hence, we can conclude that 
$$
\lim_{n \to + \infty} \| \nabla H_{\varepsilon}(\wh{V}_n) \|=0 \hspace{1cm} \text{a.s.}
$$
It clearly implies that equation \eqref{eq:convasVn} holds true
since $(\wh{V}_n)$ is a bounded sequence with a unique adherence point $v^{\ast}$. \\
\newline
$\bullet$
It now remains  to investigate the almost sure convergence of the matrix $\overline{S}_n$.  We observe from equation  \eqref{eq:SnRSGN} that
$S_n$ can be splitted into two terms, 
\begin{equation}
\label{eq:SplitSn}
S_n=M_n + \Sigma_n
\end{equation}
with
\begin{eqnarray*}
M_n  =   \sum_{k=1}^n \Phi_k \Phi_k^T -G_{\varepsilon}( \wh{V}_{k-1}) \hspace{1cm} \text{and} \hspace{1cm}
\Sigma_n  =  \Id + \sum_{k=1}^n G_{\varepsilon}( \wh{V}_{k-1}) +R_n,
\end{eqnarray*}
where the vector $\Phi_k$ stands for $\Phi_k =\nabla_v \he(X_{k}, \wh{V}_{k-1})$. 
Using the  assumption $0 < \beta < 1/2$, we have from \eqref{eq:DecRn} that
$$
\lim_{n \to + \infty}   \frac{1}{n} \bigl(\Id  + R_n \bigr)  = 0.
$$
Moreover, it follows from \eqref{Grad.H} that $G_{\varepsilon}$ is a continuous function from
$\R^J$ to $\R^{J \times J}$. Consequently, we 
deduce from convergence \eqref{eq:convasVn} together with the Cesaro mean 
convergence theorem that
$$
\lim_{n \to + \infty} \frac{1}{n}\sum\limits_{k=1}^n G_{\varepsilon}( \wh{V}_{k-1})  = G_{\varepsilon}( v^{\ast}) \hspace{1cm} \text{a.s.}
$$
which implies that
\begin{equation}
\label{eq:ascvgSigman}
\lim_{n \to + \infty} \frac{1}{n}\Sigma_n  =  G_{\varepsilon}( v^{\ast}) \hspace{1cm} \text{a.s.}
\end{equation}
Hereafter, we focus our attention on the first term $M_n$ in the right-hand side of \eqref{eq:SplitSn}. For any $u \in \R^J$, let
$$
M_n(u)=u^TM_nu= \sum_{k=1}^n \xi_k(u)  
$$
where, for all $n \geq 1$, $\xi_n(u)=u^T\bigl(\Phi_n \Phi_n^T -G_{\varepsilon}( \wh{V}_{n-1})\bigr)u$.
It follows from \eqref{eq:defG} that for all $n \geq 1$,
$\E[\Phi_{n+1}\Phi_{n+1}^T | \cF_n] = G_{\varepsilon}( \wh{V}_{n})$. Hence, for all $n \geq 1$,
$\E[\xi_{n+1}(u) | \cF_n] =0$. Furthermore, we obtain from \eqref{eq:boundgradient} and \eqref{eq:defG} that
for all $n \geq 1$, $\E[\xi_{n+1}^2(u) | \cF_n] \leq 16 || u ||^2$.
Consequently, $(M_n(u))$ is a locally square-integrable martingale with predictable quadratic variation satisfying
$$
\langle M(u) \rangle_n = \sum_{k=1}^n \E[\xi_k^2(u) | \cF_{k-1}] \leq 
16n  ||u||^4.
$$
We deduce from the strong law of large numbers for martingales 
given (e.g.\ by Theorem 1.3.24 in \cite{Duflo1997}) that
\begin{equation*}
\lim_{n \to + \infty} \frac{1}{n}M_n(u)  = 0 \hspace{1cm} \text{a.s.}
\end{equation*}
which may be translated immediately into the matricial form
\begin{equation}
\label{eq:ascvgMn}
\lim_{n \to + \infty} \frac{1}{n}M_n = 0 \hspace{1cm} \text{a.s.}
\end{equation}
Finally, the convergence \eqref{eq:convasSnSGN} follows from the decomposition \eqref{eq:SplitSn} together with \eqref{eq:ascvgSigman} and \eqref{eq:ascvgMn},
which completes the proof of Theorem \ref{theo:asconvVn}.
\end{proof}

\noindent
It is straightforward to obtain the almost sure\ convergence of $\wh{W}_n$ as follows.

\begin{proof}[Proof of Corollary \ref{cor:convWnas}]
By Theorem  \ref{theo:asconvVn}  one has that $\wh{V}_n$ converges a.s.\ to $v^{\ast}$  under the assumption that $\alpha +\beta < 1/2$. Then, the almost
sure convergence of $\wh{W}_n$ to $W_{\varepsilon}(\mu,\nu)$ follows from assumption \eqref{Integrabilitycost2} and the arguments in the proof of \cite[Theorem 3.5]{Stochastic_Bigot_Bercu}.   
\end{proof}

\subsection{Proofs of the almost sure rates of convergence}\label{subsec:as-rate-proof}
We now establish the almost sure rates of convergence rates for the SGN algorithm. In contrast with the previous results, we emphasize that the 
regularization parameter $\varepsilon$ must now be small enough, in the sense of condition \eqref{eq:condeps}. This  entails the key inequality \eqref{ineq:eiglower} deduced from Proposition \ref{prop:keystone}.

\begin{proof}[Proof of Theorem \ref{theo:asrates}]
To alleviate the notation, we denote by $\| A \|$ either the operator norm $\| A \|_2$ or the Frobenius norm $\| A\|_F$ all along the proof.
Since these two norms are equivalent and verify that $\|.\|_2^2 \leq  \| .\|_F^2 \leq J \| . \|_2^2$, the  upper bounds derived below might hold up to multiplicative constant depending on $J$, which will not affect the results that are purely asymptotic. 
\newline
$\bullet$
Our starting point when $\alpha = 0$ is equation \eqref{eq:SNgen} written with $S_n = n \overline{S}_n$. We recall that the martingale increment is $\varepsilon_{n+1} = \nabla_v \he(X_{n+1}, \wh{V}_n) - \nabla \He(\wh{V}_n),$ so that for all $n \geq 0$,
\begin{align}
\wh{V}_{n+1} - v^\ast&=   \wh{V}_n - v^\ast-  \proj S_n^{-1} \proj \bigl(\nabla \He(\wh{V}_n) + \varepsilon_{n+1} \bigl) \nonumber\\
&  =\wh{V}_n - v^\ast -  \frac{1}{n} \bigl(\proj \bigl( \overline{S}_n^{\,-1} -G_{\varepsilon}^{-}(v^\ast) \bigr)\proj\bigr)\bigl(\nabla \He(\wh{V}_n) + \varepsilon_{n+1} \bigl)   \label{eq:decVhatn1} \\
& -\frac{1}{n}  G_{\varepsilon}^{-}(v^\ast)  \bigl(\nabla \He(\wh{V}_n) + \varepsilon_{n+1} \bigl) \nonumber
\end{align}
where we decomposed $S_n^{-1}  = n^{-1} \overline{S}_n^{\ -1}= n^{-1} G_{\varepsilon}^{-}(v^\ast) + 
n^{-1} (\overline{S}_n^{-1}-G_{\varepsilon}^{-}(v^\ast))$ and $\proj G_{\varepsilon}(v^\ast)=G_{\varepsilon}(v^\ast)$ which implies that 
$\proj G_{\varepsilon}^{-}(v^\ast) \proj =G_{\varepsilon}^{-}(v^\ast)$.
The rest of the proof then consists in a linearization of $\nabla \He(\wh{V}_n)$ around $v^\ast$. For that purpose, denote
$$
D_n=\bigl(\proj \bigl( \overline{S}_n^{\,-1} -G_{\varepsilon}^{-}(v^\ast) \bigr)\proj\bigr) \hspace{1cm} \text{and} \hspace{1cm}  \delta_n= \nabla \He(\wh{V}_n) - \nabla^2 H_{\varepsilon}(v^\ast)(\wh{V}_n - v^\ast).
$$
We obtain from \eqref{eq:decVhatn1} that for all $n \geq 0$, 
\begin{eqnarray*}
\wh{V}_{n+1} - v^\ast & = &   \Bigl( \proj -\frac{1}{n} G_{\varepsilon}^{-}(v^\ast)\nabla^2 H_{\varepsilon}(v^\ast)\Bigr)\bigl(\wh{V}_n - v^\ast \bigr) 
-\frac{1}{n}\proj \overline{S}_n^{\,-1}  \proj \varepsilon_{n+1} \nonumber \\
 & - &
\frac{1}{n}\proj \overline{S}_n^{\,-1}  \proj \delta_n - \frac{1}{n} D_n \nabla^2 H_{\varepsilon}(v^\ast)\bigl(\wh{V}_n - v^\ast \bigr).
\end{eqnarray*}
Hence, by setting $\wh{U}_n=G_{\varepsilon}^{1/2}(v^\ast)\bigl(\wh{V}_n 
- v^\ast\bigr)$, we obtain that for all $n \geq 0$, 
\begin{equation}
\label{eq:decVhatn3}
\wh{U}_{n+1} =   \Bigl( \proj -\frac{1}{n} \Gamma_{\varepsilon}(v^\ast)\Bigr)\wh{U}_n -  \frac{1}{n}  A_n\varepsilon_{n+1}   -\frac{1}{n} T_n
\end{equation}
where $\Gamma_{\varepsilon}(v^\ast)=G_{\varepsilon}^{-1/2}(v^\ast)\nabla^2 H_{\varepsilon}(v^\ast)G_{\varepsilon}^{-1/2}(v^\ast)$ and
$T_n=A_n \delta_n+B_n\bigl(\wh{V}_n - v^\ast\bigr)$ with
\begin{align}
A_n &=G_{\varepsilon}^{1/2}(v^\ast)\proj \overline{S}_n^{\,-1}  \proj, \label{eq:defAn}\\
B_n &=G_{\varepsilon}^{1/2}(v^\ast)D_n\nabla^2 H_{\varepsilon}(v^\ast). 
\label{eq:defBn}
\end{align}
Thanks to  inequality \eqref{ineq:eiglower}, we have that 
$
\tcb{\lambda_{\min}^{\langle \bv_J \rangle^\perp}(\Gamma_{\varepsilon}(v^\ast))} \geq 1.
$
 For all $0\leq k \leq n$, let
\begin{equation}
P_k^n=\prod_{i=k+1}^n \Bigl( \proj -\frac{1}{i} \Gamma_{\varepsilon}(v^\ast)\Bigr) \label{eq:defPkn}
\end{equation}
with the usual convention that $P_n^n=\proj$. We deduce from \eqref{eq:decVhatn3} that for all $n \geq 0$, 
\begin{equation}
\label{eq:maindecVhatn}
\wh{U}_{n+1} =   P_0^n\wh{U}_1 -\sum_{k=1}^n\frac{1}{k}  P_k^nA_k\varepsilon_{k+1}   -\sum_{k=1}^n\frac{1}{k}  P_k^n T_k.
\end{equation}
The first term of \eqref{eq:maindecVhatn} is easy to handle. If $\rho$ stands for to the minimal eigenvalue of $\Gamma_{\varepsilon}(v^\ast)$ when restricted to act on the subspace $\langle \bv_J \rangle^\perp$, a simple diagonalization of the matrix $\Gamma_{\varepsilon}(v^\ast)$  leads, for all  $0\leq k \leq n$, to
\begin{equation}
\label{eq:ubnormPn0}
\|  P_k^n \| \leq  \kappa \left(\frac{k}{n}\right)^\rho
\end{equation}
where $\kappa>0$.
Concerning the middle term of \eqref{eq:maindecVhatn}, let $(M_n)$ be the multidimensional martingale defined by $M_1=0$ and, for all $n \geq 1$,
$$
M_{n+1}=\sum_{k=1}^{n} A_k\varepsilon_{k+1}.
$$
We infer from \eqref{Grad.H} and \eqref{eq:defG} that $ \|\varepsilon_{n+1}\| \leq 4$,  $\E[\varepsilon_{n+1} | \cF_n]=0$ and
$$
\E[\varepsilon_{n+1}\varepsilon_{n+1}^T| \cF_n]= G_{\varepsilon}(\wh{V}_n)
-\nabla H_{\varepsilon}(\wh{V}_n) \nabla H_{\varepsilon}(\wh{V}_n)^T.
$$
Moreover, it follows from \eqref{eq:convasVn} 
$$
\lim_{n \to + \infty}  \nabla H_{\varepsilon}(\wh{V}_n)=\nabla H_{\varepsilon}(v^\ast)=0
\hspace{1cm} \text{and} \hspace{1cm}
\lim_{n \to + \infty}  G_{\varepsilon}(\wh{V}_n)= G_{\varepsilon}(v^\ast) \hspace{1cm} \text{a.s.}
$$
which ensures via \eqref{eq:convasSnSGN} and \eqref{eq:defAn} that
$$
\lim_{n \to + \infty}  A_n\E[\varepsilon_{n+1}\varepsilon_{n+1}^T | \cF_n]A_n^T=\proj \hspace{1cm} \text{a.s.}
$$
Consequently, we have from the Cesaro mean convergence theorem that the predictable quadratic variation of the multidimensional martingale
$(M_n)$ satisfies
$$
\lim_{n \to + \infty} \frac{1}{n} \langle M \rangle_n =\lim_{n \to + \infty} \frac{1}{n} \sum_{k=2}^n 
A_{k-1}\E[\varepsilon_{k}\varepsilon_{k}^T | \cF_{k-1}]A_{k-1}=\proj \hspace{1cm} \text{a.s.}
$$
Hence, we deduce from the strong law of large numbers for multidimensional martingales given by Theorem 4.3.16
in \cite{Duflo1997} that
\begin{equation*}
\|M_n\|^2 =\mathcal{O}(n \log n) \hspace{1cm} \text{a.s.} 
\end{equation*}
Therefore, there exists a finite positive random variable C such that for all $n \geq 1$,
\begin{equation}
\|M_{n+1}\| \leq C \sqrt{n \log n} \hspace{1cm} \text{a.s.} 
\label{eq:mainrate2}
\end{equation}
Hereafter, denote by $Q_{n+1}$ the middle term of \eqref{eq:maindecVhatn}.
We obtain from a simple Abel transform that
\begin{eqnarray}
Q_{n+1} & = & \sum_{k=1}^n\frac{1}{k}  P_k^n (M_{k+1}-M_{k})=\frac{1}{n}M_{n+1} + \sum_{k=1}^{n-1}\frac{1}{k}  P_k^n M_{k+1} -
\sum_{k=2}^n\frac{1}{k}  P_k^n M_{k} \nonumber \\
& = & \frac{1}{n}M_{n+1} + \sum_{k=1}^{n-1} \Bigl(\frac{1}{k}  P_k^n-\frac{1}{k+1}  P_{k+1}^n\Bigr)M_{k+1} \nonumber \\
& = & \frac{1}{n}M_{n+1} + \sum_{k=1}^{n-1} \frac{1}{k(k+1)}\bigl( \proj - \Gamma_{\varepsilon}(v^\ast) \bigr) P_{k+1}^nM_{k+1}
\label{eq:mainrate3}
\end{eqnarray}
It follows from \eqref{eq:ubnormPn0}, \eqref{eq:mainrate2}, \eqref{eq:mainrate3} that for all $n \geq 1$,
\begin{eqnarray*}
\| Q_{n+1} \| & \leq  & C \Bigl( \frac{\sqrt{n\log n}}{n} + \frac{\kappa}{n^{\rho}} \sum_{k=1}^{n-1}
\frac{(k+1)^{\rho}}{k(k+1)} \sqrt{k \log k} \Bigr) \hspace{1cm} \text{a.s.} \nonumber \\
& \leq  & C \Bigl( \Bigl(\frac{\log n}{n} \Bigr)^{1/2}+ \frac{\kappa \sqrt{\log n} }{n^{\rho}} \sum_{k=1}^n
\frac{1}{k^a} \Bigr) \hspace{1cm} \text{a.s.} 
\end{eqnarray*}
where $a=3/2 - \rho < 1$. Consequently, we deduce that for all $n \geq 1$,
\begin{equation}
\| Q_{n+1} \| \leq C \Bigl( \Bigl(\frac{\log n}{n} \Bigr)^{1/2} + \frac{\kappa n^{1-a} \sqrt{\log n} }{(1-a)n^{\rho}} \Bigr)
\leq D \Bigl(\frac{\log n}{n} \Bigr)^{1/2} \hspace{1cm} \text{a.s.} 
\label{eq:mainrate5}
\end{equation}
where
$$
D= \frac{C(1-a + \kappa)}{1-a}.
$$
The last term of \eqref{eq:maindecVhatn} is much more difficult to handle. Denote for all $n \geq 1$,
\begin{equation}
\Delta_n=\sum_{k=1}^n\frac{1}{k}  P_k^n T_k 
\label{eq:Deltanrate}
\end{equation}
We recall that $T_n=A_n \delta_n+B_n\bigl(\wh{V}_n - v^\ast\bigr)$ where $A_n$ and $B_n$ are given by \eqref{eq:defAn} and \eqref{eq:defBn}.
We already saw from \eqref{eq:lingrad1} that
$$
\| \delta_n \| \leq \frac{2 \sqrt{2}}{\varepsilon} \| \wh{V}_n - v^\ast \|^2
$$
which implies that
\begin{equation}
\label{eq:mainrate6}
\| T_n \| \leq \frac{2 \sqrt{2}}{\varepsilon} \| A_n \| \, \| \wh{V}_n - v^\ast \|^2 +   \| B_n \| \, \| \wh{V}_n - v^\ast \|. 
\end{equation}
Moreover, it follows from \eqref{eq:convasVn} and \eqref{eq:convasSnSGN} that
$$
\lim_{n \to + \infty}  \| A_n \|=\| G_{\varepsilon}^{-1/2}(v^\ast)\| 
\hspace{1cm} \text{and} \hspace{1cm}
\lim_{n \to + \infty}  \| B_n \|= 0 \hspace{1cm} \text{a.s.}
$$
Consequently, we obtain from \eqref{eq:convasVn} and \eqref{eq:mainrate6} 
that it exists a positive constant $b=(4 \kappa)^{-1}$ where $\kappa$ is introduced in \eqref{eq:ubnormPn0},
such that for $n$ large enough,
\begin{equation}
\label{eq:mainrate7}
\| T_n \| \leq b  \| \wh{V}_n - v^\ast \|  \hspace{1cm} \text{a.s.} 
\end{equation}
Define for all $n \geq 1$,
\begin{equation}
L_n= \frac{1}{n} \sum_{k=1}^n \| \wh{V}_k - v^\ast \|.
\label{eq:Lnrate}
\end{equation}
We deduce from \eqref{eq:Deltanrate} together with \eqref{eq:ubnormPn0} and \eqref{eq:mainrate7} 
that for all $n \geq 1$,
\begin{equation}
\label{eq:mainrate8}
\| \Delta_{n} \| \leq   \frac{\kappa}{n^\rho}\sum_{k=1}^{n} 
\frac{k^{\rho}}{k} \|T_{k}\| \leq \frac{E}{n^\rho} + \frac{\kappa b}{n^\rho}\sum_{k=1}^{n} 
\frac{k^{\rho}}{k} \| \wh{V}_k - v^\ast \| \leq \frac{E}{n^\rho} +\kappa b L_n
\hspace{1cm} \text{a.s.} 
\end{equation}
where $E$ is  a finite positive random variable. Putting together the three contributions \eqref{eq:ubnormPn0}, \eqref{eq:mainrate5} and
\eqref{eq:mainrate8}, we obtain from \eqref{eq:maindecVhatn} that for all 
$n \geq 1$,
$$
\| \wh{U}_{n+1} \| \leq    \frac{\tau \|\wh{U}_1\|+E}{n^\rho} + D \Bigl(\frac{\log n}{n} \Bigr)^{1/2} +\kappa b L_n \hspace{1cm} \text{a.s.} 
$$
which implies that a  finite positive random variable $F$ exists and a constant $0<c<1/2$ such that for all $n \geq 1$,
\begin{equation}
\label{eq:mainrate9}
\| \wh{V}_{n+1} - v^\ast \| \leq  F \Bigl(\frac{\log n}{n} \Bigr)^{1/2}  +c L_n \hspace{1cm} \text{a.s.} 
\end{equation}
Herafter, we have from \eqref{eq:Lnrate} and \eqref{eq:mainrate9} that or 
all $n \geq 1$,
\begin{eqnarray*}
L_{n+1}  &=& \Bigl(1- \frac{1}{n+1} \Bigl)L_{n} + \frac{1}{n+1} \|  \wh{V}_{n+1} - v^\ast \|, \\
&\leq & \Bigl(1- \frac{1}{n+1} \Bigl)L_{n} + \frac{1}{n+1} \Bigl( F \Bigl(\frac{\log n}{n} \Bigr)^{1/2}  +c L_n \Bigr) \hspace{1cm} \text{a.s} \\
&\leq & \Bigl(1- \frac{d}{n+1} \Bigl)L_{n} + \frac{F}{(n+1)} \Bigl(\frac{\log n}{n} \Bigr)^{1/2} \hspace{1cm} \text{a.s}
\end{eqnarray*}
where $d=1-c$. A straightforward induction yields that for all $n \geq 1$,
\begin{equation}
\label{eq:mainrate10}
L_n \leq \prod_{k=2}^n \Bigl(1 - \frac{d}{k}  \Bigr) L_1 
+ \sum_{k=2}^n \prod_{i=k+1}^n  \Bigl(1 - \frac{d}{i} \Bigr) \frac{F}{(k+1)} \Bigl(\frac{\log k}{k} \Bigr)^{1/2} \hspace{1cm} \text{a.s}
\end{equation}
However, it is well-known that
$$
\prod_{k=2}^n \Bigl(1 - \frac{d}{k}  \Bigr) \leq \Bigl(\frac{2}{n+1}\Bigr)^d
 \hspace{1cm}\text{and} \hspace{1cm}
\prod_{i=k+1}^n \Bigl(1 - \frac{d}{i}  \Bigr) \leq \Bigl(\frac{k+1}{n+1}\Bigr)^d.
$$
Hence, we obtain from \eqref{eq:mainrate10} that for all $n \geq 1$,
$$
L_n  \leq   \Bigl(\frac{2}{n+1}\Bigr)^d L_1+ 
F  \Bigl(\frac{1}{n+1}\Bigr)^d \sum_{k=2}^n \frac{(k+1)^d}{(k+1)}\Bigl(\frac{\log k}{k} \Bigr)^{1/2} 
\hspace{1cm} \text{a.s}
$$
Since $1/2<d<1$, it implies that
\begin{equation*}
L_n \leq  \frac{2^d L_1}{n^d} + 
\frac{F \bigl( \log n \bigr)^{1/2}}{n^d}
\sum_{k=1}^{n} \frac{1}{k^{3/2-d}} \leq \frac{2^d L_1}{n^d} + \frac{2 F}{2d-1} \Bigl(\frac{\log n}{n} \Bigr)^{1/2}  \hspace{1cm} \text{a.s}
\end{equation*}
leading to
\begin{equation}
\label{eq:mainrateLn}
L_n =  \mathcal{O}\Bigl( \Bigl(\frac{ \log n}{n}\Bigr)^{1/2} \,\Bigr) \hspace{1cm} \text{a.s}
\end{equation}
Finally, it follows from \eqref{eq:mainrate9} and \eqref{eq:mainrateLn} that
\begin{equation*}
\| \wh{V}_{n} - v^\ast \|^2  = \mathcal{O} \Bigl(\frac{ \log n}{n}\Bigr)\hspace{1cm} \text{a.s.} 
\end{equation*}
which completes the proof of \eqref{eq:asrateVn}. 
\newline
$\bullet$
We now focus our attention on \eqref{eq:asrateSnSGN}. We have from \eqref{eq:SplitSn} that
\begin{equation}
\label{eq:SplitSnovline}
\overline{S}_n - G_{\varepsilon}(v^\ast)  = \frac{1}{n} M_n + \frac{1}{n} \bigl( I_J + R_n \bigr) 
+\frac{1}{n} \sum_{k=1}^n \bigl(  G_{\varepsilon}( \wh{V}_{k-1})  - G_{\varepsilon}( v^{\ast}) \bigr).
\end{equation}
On the one hand, let $M_n(u)=u^TM_n u$ where $u\in \R^J$. We already saw that $(M_n(u))$ is a locally square-integrable martingale with 
increments bounded by $8\|u\|^2$. Moreover, its predictable quadratic variation satisfies
$\langle M(u) \rangle_n \leq 16n  \|u\|^4$. Therefore, we obtain from the 
third part of Theorem 1.3.24 in \cite{Duflo1997} that
$$
|M_n(u)|^2=\mathcal{O}(n \log n) \hspace{1cm} \text{a.s.} 
$$
which implies that
\begin{equation}
\label{eq:rateMnSn}
\frac{1}{n} \|M_n\|= \mathcal{O}\Bigl( \Bigl(\frac{ \log n}{n}\Bigr)^{1/2} \,\Bigr)\hspace{1cm} \text{a.s.} 
\end{equation}
On the other hand, we already saw  from Lemma \ref{lem:LipGv} that
\begin{equation*}
\frac{1}{n} \sum_{k=1}^n \|  G_{\varepsilon}( \wh{V}_{k-1})  - G_{\varepsilon}( v^{\ast}) \| \leq \frac{4L_n}{\varepsilon} 
\end{equation*}
which ensures via \eqref{eq:mainrateLn} that
\begin{equation}
\label{eq:mainratesumGn}
\frac{1}{n} \sum_{k=1}^n \|  G_{\varepsilon}( \wh{V}_{k-1})  - G_{\varepsilon}( v^{\ast}) \| = 
\mathcal{O}(
\Bigl( \Bigl(\frac{ \log n}{n}\Bigr)^{1/2} \,\Bigr)\hspace{1cm} \text{a.s.} 
\end{equation}
Furthermore, we also have
\begin{equation}
\label{eq:ratemajRn}
\frac{1}{n} \|  R_n \| \leq \Bigl(\frac{\gamma \max(\nu)}{1-\beta} \Bigr)\frac{1}{n^\beta}
\end{equation}
where $\beta<1/2$. Consequently, we deduce the almost sure rate of convergence \eqref{eq:asrateSnSGN}  for $\overline{S}_n$ from the conjunction of \eqref{eq:SplitSnovline}, \eqref{eq:rateMnSn}, \eqref{eq:mainratesumGn} and \eqref{eq:ratemajRn}. Finally, we obtain the almost sure rate of convergence \eqref{eq:asrateSnSGN}  for $\overline{S}_n^{-1}$  from the identity
$$
\overline{S}_n^{-1}- G_{\varepsilon}^{-} (v^\ast)=\overline{S}_n^{-1} \bigl( G_{\varepsilon} (v^\ast) - \overline{S}_n \bigr) G_{\varepsilon}^{-} (v^\ast),
$$
which completes the proof of Theorem \ref{theo:asrates}.
\end{proof}

\subsection{Proofs of the asymptotic normality results}\label{subsec:an-conv-proof}

\begin{proof}[Proof of Theorem \ref{theo:anVn}]
We now prove the asymptotic normality for the SGN algorithm.
\newline
$\bullet$
We recall from \eqref{eq:maindecVhatn} that for all $n \geq 1$,
\begin{equation}
\label{eq:maindecCLT}
\sqrt{n}\bigl( \wh{V}_{n+1} -v^\ast \bigr) =  -\sqrt{n} G_{\varepsilon}^{-1/2}(v^\ast) Q_{n+1} + R_n 
\end{equation}
where $R_n = \sqrt{n} G_{\varepsilon}^{-1/2} (v^\ast) \bigl( P_0^n \wh{U}_1  - \Delta_n\bigr)$ with
$\wh{U}_1=G_{\varepsilon}^{1/2}(v^\ast)\bigl(\wh{V}_1 - v^\ast\bigr)$,
\begin{align*}
Q_{n+1} &=\sum_{k=1}^n\frac{1}{k}  P_k^nA_k\varepsilon_{k+1}, \\
\Delta_n &=\sum_{k=1}^n\frac{1}{k}  P_k^nT_k,
\end{align*} 
On the one hand, we claim that the remainder $R_n$ vanishes almost surely,
\begin{equation}
\label{eq:RnCLT}
\lim_{n \rightarrow \infty} R_n = 0  \hspace{1cm} \text{a.s.}
\end{equation}
As a matter of fact, we obviously have from \eqref{eq:ubnormPn0} that
$$
\lim_{n \rightarrow \infty}  \sqrt{n}  P_0^n \wh{U}_1= 0  \hspace{1cm} \text{a.s.}
$$
Moreover, we deduce from  the proof of Theorem \ref{theo:asrates} together with \eqref{eq:mainrate6} that
\begin{equation}
\label{eq:prclt1}
\| T_n \|= \mathcal{O}\Bigl( \frac{\log n}{n}\Bigr) + \mathcal{O}\Bigl( \frac{1}{n^\beta} 
\Bigl(\frac{ \log n}{n}\Bigr)^{1/2}\, \Bigr)=
\mathcal{O}\Bigl( \frac{1}{n^\beta} \Bigl(\frac{ \log n}{n}\Bigr)^{1/2}\, \Bigr)
\hspace{1cm} \text{a.s.}
\end{equation}
since $0< \beta <1/2$ and
$$
\| B_n \| =\mathcal{O}\Bigl( \frac{1}{n^\beta} \Bigr)
\hspace{1cm} \text{a.s.}
$$
where $B_n$ is defined by \eqref{eq:defBn}. 
Therefore, we obtain from \eqref{eq:prclt1} that there exists a finite positive random variable C such that for all $n \geq 1$
\begin{equation}
\|T_{n}\| \leq \frac{C}{n^\beta} \Bigl(\frac{ \log n}{n}\Bigr)^{1/2} \hspace{1cm} \text{a.s.} 
\label{eq:prclt2}
\end{equation}
Consequently, it follows from \eqref{eq:ubnormPn0} and \eqref{eq:prclt2}  
that for all $n \geq 1$,
\begin{equation*}
\| \Delta_{n} \|  \leq   \frac{C \kappa}{n^{\rho}} \sum_{k=1}^{n}
\frac{k^{\rho}}{k^{1+\beta}} \Bigl(\frac{\log k}{k} \Bigr)^{1/2}
 \leq   \frac{C \kappa \sqrt{ \log n}}{n^{\rho}} \sum_{k=1}^{n}
\frac{1}{k^{a}} 
 \hspace{1cm} \text{a.s.} 
\end{equation*}
where $a=3/2 + \beta- \rho$, leading to
\begin{equation}
\| \Delta_{n} \| \leq  \frac{D}{n^{\beta}}  \Bigl(\frac{\log n}{n} \Bigr)^{1/2} \hspace{1cm} \text{a.s.} 
\label{eq:prclt3}
\end{equation}
with $D= \kappa C/(1-a)$. Hence, as $\beta>0$, we infer from \eqref{eq:prclt3} that
$$
\lim_{n \rightarrow \infty}  \sqrt{n}  \Delta_n= 0  \hspace{1cm} \text{a.s}
$$
which clearly implies that \eqref{eq:RnCLT} is satisfied. On the other hand,
$Q_{n+1}$ is a sum of weighted martingale differences.
We deduce from the first part of Proposition B.2 in \cite{Zhang2016} that
\begin{equation}
\label{eq:prclt4}
\sqrt{n} \, Q_{n+1} \liml \mathcal{N}\bigl(0, \Sigma \bigr)
\end{equation}
where the asymptotic covariance matrix $\Sigma$ is given by the integral form
\begin{eqnarray*}
\Sigma & = & \int_0^\infty \Bigl( \exp\Bigl(-\Bigl(\Gamma_{\varepsilon}(v^\ast) -\frac{1}{2}\proj\Bigr)s\Bigr)\Bigr)^2\,ds \\
& = &\int_0^\infty \Bigl( \exp\Bigl(-2s\Bigl(\Gamma_{\varepsilon}(v^\ast) -\frac{1}{2}\proj\Bigr)\Bigr)\,ds \\
& = & \frac{1}{2} \int_0^\infty \Bigl( \exp\Bigl(-t\bigl(\Gamma_{\varepsilon}(v^\ast) -\frac{1}{2}\proj\bigr)\Bigr)\,dt \\
& = & \frac{1}{2} \left[ - \Bigl(\Gamma_{\varepsilon}(v^\ast) -\frac{1}{2}\proj\Bigr)^{-} \exp\Bigl(-t\Bigl(\Gamma_{\varepsilon}(v^\ast) -\frac{1}{2}\proj\Bigr)\right]_0^\infty\\
& = & \Bigl(2\Gamma_{\varepsilon}(v^\ast) -\proj\Bigr)^{-}.
\end{eqnarray*}
Therefore, we obtain from \eqref{eq:prclt4} that
\begin{equation}
\label{eq:prclt5}
\sqrt{n} \,G_{\varepsilon}^{-1/2}(v^\ast) Q_{n+1} \liml \mathcal{N}\Bigl(0, G_{\varepsilon}^{-1/2}(v^\ast)  \bigl( 2 \Gamma_{\varepsilon}(v^\ast) -\proj \bigr)^{-} G_{\varepsilon}^{-1/2}(v^\ast)\Bigr).
\end{equation}
Finally, it follows from \eqref{eq:maindecCLT} together with \eqref{eq:RnCLT} and \eqref{eq:prclt5} that
$$
\sqrt{n} \bigl( \wh{V}_n -v^\ast \bigr) \liml \mathcal{N}\Bigl(0, G_{\varepsilon}^{-1/2}(v^\ast)  \bigl( 2 \Gamma_{\varepsilon}(v^\ast) -\proj \bigr)^{-} G_{\varepsilon}^{-1/2}(v^\ast)\Bigr),
$$
which is exactly what we wanted to prove.\\
\newline
$\bullet$
It only remains to prove the asymptotic normality \eqref{eq:anWn}. We already saw from 
inequality \eqref{TAYLOR1} that for all $v \in \R^J$, 
$$
H_{\varepsilon}(v)-H_{\varepsilon}(v^\ast) \leq \frac{1}{2 \varepsilon} \|v-v^\ast\|^2.
$$
Moreover, we have from \eqref{DefWn} that $\wh{W}_n-W_{\varepsilon}(\mu,\nu)$ can be splitted into two terms, 
\begin{align}
\sqrt{n}\bigl(\wh{W}_n-W_{\varepsilon}(\mu,\nu)\bigr) &= \frac{1}{\sqrt{n}} \sum_{k=1}^n \bigl(H_{\varepsilon}(v^\ast)- h_{\varepsilon}(X_k,\hVkm)\bigr),  \nonumber \\
& \hspace{-2cm} = \frac{1}{\sqrt{n}} \sum_{k=1}^n \bigl(H_{\varepsilon}(\hVkm)- h_{\varepsilon}(X_k,\hVkm)\bigr)  - 
\frac{1}{\sqrt{n}} \sum_{k=1}^n \bigl(H_{\varepsilon}(\hVkm)-  H_{\varepsilon}(v^\ast)\bigr). \label{eq:decompTCL}
\end{align}
The second term in equation \eqref{eq:decompTCL} goes to $0$ a.s.\ thanks to the almost sure rate of convergence \eqref{eq:asrateVn},
$$
\frac{1}{\sqrt{n}} \sum_{k=1}^n \bigl(H_{\varepsilon}(\hVkm)-  H_{\varepsilon}(v^\ast)\bigr) \leq \frac{1}{2 \varepsilon \sqrt{n}} \sum_{k=1}^n \|\hVkm-v^\ast\|^2 
= \mathcal{O} \Bigl( \frac{\log^2(n)}{\sqrt{n}} \Bigr) \hspace{1cm}\text{a.s.}
$$
Finally, the first term in equation \eqref{eq:decompTCL} is dealt using argument from the proof of Theorem 3.5 in \cite{Stochastic_Bigot_Bercu}, allowing to prove that it satisfies the asymptotic normality \eqref{eq:anWn}. This completes the proof of Theorem \ref{theo:anVn}.
\end{proof}

\subsection{Proofs of the non-asymptotic rates of convergence}\label{subsec:nonasymp-proof}
We first detail  some keystone results related to  the use of the  Kurdyka-\L ojasiewicz functional inequality that is at the heart of the proof of Theorem \ref{theo:rates-nonasymp}.

\subsubsection{Kurdyka-\L ojasiewicz inequality.} The analysis that we carry out is essentially based on the  so-called Kurdyka-\L ojasiewicz   functional inequality. 
We refer to the initial works \cite{Kurdyka,Lojasiewicz} and to \cite{Bolte2,Bolte3}  for the use of such inequality in deterministic optimization. To this end, let
$\widetilde{H}_{\varepsilon}$ be the positive and convex function defined, for all $v \in \langle \bv_J \rangle^\perp$, by
$$
\widetilde{H}_{\varepsilon}(u) =  
H_{\epsilon}(G_{\ast}^{-1/2} u) - H_{\epsilon}(G_{\ast}^{-1/2} u^\ast),  
$$
where $G_{\ast}^{-1/2}$ stands for the square root of the Moore-Penrose inverse of $G_{\ast} = G_{\varepsilon}(v^{\ast})$, and $u^\ast = G_{\ast}^{1/2} v^{\ast}$.  
Since  $G_{\ast}^{-1/2}$ is symmetric, we notice that $\nabla \widetilde{H}_{\varepsilon}(u) = G_{\ast}^{-1/2}\nabla H_{\varepsilon}(G_{\ast}^{-1/2} u ) $ and $\nabla^2 \widetilde{H}_{\varepsilon}(u)  = G_{\ast}^{-1/2}\nabla^2 H_{\varepsilon}(G_{\ast}^{-1/2} u ) G_{\ast}^{-1/2}$. Hence, we obtain from  the upper bounds \eqref{eq:boundHepsgradient}  and \eqref{eq:boundlambdamax} that for all $u \in  \langle \bv_J \rangle^\perp$, 
\begin{equation} \label{eq:boundgradient:Htilde}
\| \nabla \widetilde{H}_{\varepsilon}(u)  \| \leq 2 \lambda_{\max}(G_{\ast}^{-1/2}),
\end{equation}
and
\begin{equation} \label{eq:boundlambdamax:Htilde}
\lambda_{\max}\bigl( \nabla^2 \widetilde{H}_{\varepsilon}(u) \bigr) \leq \frac{1}{\varepsilon} \lambda_{\max}(G_{\ast}^{-}).
\end{equation}
Consequently, it follows from the Taylor-Lagrange formula that  for all $u \in  \langle \bv_J \rangle^\perp$, 
\begin{equation} \label{TAYLOR1tildeH}
\widetilde{H}_{\varepsilon}(u) \leq \frac{1}{2 \varepsilon}  \lambda_{\max}(G_{\ast}^{-}) \|u - u^\ast\|^2.
\end{equation}
First of all, we verify that the function $\widetilde{H}_{\varepsilon}$ satisfies  a  Kurdyka-\L ojasiewciz inequality as stated in the next proposition. We refer the reader to \cite{gadat:hal-01623986} and the references therein for further details on this topic.

\begin{prop}\label{prop:KL}
There exist two positive constants $m_\varepsilon<M$  such that, for all $u \in  \langle \bv_J \rangle^\perp$ with $u \neq u^*$,
\begin{equation} 
0 <m_\varepsilon \leq \|\nabla  \widetilde{H}_{\varepsilon}(u)\|^2+\frac{\|\nabla \widetilde{H}_{\varepsilon}(u)\|^2}{\widetilde{H}_{\varepsilon}(u)} \leq M < +\infty. \label{eq:KL}
\end{equation}
Moreover, the constant $m_\varepsilon$ can be chosen as
\begin{equation}
m_{\varepsilon} = \varepsilon \lambda_{\min}(G_{\varepsilon}(v^{\ast}))\min\Bigl(  1 ,   \frac{\varepsilon}{4 }   \Bigr) . \label{eq:choicem} 
\end{equation}
\end{prop}
The proof of this key inequality is postponned to Appendix \ref{Appendix-KL}.
From equation \eqref{eq:choicem}, we clearly observe  that the magnitude of the constant $m_{\varepsilon}$  depends on  $\varepsilon$. Nevertheless, in the analysis carried out in this paper, the regularization parameter  $\varepsilon$ is held fixed, and we will not be interested in deriving sharp upper bounds depending on $\varepsilon$ for the mean square error of $\wh{V}_n$. We believe that a careful analysis of the role of $\varepsilon$  on the convergence of the SGN algorithm is a difficult issue that is left open for future investigation.

\subsubsection{Choice of a Lyapunov function}

Hereafter, a key step in our analysis is based on the  Lyapunov function $\Phi$ defined, for all $u \in  \langle \bv_J \rangle^\perp$, by
\begin{equation}
\label{def:Phip}
\Phi(u) =\widetilde{H}_{\varepsilon}(u) \exp( \widetilde{H}_{\varepsilon}(u)).
\end{equation}
On the one hand, it follows from the elementary inequality $\exp(x) \leq 1+x\exp(x)$ that for all $u \neq u^*$,
\begin{equation}
\frac{\Phi(u)}{\widetilde{H}_{\varepsilon}(u)} 
\leq 1 + \Phi(u). \label{eq:Phi0}
\end{equation}
We shall repeatedly use inequality \eqref{eq:Phi0} in all the sequel. On the other hand, we can
easily compute for all $u \neq u^*$,
$$
\nabla \Phi(u) =  \Bigl( 1 + \frac{1}{\widetilde{H}_{\varepsilon}(u)} \Bigr) \Phi(u) \nabla \widetilde{H}_{\varepsilon}(u),
$$
which implies that
$$
\langle \nabla \Phi(u), \nabla \widetilde{H}_{\varepsilon}(u)\rangle = \Bigl( 1 + \frac{1}{\widetilde{H}_{\varepsilon}(u)} \Bigr) \Phi(u) \|\nabla \widetilde{H}_{\varepsilon}(u)
\|^2.
$$
Consequently, we deduce from Proposition \ref{prop:KL} that $ m_{\varepsilon} \Phi(u) \leq \langle \nabla \Phi(u), \nabla \widetilde{H}_{\varepsilon}(u)\rangle \le M \Phi(u)$. Moreover, if $\Sigma$ denotes a positive semi-definite matrix, by an application of  Proposition \ref{prop:KL}, we also obtain the following {\it key}  lower bound
\begin{equation}
\Bigl\langle \nabla \Phi(u), \Sigma \nabla \widetilde{H}_{\varepsilon}(u)\Bigr\rangle 
\ge m_{\varepsilon} \lambda_{\min}(\Sigma) \Phi(u),   
\label{eq:outil-Phip2}
\end{equation}
that will be useful to control the non-asymptotic rate of convergence  of the SGN algorithm. Finally, as remarked in  \cite{gadat:hal-01623986}, a straightforward computation leads, for all $u \neq u^*$, to
$$
\nabla^2 \Phi(u) = \Phi(u)   \Bigl( \Bigl( 1 + \frac{2}{\widetilde{H}_{\varepsilon}(u) }\Bigr)   \nabla \widetilde{H}_{\varepsilon}(u)  \nabla \widetilde{H}_{\varepsilon}(u)^T  +  \Bigl( 1 + \frac{1}{\widetilde{H}_{\varepsilon}(u) }\Bigr) \nabla^2 \widetilde{H}_{\varepsilon}(u)  \Bigr). 
$$
Hence, using the fact that $ \lambda_{\max}\bigl( \nabla \widetilde{H}_{\varepsilon}(u)  \nabla \widetilde{H}_{\varepsilon}(u)^T\bigr)=   \| \nabla \widetilde{H}_{\varepsilon}(u) \|^2$, we obtain that for all $u \neq u^*$,
$$
\lambda_{\max}(\nabla^2 \Phi(u)) \leq  \Phi(u)   \Bigl( \Bigl( 1 + \frac{2}{\widetilde{H}_{\varepsilon}(u) }\Bigr)  \| \nabla \widetilde{H}_{\varepsilon}(u) \|^2   +  \Bigl( 1 + \frac{1}{\widetilde{H}_{\varepsilon}(u) }\Bigr) \lambda_{\max}\bigl(\nabla^2 \widetilde{H}_{\varepsilon}(u)\bigr)  \Bigr).
$$
Therefore, using inequality \eqref{eq:Phi0} together with the upper bounds \eqref{eq:boundgradient:Htilde} and \eqref{eq:boundlambdamax:Htilde}, we obtain 
that for all $u \in  \langle \bv_J \rangle^\perp$ with $u \neq u^*$,
\begin{equation}
\lambda_{\max}(\nabla^2 \Phi(u)) \leq
\delta_{\epsilon}  \lambda_{\max}(G_{\ast}^{-}) (1+\Phi(u))\label{eq:outil-Phip3}
\end{equation}
where $\delta_{\varepsilon}=   2(6  +  \varepsilon^{-1})$.
Inequality \eqref{eq:outil-Phip3} will also be crucial to derive the non-asymptotic rate of convergence  of the SGN algorithm.
Finally,  thanks to the following result, we will be able to relate the study of the Lyapunov function $\Phi$ to the quadratic risk of  $\wh{V}_n$ and $S_n$. 
\begin{prop} \label{prop:boundu}
There exists a positive constant $d_{\varepsilon}$ such that for all $u \in  \langle \bv_J \rangle^\perp$,
 \begin{equation}
 \|u-u^{\ast}\|^{2} \leq d_{\varepsilon} \Phi(u). \label{eq:boundu}
\end{equation}
\end{prop}
\begin{proof}
First, one can verify that for all $u \in  \langle \bv_J \rangle^\perp$
in a neighborhood of $u^{\ast}$ with $u \neq u^{\ast}$, the function $\|u-u^{\ast}\|^{-2} \widetilde{H}_{\varepsilon}(u)$ is  lower bounded.
Moreover, using that $\widetilde{H}_{\varepsilon}$ is a convex function on $\langle \bv_J \rangle^\perp$ that attains its minimal value at $u^{\ast}$ with a non-degenerate minimum, we also have
$$
\liminf_{\|u\| \rightarrow + \infty} \frac{\widetilde{H}_{\varepsilon}(u)}{\|u\|} >0.
$$
It implies that for any positive $t$,
$$
\lim_{\|u\| \rightarrow + \infty} \frac{\exp\bigl( t\widetilde{H}_{\varepsilon}(u)\bigr)}{\|u-u^{\ast}\|^2} = + \infty.
$$
Since $\widetilde{H}_{\varepsilon}(u) \geq 0$, one always has $\Phi(u) \geq  \widetilde{H}_{\varepsilon}(u)$. Consequently, for all
$u \neq u^{\ast}$,
\begin{align*}
\frac{ \Phi(u)}{\|u-u^{\ast}\|^{2}} &=\frac{ \Phi(u)}{\|u-u^{\ast}\|^{2}} \1_{\|u-u^{\ast}\|\leq 1} + \frac{ \Phi(u)}{\|u-u^{\ast}\|^{2}} \1_{\|u-u^{\ast}\|\geq 1},\\
& \geq \frac{\widetilde{H}_{\varepsilon}(u)}{\|u-u^{\ast}\|^2} \1_{\|u-u^{\ast}\|\leq 1} + \min_{\|u-u^{\ast}\| \ge 1} 
\frac{\widetilde{H}_{\varepsilon}(u) \exp\bigl(\widetilde{H}_{\varepsilon}(u)\bigr)}{\|u-u^{\ast}\|^2}   \geq \frac{1}{d_\varepsilon}
 \end{align*}
for some positive constant $d_\varepsilon$, where we used the local behavior around $u^{\ast}$ of $\widetilde{H}_{\varepsilon}$ to derive a lower bound for the first term and the asymptotic behavior of $\widetilde{H}_{\varepsilon}$ for the second one with $t=1$, which is exactly what we wanted to prove.
 \end{proof}

\subsubsection{A recursive inequality and proof of Theorem \ref{theo:rates-nonasymp}} 

We  first describe the  one-step evolution of the sequence $(\Phi(G_{\ast}^{1/2} \hVn))$ where $(\wh{V}_n)$ is the recursive sequence defined by \eqref{eq:SNgen} corresponding to the SGN algorithm. From this analysis, we shall also deduce the  rate of convergence of the expected quadratic risk associated with $\wh{V}_n$ and 
$\overline{S}_n$. Denote $\hUn = G_{\ast}^{1/2} \hVn$.

\begin{prop}\label{prop:Phip} Assume that $\alpha \in [0,1/2[$ and that $\alpha + \beta < 1/2$. Then, there exist an integer $n_0 \geq J$ and a positive constant $c_{\varepsilon}>0$ such that, for all $n \geq n_0 $,
\begin{equation}
\E \bigl[\Phi(\hUnp) | \mathcal{F}_n\bigr] \leq \Bigl(1 - \frac{m_{\varepsilon}  n^{\alpha}  \lambda_{\min}(   G_{\ast} ) \lambda_{\min}(S_n^{-1})}{2}  \Bigr)  \Phi(\hUn)  +   c_{\varepsilon}  n^{2\alpha} \lambda^2_{\max}(  S_n^{-1}).  \label{eq:Phipi}
\end{equation}
\end{prop}

\noindent
The proof of Proposition \ref{prop:Phip} is postponed to Appendix \ref{Appendix-KL}.
We are now in position to establish the non-asymptotic rates of convergence for the SGN algorithm.

\begin{proof}[Proof of Theorem \ref{theo:rates-nonasymp}]
In the proof, we use the notation  $c_{\varepsilon}$ to denote a positive constant (depending on $\varepsilon$ and possibly on $\alpha$ and $\beta$) that is independent from $n$ and whose value may change from line to line. 
We only consider the situation where $\alpha \in ]0,1/2[$ and our analysis is based on the function $\Phi$. 
We will establish that for all $n \geq 1$,
\begin{equation}
\label{eq:ratephifast}
\E[ \Phi(\hUn)] \leq \frac{c_{\varepsilon}}{n^{1-\alpha}}.
\end{equation}

\noindent \underline{Step 1: Preliminary rate.}
 Our starting point is inequality \eqref{eq:Phipi} that we combine with \eqref{eq:boundlambda_SnSGN_max} and \eqref{eq:boundlambda_SnSGN_min} to obtain that there exists an integer $n_0$ such that, for all $n \ge n_0$,
\begin{equation}
\E \bigl[\Phi(\hUnp) | \mathcal{F}_n\bigr] \leq (1-c_1(n)n^{-1+\alpha}) \Phi(\hUn)+c_2(n) n^{2(\alpha + \beta-1)}
\label{eq:recPhifinal}
\end{equation}
where $n_J=n/J$,
\begin{align*}
c_1(n) &= \frac{m_{\varepsilon}  \lambda_{\min}(   G_{\ast} ) }{8+4 \gamma \max(\nu)+2 n^{-1}}, \\
c_2(n) &= \frac{c_{\varepsilon} }{ J^{2(\beta-1)}\bigl(\gamma \min(\nu)(1-n_{J}^{-1})^{1-\beta} + (1-2 \gamma \min(\nu)) n_J^{-1} \bigr)^2}.
\end{align*}
By taking the expectation on both sides of  \eqref{eq:recPhifinal}, we obtain that for all $n \geq n_0$,
\begin{equation*}
\E [\Phi(\hUnp)] \leq  \left(1-c_1(n) n^{-1+\alpha}\right) \E [\Phi(\hUn)]+ c_2(n) n^{-2(1-\alpha-\beta)}.
\end{equation*}
Hereafter, it is not hard to see that for $n$ large enough, there exist two positive constants $c_1$ and $c_2$, depending on $\varepsilon$, such that $c_1(n) \geq c_1$  and  $c_2(n) \leq c_2$. 
Hence, there exists an integer $n_0$ such that, for all $n \geq n_0$,
$$
\mathbb{E} [\Phi(\hUnp)]  \leq  \left(1-c_1n^{-1+\alpha}\right) \mathbb{E} [\Phi(\hUn)] + c_2 n^{-2(1-\alpha-\beta)}. 
$$
Therefore, it follows from the proof of Lemma A.3 in \cite{Stochastic_Bigot_Bercu} that there exist a positive constant $c_{\varepsilon}$ and an integer $n_0$ 
such that, for all $n \geq n_0$, 
\begin{equation*}
\mathbb{E} [\Phi(\hUn)] \leq \frac{c_{\varepsilon} }{n^{1-2(\alpha+\beta)}}.
\end{equation*}
We emphasize that at this stage, we do not obtain the announced result that necessitates further  work with a plug-in strategy. The rest of the proof details this
additional step. \\

\noindent \underline{Step 2: Plug-in.}
Let us now explain how one may improve the above result from $n^{-(1-2(\alpha+\beta))}$ to $n^{-(1-2\alpha)}$. First of all, thanks to Proposition \ref{prop:boundu}, we obtain via Step 1 that
\begin{equation}
\mathbb{E} [\|\hUn - u^{\ast}\|^2] \leq \frac{ c_{\varepsilon}}{n^{1-2(\alpha+\beta)}}. \label{eq:Uphislow}
\end{equation}
Next, we shall consider the study of the convergence rate of $\overline{S}_n$ to improve the pessimistic bounds \eqref{eq:boundlambda_SnSGN_max} and \eqref{eq:boundlambda_SnSGN_min}. If $p_n$ denotes the largest integer such that $p_n J \leq n$, we already saw from
\eqref{eq:SnRSGN} and \eqref{eq:DecRn} that for all $n \geq 1$,

\begin{equation*}
\overline{S}_n = \overline{\Sigma}_n + \overline{R}_n
\end{equation*}
where
\begin{align*}
\overline{\Sigma}_n & = \frac{1}{n}  \sum\limits_{k=1}^{n}  \nabla_v \he(X_{k}, \wh{V}_{k-1}) \nabla_v \he(X_{k}, \wh{V}_{k-1})^T, \\
\overline{R}_n & = \frac{1}{n}  \Id  + \frac{1}{n} \Bigl(\sum\limits_{m=1}^{p_n}  m^{-\beta} \Bigr)  \gamma  \diag(\nu) + \frac{1}{n}  \sum\limits_{k=p_n+1}^{n} \gamma   \Bigl( 1+  \bigl\lfloor \frac{k}{J} \bigr\rfloor \Bigr)^{-\beta}  Z_k Z_k^T.
\end{align*}
On the one hand, it is not hard to see that it exists a positive constant $c_3$ such that
\begin{equation}
\label{eq:rateRn}
\| \overline{R}_n \|_F^2 \leq   \frac{c_{3}}{n^{2\beta}}. 
\end{equation}
On the other hand, starting from the fact that
$$
\overline{\Sigma}_{n+1} =  \overline{\Sigma}_n +  \frac{1}{n+1} \Bigl( \nabla_v \he(X_{n+1}, \wh{V}_{n}) \nabla_v \he(X_{n+1}, \wh{V}_{n})^T  - \overline{\Sigma}_n  \Bigr),
$$
we obtain that
\begin{align*}
 \hspace{-4ex} \| \overline{\Sigma}_{n+1} -  G_{\ast}\|^2_F & =   \| \overline{\Sigma}_{n} -  G_{\ast}\|^2_F +  \frac{1}{(n+1)^2} \|   \nabla_v \he(X_{n+1}, \wh{V}_{n}) \nabla_v \he(X_{n+1}, \wh{V}_{n})^T - \overline{\Sigma}_n\|^2_F \\
 &  +\frac{2}{n+1} \langle  \tilde{S}_{n} -  G_{\ast} ,   \nabla_v \he(X_{n+1}, \wh{V}_{n}) \nabla_v \he(X_{n+1}, \wh{V}_{n})^T - \overline{\Sigma}_n \rangle_{F}.
\end{align*}
Therefore, by taking the conditional expectation on both sides of the above equality, we obtain that
\begin{eqnarray}
\E \Bigl[ \| \overline{\Sigma}_{n+1} -  G_{\ast}\|^2_F | \mathcal{F}_n \Bigr] & = &  \| \overline{\Sigma}_{n} -  G_{\ast}\|^2_F +\frac{2}{n+1} \langle  \overline{\Sigma}_{n} -  G_{\ast} ,   G_{\varepsilon}(\wh{V}_{n}) - \overline{\Sigma}_n \rangle_{F}\nonumber \\
 &  &+  \frac{1}{(n+1)^2}  \E \Bigl[ \|   \nabla_v \he(X_{n+1}, \wh{V}_{n}) \nabla_v \he(X_{n+1}, \wh{V}_{n})^T - \overline{\Sigma}_n \|^2_F  | \mathcal{F}_n \Bigr], \nonumber\\ 
 & = &  \| \overline{\Sigma}_{n} -  G_{\ast}\|^2_F - \frac{2}{n+1} \|  \overline{\Sigma}_{n} -  G_{\ast} \|_{F}^2  +\frac{2}{n+1} \langle  \overline{\Sigma}_{n} -  G_{\ast} ,   G_{\varepsilon}(\wh{V}_{n})  - G_{\ast} \rangle_{F}\nonumber \\
&&  +  \frac{1}{(n+1)^2}  \E \Bigl[ \|   \nabla_v \he(X_{n+1}, \wh{V}_{n}) \nabla_v \he(X_{n+1}, \wh{V}_{n})^T  - \overline{\Sigma}_n \|^2_F  | \mathcal{F}_n \Bigr],\nonumber\\
 & \leq & \| \overline{\Sigma}_{n} -  G_{\ast}\|^2_F \Bigl(1-\frac{1}{n+1}\Bigr) +\frac{1}{n+1} \| G_{\varepsilon}(\wh{V}_{n})  - G_{\ast}\|_{F}^2  \nonumber\\
 & &  +  \frac{1}{(n+1)^2}  \E \Bigl[ \|   \nabla_v \he(X_{n+1}, \wh{V}_{n}) \nabla_v \he(X_{n+1}, \wh{V}_{n})^T  - \overline{\Sigma}_n \|^2_F  | \mathcal{F}_n \Bigr]
 \label{eq:rectSn}
\end{eqnarray}
where the last line follows from Cauchy-Schwarz and Young inequalities. Moreover, we deduce from inequality \eqref{eq:boundgradient} that
$$ \E \Bigl[ \|   \nabla_v \he(X_{n+1}, \wh{V}_{n}) \nabla_v \he(X_{n+1}, \wh{V}_{n})^T  - \overline{\Sigma}_n \|^2_F  | \mathcal{F}_n \Bigr]  $$
 $$ \leq  2  \E \Bigl[ \|   \nabla_v \he(X_{n+1}, \wh{V}_{n}) \nabla_v \he(X_{n+1}, \wh{V}_{n})^T  \|^2_F  | \mathcal{F}_n \Bigr] + 2 \| \overline{\Sigma}_n  \|^2_F 
 \leq  2 \times 4^2 +2\times 4=40.
$$
Furthermore, we obtain from \eqref{ineq:Gstar} that
$$
\|  G_{\varepsilon}(\wh{V}_{n})  - G_{\ast} \|_F^2 \leq \frac{16  J}{\varepsilon^2}   \|\wh{V}_{n} - v_{\ast} \|^2.
$$
Using the previous bounds in \eqref{eq:rectSn} leads to the recursive inequality
$$
\E \Bigl[ \| \overline{\Sigma}_{n+1} -  G_{\ast}\|^2_F \Bigr] \leq  \Bigl(1 - \frac{1}{n+1}\Bigr)  \E \bigl[  \| \overline{\Sigma}_{n} -  G_{\ast}\|^2_F \bigr]  +  \frac{40}{(n+1)^2}  + \frac{16  J }{(n+1)\varepsilon^{2}}   \E \bigl[ \|\wh{V}_{n} - v^{\ast} \|^2 \bigr].
$$
Since $\hUn - u^{\ast} =G_{\ast}^{1/2}\bigl(\wh{V}_{n} - v^{\ast} \bigr)$, inequality \eqref{eq:Uphislow} ensures that
$$
\E \Bigl[ \| \overline{\Sigma}_{n+1} -  G_{\ast}\|^2_F \Bigr] \leq  \Bigl(1 - \frac{1}{n+1}\Bigr)  \E \bigl[  \| \overline{\Sigma}_{n} -  G_{\ast}\|^2_F \bigr]  +  \frac{40}{(n+1)^2}  + \frac{16  J c_\varepsilon \lambda_{max}( G_\ast^{-})}{n^{2(1-\alpha - \beta)}\varepsilon^{2}}.
$$
By applying Lemma A.3 in \cite{Stochastic_Bigot_Bercu}, we obtain that for $n$ large enough,
\begin{equation*} 
\E \bigl[  \| \overline{\Sigma}_{n} -  G_{\ast}\|^2_F \bigr] \leq \frac{c_{\varepsilon}}{n^{(1-2(\alpha+\beta))}}.
\end{equation*}
From this last inequality and  \eqref{eq:rateRn}, we conclude that for $n$ large enough,
\begin{equation} \label{eq:ratebarSn}
\E \bigl[ \| \overline{S}_{n} -  G_{\ast}\|^2_F \bigr] \leq  \frac{c_{\varepsilon}}{n^{(1-2(\alpha+\beta))}}+ \frac{c_{3}}{n^{2\beta}}.
\end{equation}
Since $4\beta < 1-2\alpha$, it follows that $1-2(\alpha+\beta)>2\beta$, which means that  $n^{-(1-2(\alpha+\beta))} $ decays faster than $n^{-2\beta}$ and the second inequality in \eqref{eq:nonasymp1} holds true. \\
\newline
$\bullet$
From now on, we focus our attention on the first inequality in \eqref{eq:nonasymp1}. It follows from \eqref{eq:Phipi} and the previous calculation that for $n$ large enough,
\begin{equation}
\E[\Phi(\hUnp)] \leq (1 - c_1(n) n^{-1+\alpha}) \E[ \Phi(\hUn)]  +   c_{\varepsilon} n^{2\alpha-2}  \E \bigl[ \lambda^2_{\max}(  \overline{S}_n^{-1}) \bigr], 
\label{eq:Phip_fast}
\end{equation}
using  that $\overline{S}_n^{-1}=nS_n^{-1}$. The identity $ \overline{S}_n^{-1} - G_{\ast}^{-} = G_{\ast}^{-}(G_{\ast} -  \overline{S}_n)  \overline{S}_n^{-1}$ implies that
$$
\lambda_{\max}(  \overline{S}_n^{-1}) \leq \| \overline{S}_n^{-1}  - G_{\ast}^{-}\|_{2} + \|G_{\ast}^{-}\|_2 \leq \lambda_{\max}( G_{\ast}^{-}) \lambda_{\max}(  \overline{S}_n^{-1} ) \|\overline{S}_n - G_{\ast}  \|_2 +  \|G_{\ast}^{-}\|_2.
$$
It follows from inequality \eqref{eq:boundlambda_SnSGN_max} that for $n$ large enough, $\lambda_{\max}(\overline{S}_{n}^{-1}) = n \lambda_{\max}(S_n^{-1}) \leq c_4 n^{\beta}$ where $c_4$ is a positive constant. It ensures that for $n$ large enough,
$$
\lambda_{\max}^2(  \overline{S}_n^{-1})    \leq 2 c_4^2 \lambda_{\max}^2( G_{\ast}^{-}) n^{2 \beta}\|\overline{S}_n - G_{\ast}  \|_2^2 + 2 \|G_{\ast}^{-}\|_2^2.
$$
Since $\|\overline{S}_n - G_{\ast}  \|_2^2 \leq \|\overline{S}_n - G_{\ast}  \|_F^2$, we obtain from inequality \eqref{eq:ratebarSn} that for $n$ large enough, 
$$
\E \Bigl[ \lambda_{\max}^2(  \overline{S}_n^{-1}) \Bigr] \leq c_{\varepsilon}  \Bigl(1 + \frac{n^{4\beta}}{n^{1-2\alpha}}   \Bigr) + 2 \|G_{\ast}^{-}\|_2^2.
$$
Consequently, we deduce from the condition $4\beta < 1-2\alpha$ that there exists a positive constant $C_\varepsilon$ such that for all $n \geq 1$,
\begin{equation}
\label{eq:UBmax2}
\E \Bigl[ \lambda_{\max}^2(  \overline{S}_n^{-1}) \Bigr] \leq C_{\varepsilon}.
\end{equation}
It follows from \eqref{eq:Phip_fast} and \eqref{eq:UBmax2} that there exist two postive constants $c_1$ and $c_\varepsilon$ such that for $n$ large enough,
\begin{equation*}
\E[\Phi(\hUnp)] \leq (1 - c_1 n^{-1+\alpha}) \E[ \Phi(\hUn)]  +   c_{\varepsilon} n^{2\alpha-2}.
\end{equation*}
Hereafter, using once again Lemma A.3 in \cite{Stochastic_Bigot_Bercu}, we obtain that for all $n\geq 1$,
\begin{equation*}
\mathbb{E} [\Phi(\hUn)] \leq  \frac{c_{\varepsilon}}{n^{1-\alpha}}
\end{equation*}
which is exactly the announced inequality \eqref{eq:ratephifast}. Finally, as $\wh{V}_{n} - v_{\ast} = G_{\ast}^{-1/2}\bigl( \hUn - u^{\ast} \bigr)$, 
the first inequality in \eqref{eq:nonasymp1} clearly follows from \eqref{eq:ratephifast} together with Proposition \ref{prop:boundu}.\\
\newline
$\bullet$
It only remains to prove the inequalities  \eqref{eq:nonasymp2} and \eqref{eq:nonasymp3}.
We already saw the decomposition
\begin{align}
\hWn-W_{\varepsilon}(\mu,\nu) &= \frac{1}{n} \sum_{k=1}^n \bigl( H_{\varepsilon}(v^\ast) -h_{\varepsilon}(X_k,\hVkm)\bigr), \nonumber\\
& = \frac{1}{n} \sum_{k=1}^n \xi_{k} -\frac{1}{n} \sum_{k=1}^n \bigl(H_{\varepsilon}(\hVkm) - H_{\varepsilon}(v^\ast)\bigr),
\label{eq:Wn-decomposition}
\end{align}
where the martingale increment $\xi_{k}=-h_{\varepsilon}(X_k,\hVkm) +H_{\varepsilon}(\hVkm)$. As $\mathbb{E}[\xi_{k+1}\, \vert \mathcal{F}_k] = 0$, 
it follows from \eqref{eq:Wn-decomposition} together with inequality \eqref{TAYLOR1} and the first inequality in \eqref{eq:nonasymp1} that for all $n \geq 1$,
$$
\Bigl| \E \bigl[\hWn\bigr]  -W_{\varepsilon}(\mu,\nu)\Bigr| = \Bigl| 
\frac{1}{n} \sum_{k=1}^n \E \bigl[H_{\varepsilon}(\hVkm) - H_{\varepsilon}(v^\ast)\bigr]\Bigr| \leq \frac{c_{\varepsilon}}{n} \sum_{k=1}^n \frac{1}{k^{1-\alpha}}
\leq \frac{c_{\varepsilon, \alpha}}{n^{1-\alpha}},
$$
which proves  inequality \eqref{eq:nonasymp2}.

Regarding the $\mathbb{L}^1$ risk $\E\bigl[\bigl|\hWn-W_{\varepsilon}(\mu,\nu)\bigr|\bigr]$, we still use the decomposition \eqref{eq:Wn-decomposition}. The triangle inequality and inequality \eqref{TAYLOR1} imply that
\begin{eqnarray*}
\E\bigl[\bigl|\hWn-W_{\varepsilon}(\mu,\nu)\bigr|\bigr]& \leq &\frac{1}{n} \E \Bigl[\Bigl| \sum_{k=1}^n \xi_k\Bigr|\Bigr]+ 
\frac{1}{n} \sum_{k=1}^n \E\bigl[
\bigl|H_{\varepsilon}(\hVkm) - H_{\varepsilon}(v^\ast)\bigr|\bigr],  \\
& \leq & \frac{1}{n} \Bigl(\E \Bigl(\sum_{k=1}^n \xi_{k} \Bigr)^{\!2}\Bigr)^{\!1/2}
+
\frac{1}{2 \varepsilon n} \sum_{k=1}^n \E[\|\hVkm-v^\ast\|^2], 
\end{eqnarray*}
where the last line comes from the Cauchy-Schwarz inequality. Let us now prove that  $\sup_{k \ge 1} \E[\xi_k^2]< + \infty$. To this end, we observe that
$$
\E [\xi_k^2] = \E\left( \E\left[  \left( h_{\varepsilon}(X_k,\hVkm) - H_{\varepsilon}(\hVkm) \right)^2\, \vert \mathcal{F}_k \right] \right) \leq \E[h^2_{\varepsilon}(X_k,\hVkm)],
$$
thanks to the property that $\E\left[  h_{\varepsilon}(X_k,\hVkm) \, \vert \mathcal{F}_k \right] =  H_{\varepsilon}(\hVkm)$.
Using that 
$
\|\partial_v h_{\varepsilon}(x,v)\| = \|\pi(x,v)-\nu\| \leq 2,
$
it follows by integration that
$
h_{\varepsilon}(x,v) \leq h_{\varepsilon}(x,v^\ast) + 2 \|v - v^\ast\|.
$
We then deduce that
$$
\E [\xi_k^2] \leq  2 \E [ h^2_{\varepsilon}(X,v^\ast)] + 4 \E[ \|\hVkm - v^\ast\|^2]
$$
Using the first inequality in \eqref{eq:nonasymp1}, it follows  that $(\E [\|\hVkm - v^\ast\|^2])_{k \ge 1}$ is a bounded sequence. Moreover, arguing as in the proof of \cite[Theorem 3.5]{Stochastic_Bigot_Bercu}, 
the condition  $\int_{\mathcal{X}} c^2(x,y_j) d\mu(x) < + \infty$, for any $1 \leq j \leq J$, implies   that $ \E [ h^2_{\varepsilon}(X,v^\ast)]$ is finite. Therefore, we conclude that $\sup_{k \ge 1} \E[\xi_k^2]< + \infty$. Hence, using once again the first inequality in \eqref{eq:nonasymp1} together with
a conditional expectation argument, we obtain that
$$
\E\bigl[\bigl|\hWn-W_{\varepsilon}(\mu,\nu)\bigr|\bigr] \leq \frac{c_{\varepsilon}}{n} \sqrt{n} + \frac{c_{\varepsilon}}{2 \varepsilon n} \sum_{k=1}^n \frac{1}{k^{1-\alpha}} \leq \frac{c_{\varepsilon}}{\sqrt{n}},
$$
which proves  inequality \eqref{eq:nonasymp3}.
This achieves the proof of Theorem \ref{theo:rates-nonasymp}.
\end{proof}


\appendix

\section{Appendix - Proofs of auxiliary results} \label{Appendix}

  
This appendix contains the proofs of some auxiliary results of the paper.  

\subsection{Proof of Lemma \ref{lem:LipGv}} \label{sec:LipGv}
 
Since $\nabla_{v} \he(X, v) =  \pi(X,v)  -  \nu$, we first remark that
\begin{eqnarray}
G_{\varepsilon}(v)  & = &  L(v) + \nu \nu^T  - \nu  \E \bigl[  \pi(X,v)  \bigr]^T -   \E \bigl[  \pi(X,v)  \bigr] \nu^T,  \nonumber \\
& = & L(v) - \nu \nu^T - \nu  \nabla \He(v)^T -  \nabla \He(v) \nu^T, \label{eq:decompG}
\end{eqnarray}
where $L(v) =   \E \bigl[ \pi(X,v) \pi(X,v)^T \bigr]$. For $u \in \langle \bv_J \rangle^\perp$ and $t \in [0,1]$, we define the real-valued function $\phi_u(t) = u^T G_{\varepsilon}(v_t) u$ with $v_t=v^\ast +t(v-v^\ast)$. It follows from the decomposition \eqref{eq:decompG}  that
$$
\phi_u(t) = u^T L(v_t) u - \langle u, \nu \rangle^2 - 2 \langle u, \nu \rangle \langle \Phi(t), u \rangle
$$
where $\Phi(t) = \nabla \He(v_t)$ is a vector-valued function satisfying
$$
\Phi^\prime(t)= \nabla^2 H_{\varepsilon}(v_t)(v - v^\ast) 
\hspace{1cm}\text{and} \hspace{1cm} 
\Phi^{\prime \prime}(t)=\nabla^3 H_{\varepsilon}(v_t)[v - v^\ast,v - v^\ast],
$$
where $\nabla^3 H_{\varepsilon}$ stands for the third-order tensor derivative of $H_{\varepsilon}$. Now, since $\Phi(1) = \nabla \He(v)$ and $\Phi(0) = \nabla \He(v^\ast) = 0$, and using the property that $\phi_u(1) - \phi_u(0) = \int_{0}^{1} \phi^\prime_u(t) dt$, we obtain that
\begin{equation}
u^T (L(v) - L(v^\ast)) u - 2 \langle u, \nu \rangle \langle  \nabla \He(v), u \rangle = \int_{0}^{1} \phi^\prime_u(t) dt. \label{eq:diffPhi}
\end{equation}
We clearly have
$$
\phi^\prime_u(t) = \frac{\partial}{\partial t} u^T L(v_t) u - 2 \langle u, \nu \rangle \langle \Phi^\prime(t), u \rangle
$$
where 
$$
\frac{\partial}{\partial t} u^T L(v_t) u = \langle \nabla_{v} \psi_{u}(v_t) , v-v^{\ast} \rangle
\hspace{1cm}\text{with}\hspace{1cm} \psi_{u}(v) = u^T L(v) u.
$$
In addition, we also have $\langle \Phi^\prime(t), u \rangle = u^T \nabla^2 H_{\varepsilon}(v_t)(v - v^\ast)$. Since
$$
u^T L(v) u  =  \E \bigl[ \langle u , \pi(X,v) \rangle^2  \bigr]
$$
and
$$
\nabla_{v} \pi(x,v) = \frac{1}{\varepsilon} \Bigl( \diag(\pi(x,v)) - \pi(x,v) \pi(x,v)^T  \Bigr),
$$
one obtains that
$$
 \nabla_{v} \psi_{u}(v_t) = \frac{2}{\varepsilon}  \E \bigl[ \langle u , \pi(X,v_t) \rangle \left( \diag(\pi(X,v_t)) u - \langle u , \pi(X,v_t) \rangle \pi(X,v_t)  \right) \bigr].
$$
Consequently, 
\begin{align*}
\phi^\prime_u(t) & =  \frac{2}{\varepsilon}  \E \bigl[ \langle u , \pi(X,v_t) \rangle \left( u^T \diag(\pi(X,v_t)) (v-v^{\ast})  - \langle u , \pi(X,v_t) \rangle \langle \pi(X,v_t), v-v^{\ast} \rangle  \right) \bigr]  \\
&  - 2  \langle u, \nu \rangle  u^T \nabla^2 H_{\varepsilon}(v_t)(v - v^\ast).
\end{align*}
Then, we deduce from equality \eqref{Hessian.H} that
$$
\phi^\prime_u(t) = \frac{2}{\varepsilon}   \E \bigl[ \langle u , \pi(X,v_t) - \nu \rangle  u^T A_{\varepsilon}(X,v_t) (v-v^{\ast})  \bigr],
$$
with
$
A_{\varepsilon}(x,v) = \diag(\pi(x,v)) - \pi(x,v) \pi(x,v)^T.
$
It follows from Cauchy-Schwarz inequality and the upper bound \eqref{eq:boundgradient} that
$$
|\langle u , \pi(X,v_t) - \nu \rangle| \leq 2 \| u \|,
$$
which together with the fact that $\lambda_{\max}(A_{\varepsilon}(x,v) ) \leq 1$ yields
$$
|\phi^\prime_u(t)| \leq \frac{4}{\varepsilon} \| v-v^{\ast} \| \| u \|^2.
$$
Hence, inserting the above upper bound in \eqref{eq:diffPhi}, we obtain that
$$
\left| u^T (L(v) - L(v^\ast)) u - 2 \langle u, \nu \rangle \langle  \nabla \He(v), u \rangle \right| \leq  \frac{4}{\varepsilon} \| v-v^{\ast} \| \| u \|^2.
$$
Therefore, in the sense of partial ordering between positive semi-definite matrices, we have shown that
$$
- \frac{4}{\varepsilon} \| v-v^{\ast} \| \Id \leq  L(v) - L(v^\ast) -  \nabla \He(v) \nu^T -  \nu \nabla \He(v)^T \leq \frac{4}{\varepsilon} \| v-v^{\ast} \| \Id.
$$
Inequality \eqref{ineq:Gstar} thus follows from the decomposition \eqref{eq:decompG} since
$$
G_{\varepsilon}(v) - G_{\varepsilon}(v^\ast)  = L(v) - L(v^\ast) - \nu  \nabla \He(v)^T -  \nabla \He(v) \nu^T,
$$
which completes the proof of Lemma \ref{lem:LipGv}.

\subsection{A recursive formula to compute the inverse of $S_n$  for the stochastic Newton algorithm} \label{sec:recSN}

In this section, we discuss the construction of a recursive formula to compute, from the knowledge of $S_{n-1}^{-1}$, the inverse of the matrix $S_{n}$ defined by the recursive equation \eqref{eq:SnSN}  that corresponds to the use of the stochastic Newton (SN)  algorithm. To this end, let us first recall the following matrix inversion lemma classically referred to as the Sherman-Morrison-Woodbury (SMW) formula \cite{Hager1989},  also known as Woodbury's formula or Riccati's matrix identity.
\begin{lem}[SMW formula]\label{lem:SMW}
Suppose that $A$ and $C$ are invertible matrices of size $d \times d$ and $q \times q$ respectively. Let $U$ and $V$ be $d \times q$ and $q \times d$ matrices. Then, $A + UCV$ is invertible iff $C^{-1} + VA^{-1}U$ is invertible. In that case, we have
    \begin{equation}
        (A + UCV)^{-1} = A^{-1} - A^{-1}U (C^{-1} + VA^{-1}U)^{-1} VA^{-1}. \label{eq:SMW}
    \end{equation}
\end{lem}

\noindent
A repeated use of the SMW formula allows to prove the following result.

\begin{prop} \label{prop:rec_inv_Sn}
Let $\proj =  \Id - \bv_J \bv_J^T$ be the projection matrix onto $\langle \bv_J \rangle^\perp$.
Suppose that
$$
\mathbb{S}_{n} = \Id + \sum\limits_{k=1}^{n} \nabla_v^2 \he(X_{k}, \wh{V}_{k-1})   =  \mathbb{H}_n +  \bv_J \bv_J^T,
$$
where
$$
\mathbb{H}_n = \proj + \sum\limits_{k=1}^{n} \nabla_v^2 \he(X_{k}, \wh{V}_{k-1}),
$$
with $\mathbb{H}_0 = \proj$. Then, for all $n \geq 1$, one has $\mathbb{S}_{n}^{-1} =  \mathbb{H}_n^{-} +  \bv_J \bv_J^T$ with $\mathbb{H}_n^{-}$ that satisfies the recursive formula
\begin{align}
\mathbb{H}_n^{-} & =  \proj\Bigl(\mathbb{H}_{n-1}+ \frac{1}{\varepsilon}\diag\left(  \pi_n  \Bigr)\right)^{-1} \proj,  \label{eq:rec_inv_Sn}  \\
& =  \mathbb{H}_{n-1}^{-} - \mathbb{H}_{n-1}^{-} \left(\mathbb{H}_{n-1}^{-}  +  \varepsilon \diag\left( \pi_n^{-1} \right)\right)^{-1} \mathbb{H}_{n-1}^{-}, \label{eq:rec_inv_Sn2} 
\end{align}
where $\pi_n^{-1}$ stands for the vector whose entries are the inverse of those of $\pi_n = \pi(X_{n},\wh{V}_{n-1})$.
\end{prop}

\begin{proof}
For $n \geq 1$, we define
\begin{equation}
\Tilde{S}_n = \Id + \sum\limits_{k=1}^n \left(\nabla_v^2 \he(X_{k}, \wh{V}_{k-1}) +  \bv_J \bv_J^T \right)  = \Tilde{S}_{n-1} + \Sigma_n  = \mathbb{H}_n + (n+1)  \bv_J \bv_J^T  ,   \label{eq:tildeSn}
 \end{equation}
 where $\Sigma_n = \nabla_v^2 \he(X_{n}, \wh{V}_{n-1})  +  \bv_J \bv_J^T$ and $\Tilde{S}_0 = \Id = \proj + \bv_J \bv_J^T$. As discussed in Section \ref{sec:useful}, the  eigenvectors of  the matrix $\nabla_v^2 \he(X_{k}, \wh{V}_{k-1})$  associated to non-zero eigenvalues belong to $\langle \bv_J \rangle^\perp$ for any $k \geq 1$, which implies that $\mathbb{H}_n$ is a matrix of rank $J-1$ and that $\bv_J$ is its only eigenvector associated to the eigenvalue $\lambda_1 = 0$.  Hence, for all $n \geq 0$, $\mathbb{H}_n$ is also a matrix such that all its eigenvectors associated to non-zero eigenvalues belong to $\langle \bv_J \rangle^\perp$. Therefore, for any $n \geq 0$,  the inverse of  the matrix $\Tilde{S}_n$ (which is of full rank $J$) satisfies the  identity
\begin{equation}
 \Tilde{S}_n^{-1} = \mathbb{H}_n^{-} + \frac{1}{n +1} \bv_J \bv_J^T. \label{eq:tildeSn_inv}
 \end{equation}
 Moreover, given that $\mathbb{S}_n =\mathbb{H}_n + \bv_J \bv_J^T$ one has that
 \begin{equation}
\mathbb{S}_n^{-1} = \mathbb{H}_n^{-} +  \bv_J \bv_J^T = \Tilde{S}_n^{-1} + \frac{n}{n +1} \bv_J \bv_J^T. \label{eq:Sn_inv}
 \end{equation}
The computation of $ \Tilde{S}_n^{-1}$ can be done recursively as follows. By applying the SMW formula \eqref{eq:SMW} with $U = V = \Id$ we obtain that
\begin{equation}
  \Tilde{S}_n^{-1} =  \Tilde{S}_{n-1}^{-1} - \Tilde{S}_{n-1}^{-1}\left(\Sigma_n^{-1} + \Tilde{S}_{n-1}^{-1} \right)^{-1} \Tilde{S}_{n-1}^{-1}. \label{eq:Ric}
 \end{equation}
Now, introducing the notation $\pi_n = \pi(X_{n},\wh{V}_{n-1})$, we remark that 
$$
\nabla_v^2 \he(X_{n}, \wh{V}_{n-1})  =   \frac{1}{\varepsilon} \left( \diag(\pi_n) -\pi_n\pi_n^T  \right)
$$ is proportional to a multinomial matrix (up to a minus sign and the multiplicative factor $\varepsilon^{-1}$). Consequently, from the  pseudo inverse formula of multinomial matrices \cite{Steer05}, we obtain that the inverse of the matrix $\Sigma_n$ is given by
 \begin{equation}
 \Sigma_n^{-1} =  \varepsilon \proj \diag(\pi_n^{-1}) \proj + \bv_J \bv_J^T,   \label{sigmainv}
\end{equation}
 where $\pi_n^{-1}$ stands for the vector whose entries are the inverse of those of  $\pi_n$. Now, introducing the notation  $T_{n-1} = \Tilde{S}_{n-1}^{-1} + \bv_J \bv_J^T$ and $Q_{n-1} = \Sigma_n^{-1} + \Tilde{S}_{n-1}^{-1}$, we deduce from equation \eqref{sigmainv} that
$$
Q_{n-1}  = T_{n-1} + \varepsilon \proj \diag(\pi_n^{-1}) \proj.
$$
Consequently, by the SMW formula \eqref{eq:SMW}, it follows that
$$
Q_{n-1}^{-1} = T_{n-1}^{-1} - T_{n-1}^{-1}\proj \Bigl(\proj T_{n-1}^{-1}\proj + \frac{1}{\varepsilon}\diag(\pi_n)\Bigr)^{-1}\proj T_{n-1}^{-1}.
$$
Then, applying once again the SMW formula and equality \eqref{eq:tildeSn}, one has that
\begin{align*}
T_{n-1}^{-1} & =  \Tilde{S}_{n-1} - \frac{1}{\bv_J^T \Tilde{S}_{n-1} \bv_J   + 1 } \Tilde{S}_{n-1} \bv_J   \bv_J^T \Tilde{S}_{n-1} =   \Tilde{S}_{n-1} - \frac{n^2}{n +1}  \bv_J   \bv_J^T, \\
& =  \mathbb{H}_{n-1}  + \frac{n}{n +1}   \bv_J   \bv_J^T.
\end{align*}
Hence, by the fact that $\mathbb{H}_n$ maps the subspace $\langle \bv_J \rangle^\perp$ onto itself and since $\proj$ is the projection matrix onto $\langle \bv_J \rangle^\perp$, we thus obtain that
$$
Q_{n-1}^{-1} = T_{n-1}^{-1} - \mathbb{H}_{n-1} \Bigl(\mathbb{H}_{n-1} + \frac{1}{\varepsilon}\diag(\pi_n)\Bigr)^{-1} \mathbb{H}_{n-1}.
$$
Consequently, we have shown that
$$
 Q_{n-1}^{-1} = \Bigl(\Sigma_n^{-1} + \Tilde{S}_{n-1}^{-1}\Bigr)^{-1} = \mathbb{H}_{n-1}  + \frac{n}{n+1}   \bv_J   \bv_J^T - \mathbb{H}_{n-1} \Bigl(\mathbb{H}_{n-1} + \frac{1}{\varepsilon}\diag(\pi_n)\Bigr)^{-1} \mathbb{H}_{n-1}.
$$
Therefore, combining the above equality with \eqref{eq:tildeSn_inv}, one obtains that
\begin{align*}
\Tilde{S}_{n-1}^{-1}\Bigl(\Sigma_n^{-1} + \Tilde{S}_{n-1}^{-1} \Bigr)^{-1}\!\Tilde{S}_{n-1}^{-1}  & =  \mathbb{H}_{n-1}^{-}  + \frac{1}{n(n+1)}   \bv_J   \bv_J^T  \\
& - \mathbb{H}_{n-1}^{-}\mathbb{H}_{n-1} \!\Bigl(\mathbb{H}_{n-1} + \frac{1}{\varepsilon}\diag(\pi_n)\Bigr)^{-1} \!\mathbb{H}_{n-1}\mathbb{H}_{n-1}^{-}.
\end{align*}
Inserting the above equality into \eqref{eq:Ric} and using again \eqref{eq:tildeSn_inv}, one infers that
\begin{align*}
\hspace{-4ex}  \Tilde{S}_n^{-1} & =   \Tilde{S}_{n-1}^{-1} -  \mathbb{H}_{n-1}^{-}  - \frac{1}{n(n+1)}   \bv_J   \bv_J^T + \mathbb{H}_{n-1}^{-}\mathbb{H}_{n-1} \Bigl(\mathbb{H}_{n-1} + \frac{1}{\varepsilon}\diag(\pi_n)\Bigr)^{-1} \mathbb{H}_{n-1}\mathbb{H}_{n-1}^{-}, \\
  & =      \frac{1}{n+1}   \bv_J   \bv_J^T + \mathbb{H}_{n-1}^{-}\mathbb{H}_{n-1} \Bigl(\mathbb{H}_{n-1} + \frac{1}{\varepsilon}\diag(\pi_n)\Bigr)^{-1} \mathbb{H}_{n-1}\mathbb{H}_{n-1}^{-}, \\
    & =      \frac{1}{n+1}   \bv_J   \bv_J^T + \proj \Bigl(\mathbb{H}_{n-1} + \frac{1}{\varepsilon}\diag(\pi_n)\Bigr)^{-1} \proj,
\end{align*}
by noticing that  $\mathbb{H}_{n-1}^{-}\mathbb{H}_{n-1} = \mathbb{H}_{n-1}\mathbb{H}_{n-1}^{-}  = \proj$. Herafter, we immediately deduce
\eqref{eq:rec_inv_Sn} from the above identity together with \eqref{eq:tildeSn_inv}.
Finally, \eqref{eq:rec_inv_Sn2} follows from an application of a generalization of the SMW formula 
\cite{DENG20111561}  to the setting of the Moore-Penrose inverse to handle the situation where the matrix $A$ in equation \eqref{eq:SMW} is not invertible, which completes the proof of Proposition \ref{prop:rec_inv_Sn}.
\end{proof}

\section{ Proofs of auxiliary results related to the KL inequality} \label{Appendix-KL}

\begin{proof}[Proof of Proposition \ref{prop:KL}]

The proof consists in a study of $\widetilde{H}_{\varepsilon}(u)$ when a vector $u \in \langle \bv_J \rangle^\perp$ is either near $u^{\ast}$ or such that  $\|u\| \longrightarrow + \infty$.
First, we observe that $u \mapsto \|\nabla  \widetilde{H}_{\varepsilon}(u)\|^2+\frac{\|\nabla \widetilde{H}_{\varepsilon}(u)\|^2}{\widetilde{H}_{\varepsilon}(u)}$ is a continuous function except at $u^{\ast}$. Then, since $\widetilde{H}_{\varepsilon}(u^\ast) =  0$, a local approximation of $\widetilde{H}_{\varepsilon}$ using a Taylor expansion shows that, for all $h \in  \langle \bv_J \rangle^\perp$,
$$
\widetilde{H}_{\varepsilon}(u^{\ast}+h) =  \frac{1}{2} h^T   \nabla^2 \widetilde{H}_{\varepsilon}(u^{\ast})   h + o(\|h\|^2),
$$
and
$$
\nabla \widetilde{H}_{\varepsilon}(u^{\ast}+h) =   \nabla^2 \widetilde{H}_{\varepsilon}(u^{\ast})  h + o(\|h\|).
$$
with $\nabla^2 \widetilde{H}_{\varepsilon}(u^{\ast})  = G_{\ast}^{-1/2}\nabla^2 H_{\varepsilon}(v^{\ast}) G_{\ast}^{-1/2}$.
Since the matrices $G_{\ast}^{-1/2}$ and  $\nabla^2 H_{\varepsilon}(v^{\ast})$ are of rank $J-1$ with all eigenvectors corresponding to positive eigenvalues that belong to $\langle \bv_J \rangle^\perp$, the Courant-Fischer minmax Theorem yields:
\begin{align}
\tcr{0<}\lambda_{\min}(\nabla^2 \widetilde{H}_{\varepsilon}(u^{\ast})) & = \liminf_{u \longrightarrow u^{\ast}, u \in \langle \bv_J \rangle^\perp } \frac{\|\nabla \widetilde{H}_{\varepsilon}(u)\|^2}{\widetilde{H}_{\varepsilon}(u)}\nonumber\\
& \leq 
\limsup_{u \longrightarrow u^{\ast}, u \in \langle \bv_J \rangle^\perp} \frac{\|\nabla \widetilde{H}_{\varepsilon}(u)\|^2}{\widetilde{H}_{\varepsilon}(u)} = \lambda_{\max}(\nabla^2  \widetilde{H}_{\varepsilon}(u^{\ast})), \label{eq:Courant}
\end{align}
where $\lambda_{\min}(\nabla^2 \tilde{H}_{\varepsilon}(u^{\ast}))$ denotes the second smallest eigenvalue of the matrix $\nabla^2 \tilde{H}_{\varepsilon}(u^{\ast})$, and the notation $u \longrightarrow u^{\ast}$ corresponds to the convergence of $u \in \langle \bv_J \rangle^\perp$. 
 Inequality \eqref{eq:Courant} thus implies that the continuous function $u \mapsto \|\nabla  \widetilde{H}_{\varepsilon}(u)\|^2+\frac{\|\nabla \widetilde{H}_{\varepsilon}(u)\|^2}{\widetilde{H}_{\varepsilon}(u)}$ is upper  and lower bounded by positive constants in a neighborhood of $u^{\ast}$.
Finally, we observe that $\|\nabla  \widetilde{H}_{\varepsilon}\|$ is bounded thanks to inequality \eqref{eq:boundHepsgradient}, and it can be checked that the function $\widetilde{H}_{\varepsilon}$ is coercive  (over the finite dimensional vector space $\langle \bv_J \rangle^\perp$)   since it has a unique minimizer at $u^{\ast}$. These facts together with the boundedness of $u \mapsto \|\nabla  \widetilde{H}_{\varepsilon}(u)\|^2+\frac{\|\nabla \widetilde{H}_{\varepsilon}(u)\|^2}{\widetilde{H}_{\varepsilon}(u)}$  in a  neighborhood of $u^{\ast}$  of implies that inequality \eqref{eq:KL} holds.

Now, let us show that the constant $m$  appearing in inequality \eqref{eq:KL} can be made more explicit thanks to Lemma \ref{lem:LocalConvexity}. Indeed, since $\nabla \widetilde{H}_{\varepsilon}(u) = G_{\ast}^{-1/2} \nabla \He( G_{\ast}^{-1/2} u) $, we immediately obtain from inequality \eqref{eq:LocalConvexity} that
\begin{eqnarray} 
\langle   \nabla\widetilde{H}_{\varepsilon}(u) ,u-u^{\ast}  \rangle  & \geq  &
{\displaystyle  \frac{1 - \exp ( - \delta(u) ) }{\delta(u)}  (u - u^{\ast})^T  G_{\ast}^{-1/2} \nabla^2 H_{\varepsilon}(v^{\ast})  G_{\ast}^{-1/2} (u - u^{\ast})}, \nonumber \\
& \geq  &
{\displaystyle  \frac{1 - \exp ( - \delta(u) ) }{\delta(u)}  \| u - u^{\ast}\|^2 }, \label{eq:LocalConvexitytildeH}
\end{eqnarray}
where $\delta(u) =  \frac{\sqrt{2}}{\varepsilon} \lambda_{\max}( G_{\ast}^{-1/2} ) \|u - u^{\ast}\|$ and the second inequality above  follows from the fact that
$$
\tcb{\lambda_{\min}^{\langle \bv_J \rangle^\perp} \left( G_{\ast}^{-1/2} \nabla^2 H_{\varepsilon}(v^{\ast})  G_{\ast}^{-1/2} \right)} \geq 1,
$$
by inequality \eqref{ineq:eiglower}. Note that inequality \eqref{eq:LocalConvexitytildeH} corresponds to a local strong convex property of the function $\widetilde{H}_{\varepsilon}$ in the neighborhood of $u^{\ast}$. Then, by the Cauchy-Schwarz inequality, one has that
$$
\langle    \nabla \widetilde{H}_{\varepsilon}(u) ,u- u^{\ast}   \rangle \leq \|  \nabla \widetilde{H}_{\varepsilon}(u) \|   \|u -  u^{\ast}\|.
$$
Thus, we obtain by  inequality \eqref{eq:LocalConvexitytildeH}  that, for any $u \in  \langle \bv_J \rangle^\perp$,
\begin{equation} \label{ineq:LocalConvexity}
\|  \nabla \widetilde{H}_{\varepsilon}(u) \|  \geq 
 \frac{1 - \exp ( - \delta(u) ) }{\delta(u)}   \|u - u^{\ast}\|.
\end{equation}
\tcr{We then consider two cases. 
If $ \|u -  u^{\ast}\| \leq \frac{\varepsilon}{\lambda_{\max}( G_{\ast}^{-1/2} ) }$, then, using that the function $\delta \mapsto \frac{1}{\delta} \left( 1-\exp(-\delta)\right)$ is decreasing, it follows from \eqref{TAYLOR1tildeH} and \eqref{ineq:LocalConvexity} that,}
\begin{equation} \label{eq:m1}
 \|\nabla  \widetilde{H}_{\varepsilon}(u)\|^2+\frac{\|\nabla \widetilde{H}_{\varepsilon}(u)\|^2}{\widetilde{H}_{\varepsilon}(u)} \geq \frac{\|\nabla \widetilde{H}_{\varepsilon}(u)\|^2}{\widetilde{H}_{\varepsilon}(u)} \geq   \frac{2  \varepsilon}{  \lambda_{\max}(G_{\ast}^{-}) }        \left(1 - \exp \Bigl( -\sqrt{2}  \Bigr) \right)^2.
\end{equation} 
\tcr{ To the contrary, if $ \|u -  u^{\ast}\| \geq \frac{\varepsilon}{\lambda_{\max}( G_{\ast}^{-1/2} ) }$, then}
\begin{equation} \label{eq:m2}
 \|\nabla  \widetilde{H}_{\varepsilon}(u)\|^2+\frac{\|\nabla \widetilde{H}_{\varepsilon}(u)\|^2}{\widetilde{H}_{\varepsilon}(u)} \geq \|\nabla  \widetilde{H}_{\varepsilon}(u)\|^2 \geq \frac{\varepsilon^2}{2 \lambda_{\max}( G_{\ast}^{-} ) }    \left(1 - \exp \Bigl( -\sqrt{2}  \Bigr) \right)^2.
\end{equation}
using the inequality $1 - \exp \Bigl( -  \delta(u)  \Bigr) \geq 1 - \exp \Bigl( -  \sqrt{2}\Bigr)$ that holds for $\delta(u) \geq  \sqrt{2}$.
Consequently, since $ \left(1 - \exp \Bigl( -\sqrt{2} \Bigr) \right)^2 > 1/2$ and combining inequalities \eqref{eq:m1} and  \eqref{eq:m2}, it follows that the constant $m$ appearing in  inequality \eqref{eq:KL} can be chosen as $m  = m_{\varepsilon}$ with  $m_{\varepsilon}$ defined by \eqref{eq:choicem}.
This concludes the proof of Proposition \ref{prop:KL}.
\end{proof}

We then show the proof of the one-step evolution of the SGN algorithm.
\begin{proof}[Proof of Proposition \ref{prop:Phip} ]
First, by using the arguments from the proof of $(i)$ of Theorem \ref{theo:asconvVn}, we remark that  the matrix $S_n$ defined by   \eqref{eq:SnSGN} satisfies:
$$
\lambda_{\min}( S_n) \geq 1+ \gamma \min(\nu)  \left(\sum\limits_{m=1}^{p_n}  m^{-\beta} \right)   \quad \mbox{and} \quad \lambda_{\max}( S_n) \leq  1 + 4n + \gamma \max(\nu)  \sum_{m=1}^{p_n+1}   m^{-\beta}  ,
$$
where $p_n \geq 1$ denotes the largest integer such that $p_n J \leq n$. Consequently, using the fact that $ \frac{1}{1-\beta} \bigl(p_{n}^{1-\beta}-1\bigr) \leq \sum\limits_{m=1}^{p_n}  m^{-\beta}  \leq \frac{1}{1-\beta}    p_{n}^{1-\beta}$ the above inequalities imply that:
\begin{eqnarray}
\lambda_{\max}( S_n^{-1}) & \leq & \frac{1-\beta}{1-\beta + \gamma \min(\nu) \left(\left(n_J-1\right)^{1-\beta} - 1\right)}  \nonumber \\
& \leq & \frac{1}{1-2 \gamma \min(\nu) + \gamma \min(\nu)(n_J-1)^{1-\beta}},\label{eq:boundlambda_SnSGN_max}
\end{eqnarray}
where $n_J = n/J$, and
\begin{equation}
 \lambda_{\min}(S_n^{-1})  \geq \frac{1}{1-\beta+ 4(1-\beta)n + \gamma \max(\nu) \left(n/J\right)^{1-\beta}} \geq \frac{1}{1+ (4+2  \gamma \max(\nu))n}.\label{eq:boundlambda_SnSGN_min}
\end{equation}

\noindent
\underline{Step 1: Taylor expansion.}
First, we introduce the notation $\hUn = G_{\ast}^{1/2} \hVn$, and in the proof we repeatedly the property that the eigenvectors of $G_{\ast}$ associated to non-zero eigenvalues belong to $\langle \bv_J \rangle^\perp$. 
Using equation \eqref{eq:SNgen} \tcr{and the fact that $\proj G_{\ast} = G_{\ast}$},  a second order Taylor expansion yields
\begin{align}
\Phi(\hUnp) &= \Phi\left(\hUn -   n^{\alpha}G_{\ast}^{1/2} S_n^{-1}    \nabla_v \he (X_{n+1}, \hVn)  \right)\nonumber\\
& = \Phi(\hUn) - n^{\alpha} \left\langle \nabla \Phi(\hUn),  G_{\ast}^{1/2}  S_n^{-1} 
 \nabla_v \he (X_{n+1}, \hVn)   \right\rangle \nonumber\\
& + \frac{n^{2\alpha}}{2} \nabla^2 \Phi(\xi_{n+1}) \left( G_{\ast}^{1/2}  S_n^{-1}    \nabla_v \he (X_{n+1}, \hVn)  \right)^{\otimes 2}\label{eq:recurrence},
\end{align}
where $\xi_{n+1}$ is 
such that $\xi_{n+1} = \hUn+t_{n\tcr{+1}}  (\hUnp-\hUn)$ with $t_{n\tcr{+1}} \in (0,1)$. Now, applying inequalities \eqref{eq:boundgradient} and \eqref{eq:outil-Phip3},  the second order term in equation \eqref{eq:recurrence} can be bounded as follows:
\begin{align*}
\|\nabla^2 \Phi(\xi_{n+1}) \left( G_{\ast}^{1/2} S_n^{-1}  \nabla_v \he (X_{n+1}, \hVn)  \right)^{\otimes 2}\| & \leq \delta_{\varepsilon}   \lambda_{\max}(G_{\ast}^{-}) (1 +  \Phi(\xi_{n+1})) \times \\ 
&   \left\|G_{\ast}^{1/2} S_n^{-1} \ \nabla_v \he (X_{n+1}, \hVn)  \right\|^2 \\
& \leq 4 \delta_{\varepsilon}   \lambda_{\max}(G_{\ast}^{-}) \lambda^2_{\max}(  G_{\ast}^{1/2} S_n^{-1})  (1 +  \Phi(\xi_{n+1})),
\end{align*}
which yields the inequality:
\begin{eqnarray}
\Phi(\hUnp) & \leq & \Phi(\hUn) - n^{\alpha} \left\langle \nabla \Phi(\hUn),  G_{\ast}^{1/2} S_n^{-1}  
 \nabla_v \he (X_{n+1}, \hVn)   \right\rangle \nonumber \\
 & & +  2 \delta_{\varepsilon} \frac{ \lambda_{\max}(G_{\ast}) }{  \lambda_{\min}(G_{\ast}) } n^{2\alpha} \lambda^2_{\max}(  S_n^{-1})  (1 +  \Phi(\xi_{n+1})). \label{eq:recurrence_bis}
\end{eqnarray}

\noindent
\underline{Step 2: An auxiliary inequality.}
We now establish a technical bound to relate $\Phi(\xi_{n+1})$ and $\Phi(\hUn)$. To this end, by a first order Taylor expansion of the function $s \mapsto \widetilde{H}_\varepsilon(\hUn+s  (\hUnp-\hUn))$ where $\hUn^s = \hUn+s  (\hUnp-\hUn) $, one has that:
$$
\widetilde{H}_\varepsilon(\hUnp) =  \widetilde{H}_\varepsilon(\hUn) +  \int_0^1 \nabla \widetilde{H}_{\varepsilon}(\hUn^s)(\hUnp-\hUn)\,ds,
$$
and thus, combining the Cauchy-Schwarz inequality with the upper bounds  \eqref{eq:boundgradient} and  \eqref{eq:boundgradient:Htilde}, we obtain that:
\begin{eqnarray}
\widetilde{H}_{\varepsilon}(\xi_n) & \leq & \widetilde{H}_{\varepsilon}(\hUn) + \sup_{t \in [0,1] }\|\nabla \widetilde{H}_{\varepsilon}(\hUn^T)\| \|\hUnp-\hUn\| \nonumber \\
& \leq & \widetilde{H}_{\varepsilon}(\hUn) + \sup_{t \in [0,1] }\|\nabla \widetilde{H}_{\varepsilon}(\hUn^T)\| \|  n^{\alpha}G_{\ast}^{1/2} S_n^{-1}    \nabla_v \he (X_{n+1}, \hVn) \|  \nonumber \\
& \leq & \widetilde{H}_{\varepsilon}(\hUn) + 4 \lambda_{\max}(G_{\ast}^{-1/2}) \lambda_{\max}(G_{\ast}^{1/2} ) n^{\alpha} \lambda_{\max}(S_n^{-1}). \label{eq:boundxi}
\end{eqnarray}
Note that, under the condition $\alpha + \beta < 1/2$, it follows from inequality \eqref{eq:boundlambda_SnSGN_max} that $n^{\alpha}\lambda_{\max}(S_n^{-1}) \leq c_0$ for some constant $c_0 \geq 1$ for all $n \geq J$. Hence, inserting the upper bound \eqref{eq:boundxi} into the definition of $\Phi$ and using inequality \eqref{eq:Phi0}, we obtain that
\begin{align}
\Phi(\xi_{n+1}) &\leq \left(\widetilde{H}_{\varepsilon}(\hUn) +4   \frac{\lambda_{\max}(G_{\ast}^{1/2} )}{\lambda_{\min}(G_{\ast}^{1/2})  } n^{\alpha}  \lambda_{\max}(S_n^{-1})\right) \times \nonumber \\
&  \exp\left(\widetilde{H}_{\varepsilon}(\hUn) +4   \frac{\lambda_{\max}(G_{\ast}^{1/2} )}{\lambda_{\min}(G_{\ast}^{1/2})  } n^{\alpha} \lambda_{\max}(S_n^{-1})\right) \nonumber\\
& \leq  \exp\left(4   \frac{\lambda_{\max}(G_{\ast}^{1/2} )}{\lambda_{\min}(G_{\ast}^{1/2})  } n^{\alpha} \lambda_{\max}(S_n^{-1})\right)  \times \nonumber \\
&  \left(\widetilde{H}_{\varepsilon}(\hUn) + 4   \frac{\lambda_{\max}(G_{\ast}^{1/2} )}{\lambda_{\min}(G_{\ast}^{1/2})  } n^{\alpha} \lambda_{\max}(S_n^{-1})\right) \exp\left(\widetilde{H}_{\varepsilon}(\hUn)\right) \nonumber\\
& \leq  \exp\left(4 c_0  \frac{\lambda_{\max}(G_{\ast}^{1/2} )}{\lambda_{\min}(G_{\ast}^{1/2})  } \right) \ \left(\Phi(\hUn)+ 4 c_0  \frac{\lambda_{\max}(G_{\ast}^{1/2} )}{\lambda_{\min}(G_{\ast}^{1/2})  }  (1 + \Phi(\hUn)) \right) \nonumber \\
&  \leq \tilde{c}_{0} \left(1+ \Phi(\hUn)  \right), \label{eq:borne_Phi_p}
\end{align}
where $ \tilde{c}_{0} = \max\left(1,4  c_0  \frac{\lambda_{\max}(G_{\ast}^{1/2} )}{\lambda_{\min}(G_{\ast}^{1/2})  }\right)   \exp\left( 4  c_0  \frac{\lambda_{\max}(G_{\ast}^{1/2} )}{\lambda_{\min}(G_{\ast}^{1/2})  }  \right)$. \\

\noindent \underline{Step 3:  Derivation of a recursive inequality.} Inserting inequality \eqref{eq:borne_Phi_p} into  \eqref{eq:recurrence_bis}, and taking the conditional expectation with respect to $\mathcal{F}_n$,   we obtain that:
\begin{align}
\mathbb{E} \left[\Phi(\hUnp) \, \vert \mathcal{F}_n\right] &\leq \Phi(\hUn) 
-  n^{\alpha} \left\langle \nabla \Phi(\hUn),  G_{\ast}^{1/2} S_n^{-1} G_{\ast}^{1/2} 
\nabla \widetilde{H}_{\varepsilon}(\hUn) \right\rangle \nonumber \\
 &+ 2 \delta_{\varepsilon} \frac{ \lambda_{\max}(G_{\ast}) }{  \lambda_{\min}(G_{\ast}) } n^{2\alpha}  \lambda^2_{\max}(  S_n^{-1})  \left(1 +    \tilde{c}_{0}\ \left(1+ \Phi(\hUn)  \right) \right),  \nonumber 
\end{align}
where we used the property that $\nabla H_{\varepsilon}(\hVn )= G_{\ast}^{1/2} \nabla \widetilde{H}_{\varepsilon}(\hUn)$.

Consequently,  using  inequality \eqref{eq:outil-Phip2} and introducing $c_{\varepsilon}= 2 \tilde{c}_{0} \delta_{\varepsilon} \frac{ \lambda_{\max}(G_{\ast}) }{  \lambda_{\min}(G_{\ast}) }$, we have:
\begin{eqnarray}
\mathbb{E} \left[\Phi(\hUnp) \, \vert \mathcal{F}_n\right]& \leq& \left(1 - m_{\varepsilon}  n^{\alpha}  \lambda_{\min}(   G_{\ast} ) \lambda_{\min}(S_n^{-1}  ) + c_{\varepsilon} n^{2\alpha} \lambda^2_{\max}(  S_n^{-1})   \right)  \Phi(\hUn)  \nonumber \\ 
& & +   c_{\varepsilon} n^{2\alpha} \lambda^2_{\max}(  S_n^{-1}) \label{eq:Phip}
\end{eqnarray}
Now, thanks to inequalities  \eqref{eq:boundlambda_SnSGN_max}  and \eqref{eq:boundlambda_SnSGN_min}, and the condition $0 < \alpha + \beta <1/2$, it follows that there exists an integer $n_0$ such that, for all $n \geq n_0$, 
\begin{equation}
 m_{\varepsilon}  n^{\alpha}  \lambda_{\min}(   G_{\ast} )  \lambda_{\min}(  S_n^{-1}) \ge 2 c_{\varepsilon}  n^{2\alpha}  \lambda^2_{\max}(  S_n^{-1}) . \label{eq:Phip2}
\end{equation}
Hence, by combining inequalities \eqref{eq:Phip} and \eqref{eq:Phip2}  we obtain inequality \eqref{eq:Phipi} which concludes the proof of $i)$.
\end{proof}


\section*{Acknowledgments} The authors gratefully acknowledge financial support from the Agence Nationale de la Recherche  (MaSDOL grant ANR-19-CE23-0017).  J. Bigot and S. Gadat are members of the Institut Universitaire de France (IUF), and part of this work has been carried out with financial support from the IUF.

\bibliographystyle{acm}
\bibliography{StoGN_OTreg_revision.bib}

\begin{thebibliography}{10}

\bibitem{NIPS2017_6792}
{\sc Altschuler, J., Weed, J., and Rigollet, P.}
\newblock Near-linear time approximation algorithms for optimal transport via
  sinkhorn iteration.
\newblock In {\em Advances in Neural Information Processing Systems 30}. 2017,
  pp.~1964--1974.

\bibitem{Bach14}
{\sc Bach, F.~R.}
\newblock Adaptivity of averaged stochastic gradient descent to local strong
  convexity for logistic regression.
\newblock {\em Journal of Machine Learning Research 15}, 1 (2014), 595--627.

\bibitem{benaim1996asymptotic}
{\sc Bena{\"\i}m, M., and Hirsch, M.}
\newblock Asymptotic pseudotrajectories and chain recurrent flows, with
  applications.
\newblock {\em Journal of Dynamics and Differential Equations 8}, 1 (1996),
  141--176.

\bibitem{benamou2015iterative}
{\sc Benamou, J.-D., Carlier, G., Cuturi, M., Nenna, L., and Peyr{\'e}, G.}
\newblock Iterative {B}regman projections for regularized transportation
  problems.
\newblock {\em SIAM Journal on Scientific Computing 37}, 2 (2015),
  A1111--A1138.

\bibitem{Stochastic_Bigot_Bercu}
{\sc Bercu, B., and Bigot, J.}
\newblock Asymptotic distribution and convergence rates of stochastic
  algorithms for entropic optimal transportation between probability measures.
\newblock {\em Annals of Statistics 49}, 2 (2021), 968--987.

\bibitem{Bercu_Godichon_Portier2020}
{\sc Bercu, B., Godichon, A., and Portier, B.}
\newblock An efficient stochastic newton algorithm for parameter estimation in
  logistic regressions.
\newblock {\em SIAM Journal on Control and Optimization 58}, 1 (2020),
  348--367.

\bibitem{bigot:hal-01790015}
{\sc Bigot, J., Cazelles, E., and Papadakis, N.}
\newblock {Data-driven regularization of Wasserstein barycenters with an
  application to multivariate density registration}.
\newblock {\em {Information and Inference} 8}, 4 (2019), 719--755.

\bibitem{bigot2017geodesic}
{\sc Bigot, J., Gouet, R., Klein, T., L{\'o}pez, A., et~al.}
\newblock Geodesic {PCA} in the {W}asserstein space by convex {PCA}.
\newblock {\em Annales de l'Institut Henri Poincar{\'e}, Probabilit{\'e}s et
  Statistiques 53}, 1 (2017), 1--26.

\bibitem{bigotReview}
{\sc {Bigot, J\'er\'emie}}.
\newblock Statistical data analysis in the wasserstein space.
\newblock {\em ESAIM: ProcS 68\/} (2020), 1--19.

\bibitem{Bolte3}
{\sc Bolte, J., Daniilidis, A., and Lewis, A.}
\newblock The {\l} ojasiewicz inequality for nonsmooth subanalytic functions
  with applications to subgradient dynamical systems.
\newblock {\em SIAM J. Optim. 17}, 4 (2006), 1205--1223.

\bibitem{Bolte2}
{\sc Bolte, J., Nguyen, P., Peypouquet, J., and Suter, B.~W.}
\newblock From error bounds to the complexity of first-order descent methods
  for convex functions.
\newblock {\em Math. Program. (A)}, 165 (2017), 471--507.

\bibitem{2015-bonneel-siims}
{\sc Bonneel, N., Rabin, J., Peyr{\'e}, G., and Pfister, H.}
\newblock Sliced and radon {W}asserstein barycenters of measures.
\newblock {\em Journal of Mathematical Imaging and Vision 51}, 1 (2015),
  22--45.

\bibitem{cazelles:hal-01581699}
{\sc Cazelles, E., Seguy, V., Bigot, J., Cuturi, M., and Papadakis, N.}
\newblock {Log-PCA versus Geodesic PCA of histograms in the Wasserstein space}.
\newblock {\em {SIAM Journal on Scientific Computing} 40}, 2 (2018),
  B429--B456.

\bibitem{Godichon2020}
{\sc C{\'e}nac, P., Godichon-Baggioni, A., and Portier, B.}
\newblock An efficient averaged stochastic gauss-newton algorithm for
  estimating parameters of non linear regressions models.
\newblock Preprint - arXiv2006.12920, 2020.

\bibitem{cuturi}
{\sc Cuturi, M.}
\newblock Sinkhorn distances: Lightspeed computation of optimal transport.
\newblock In {\em Advances in Neural Information Processing Systems 26},
  C.~J.~C. Burges, L.~Bottou, M.~Welling, Z.~Ghahramani, and K.~Q. Weinberger,
  Eds. Curran Associates, Inc., 2013, pp.~2292--2300.

\bibitem{cutpey2018}
{\sc Cuturi, M., and Peyr{È}, G.}
\newblock Semidual regularized optimal transport.
\newblock {\em SIAM Review 60}, 4 (2018), 941--965.

\bibitem{DENG20111561}
{\sc Deng, C.~Y.}
\newblock A generalization of the {S}herman-{M}orrison-{W}oodbury formula.
\newblock {\em Applied Mathematics Letters 24}, 9 (2011), 1561 -- 1564.

\bibitem{Duflo1997}
{\sc Duflo, M.}
\newblock {\em Random iterative models}, vol.~34 of {\em Applications of
  Mathematics, New York}.
\newblock Springer-Verlag, Berlin, 1997.

\bibitem{ferradans2014regularized}
{\sc Ferradans, S., Papadakis, N., Peyr{\'e}, G., and Aujol, J.-F.}
\newblock Regularized discrete optimal transport.
\newblock {\em SIAM Journal on Imaging Sciences 7}, 3 (2014), 1853--1882.

\bibitem{Flamary2018}
{\sc Flamary, R., Cuturi, M., Courty, N., and Rakotomamonjy, A.}
\newblock Wasserstein discriminant analysis.
\newblock {\em Machine Learning 107}, 12 (Dec 2018), 1923--1945.

\bibitem{Frogner:2015}
{\sc Frogner, C., Zhang, C., Mobahi, H., Araya-Polo, M., and Poggio, T.}
\newblock Learning with a wasserstein loss.
\newblock In {\em Proceedings of the 28th International Conference on Neural
  Information Processing Systems - Volume 2\/} (Cambridge, MA, USA, 2015),
  NIPS'15, MIT Press, pp.~2053--2061.

\bibitem{gadat:hal-01623986}
{\sc Gadat, S., and Panloup, F.}
\newblock {Optimal non-asymptotic bound of the Ruppert-Polyak averaging without
  strong convexity}.
\newblock Preprint - arXiv:1709.03342, May 2022.

\bibitem{NIPS2016_6566}
{\sc Genevay, A., Cuturi, M., Peyr\'{e}, G., and Bach, F.}
\newblock Stochastic optimization for large-scale optimal transport.
\newblock In {\em Advances in Neural Information Processing Systems 29}. 2016,
  pp.~3440--3448.

\bibitem{pmlr-v84-genevay18a}
{\sc Genevay, A., Peyre, G., and Cuturi, M.}
\newblock Learning generative models with sinkhorn divergences.
\newblock In {\em Proceedings of the Twenty-First International Conference on
  Artificial Intelligence and Statistics\/} (Playa Blanca, Lanzarote, Canary
  Islands, 09--11 Apr 2018), A.~Storkey and F.~Perez-Cruz, Eds., vol.~84 of
  {\em Proceedings of Machine Learning Research}, PMLR, pp.~1608--1617.

\bibitem{pmlr-v97-gordaliza19a}
{\sc Gordaliza, P., Del~Barrio, E., Gamboa, F., and Loubes, J.-M.}
\newblock Obtaining fairness using optimal transport theory.
\newblock In {\em Proceedings of the 36th International Conference on Machine
  Learning\/} (2019), pp.~2357--2365.

\bibitem{gramfort2015fast}
{\sc Gramfort, A., Peyr{\'e}, G., and Cuturi, M.}
\newblock Fast optimal transport averaging of neuroimaging data.
\newblock In {\em International Conference on Information Processing in Medical
  Imaging\/} (2015), Springer, pp.~261--272.

\bibitem{Hager1989}
{\sc Hager, W.~W.}
\newblock Updating the inverse of a matrix.
\newblock {\em SIAM Review 31}, 2 (1989), 221--239.

\bibitem{KingmaB14}
{\sc Kingma, D.~P., and Ba, J.}
\newblock Adam: {A} method for stochastic optimization.
\newblock In {\em 3rd International Conference on Learning Representations,
  {ICLR} 2015, San Diego, CA, USA, May 7-9, 2015, Conference Track
  Proceedings\/} (2015), Y.~Bengio and Y.~LeCun, Eds.

\bibitem{merigot18}
{\sc Kitagawa, J., M{\'e}rigot, Q., and B., T.}
\newblock Convergence of a newton algorithm for semi-discrete optimal
  transport.
\newblock {\em Journal of the European Math Society 21}, 9 (2019), 2603--2651.

\bibitem{KlattTM20}
{\sc Klatt, M., Tameling, C., and Munk, A.}
\newblock Empirical regularized optimal transport: Statistical theory and
  applications.
\newblock {\em {SIAM} J. Math. Data Sci. 2}, 2 (2020), 419--443.

\bibitem{Kurdyka}
{\sc Kurdyka, K.}
\newblock On gradients of functions definable in o-minimal structures.
\newblock {\em Ann. Inst. Fourier (Grenoble) 48}, 3 (1998), 769--783.

\bibitem{Lojasiewicz}
{\sc Lojasiewicz, S.}
\newblock Une propri\'et\'e topologique des sous-ensembles analytiques r\'eels.
\newblock {\em Editions du centre National de la Recherche Scientifique, Paris,
  Les \'Equations aux D\'eriv\'ees Partielles\/} (1963), 87--89.

\bibitem{Merigot11}
{\sc M\'erigot, Q.}
\newblock A multiscale approach to optimal transport.
\newblock {\em Computer Graphics Forum 30}, 5 (2011), 1583--1592.

\bibitem{merigot18b}
{\sc M{\'e}rigot, Q., Meyron, J., and Thibert, B.}
\newblock {An algorithm for optimal transport between a simplex soup and a
  point cloud}.
\newblock {\em {SIAM Journal on Imaging Sciences} 11}, 2 (2018), 1363--1389.

\bibitem{Pana15}
{\sc Panaretos, V.~M., and Zemel, Y.}
\newblock Amplitude and phase variation of point processes.
\newblock {\em Annals of Statistics 44}, 2 (2016), 771--812.

\bibitem{Pana18}
{\sc Panaretos, V.~M., and Zemel, Y.}
\newblock Statistical aspects of wasserstein distances.
\newblock {\em Annual Reviews of Statistics and its Applications 6\/} (2018),
  405--431.

\bibitem{Pelletier}
{\sc Pelletier, M.}
\newblock Asymptotic almost sure efficiency of averaged stochastic algorithms.
\newblock {\em SIAM J. Control and Optimization 39\/} (08 2000), 49--72.

\bibitem{bookOT}
{\sc Peyr{\'e}, G., and Cuturi, M.}
\newblock Computational optimal transport.
\newblock {\em Foundations and Trends in Machine Learning 11}, 5-6 (2019),
  355--607.

\bibitem{rabin2015convex}
{\sc Rabin, J., and Papadakis, N.}
\newblock Convex color image segmentation with optimal transport distances.
\newblock In {\em International Conference on Scale Space and Variational
  Methods in Computer Vision\/} (2015), Springer, pp.~256--269.

\bibitem{RIGOLLET20181228}
{\sc Rigollet, P., and Weed, J.}
\newblock Entropic optimal transport is maximum-likelihood deconvolution.
\newblock {\em Comptes Rendus Mathematique 356}, 11 (2018), 1228 -- 1235.

\bibitem{robbins1971convergence}
{\sc Robbins, H., and Siegmund, D.}
\newblock A convergence theorem for non negative almost supermartingales and
  some applications.
\newblock In {\em Optimizing methods in statistics}. Elsevier, 1971,
  pp.~233--257.

\bibitem{rolet2016fast}
{\sc Rolet, A., Cuturi, M., and Peyr{\'e}, G.}
\newblock Fast dictionary learning with a smoothed {W}asserstein loss.
\newblock In {\em Proc. International Conference on Artificial Intelligence and
  Statistics (AISTATS)\/} (2016).

\bibitem{NEURIPS2018_5a9d8bf5}
{\sc Sanjabi, M., Ba, J., Razaviyayn, M., and Lee, J.~D.}
\newblock On the convergence and robustness of training gans with regularized
  optimal transport.
\newblock In {\em Advances in Neural Information Processing Systems\/} (2018),
  S.~Bengio, H.~Wallach, H.~Larochelle, K.~Grauman, N.~Cesa-Bianchi, and
  R.~Garnett, Eds., vol.~31, Curran Associates, Inc.

\bibitem{NIPS2015_5680}
{\sc Seguy, V., and Cuturi, M.}
\newblock Principal geodesic analysis for probability measures under the
  optimal transport metric.
\newblock In {\em Advances in Neural Information Processing Systems 28}. 2015,
  pp.~3294--3302.

\bibitem{Solomon:2015}
{\sc Solomon, J., de~Goes, F., Peyr{\'e}, G., Cuturi, M., Butscher, A., Nguyen,
  A., Du, T., and Guibas, L.}
\newblock Convolutional wasserstein distances: Efficient optimal transportation
  on geometric domains.
\newblock {\em ACM Trans. Graph. 34}, 4 (2015), 66:1--66:11.

\bibitem{sommerfeld2016inference}
{\sc Sommerfeld, M., and Munk, A.}
\newblock Inference for empirical {W}asserstein distances on finite spaces.
\newblock {\em Journal of the Royal Statistical Society: Series B (Statistical
  Methodology)\/} (2016).

\bibitem{Steer05}
{\sc Steerneman, A., and {van Perlo -ten Kleij}, F.}
\newblock Properties of the matrix a-xy.
\newblock {\em Linear Algebra and Its Applications 410\/} (2005), 70--86.

\bibitem{villani2008optimal}
{\sc Villani, C.}
\newblock {\em Optimal transport: old and new}, vol.~338.
\newblock Springer Science \& Business Media, 2008.

\bibitem{Pana17}
{\sc Zemel, Y., and Panaretos, V.~M.}
\newblock FrÈchet means and procrustes analysis in wasserstein space.
\newblock {\em Bernoulli 25}, 2 (05 2019), 932--976.

\bibitem{Zhang2016}
{\sc Zhang, L.-X.}
\newblock Central limit theorems of a recursive stochastic algorithm with
  applications to adaptive designs.
\newblock {\em Ann. Appl. Probab. 26}, 6 (2016), 3630--3658.

\end{thebibliography}

\end{document}